\documentclass[10Ishpt,a4paper,twoside]{amsart}
\usepackage{amsmath, latexsym, amsfonts, amssymb, amsthm}
\usepackage{graphicx, color,hyperref,dsfont,epsfig,caption,wrapfig,subfig}
\usepackage{datetime}
\usepackage{movie15}

\hypersetup{
    linktoc=page,
    linkcolor=red,          
    citecolor=blue,        
    filecolor=blue,      
    urlcolor=cyan,
    colorlinks=true           
}

\numberwithin{equation}{section}

\newtheorem{theorem}{Theorem}[section]
\newtheorem{lemma}[theorem]{Lemma}
\newtheorem{proposition}[theorem]{Proposition}
\newtheorem{corollary}[theorem]{Corollary}

\newtheorem{remark}[theorem]{Remark}

\theoremstyle{definition}

\renewcommand{\tilde}{\widetilde}          
\DeclareMathSymbol{\leqslant}{\mathalpha}{AMSa}{"36} 
\DeclareMathSymbol{\geqslant}{\mathalpha}{AMSa}{"3E} 
\DeclareMathSymbol{\eset}{\mathalpha}{AMSb}{"3F}     
\renewcommand{\leq}{\;\leqslant\;}                   
\renewcommand{\geq}{\;\geqslant\;}                   

\newcommand{\C}{\mathbb{C}}

\newcommand{\D}{\mathbb{D}}
\newcommand{\R}{\mathbb{R}}
\newcommand{\Z}{\mathbb{Z}}
\newcommand{\N}{\mathbb{N}}

\newcommand{\E}{\mathds{E}}
\renewcommand{\P}{\mathds{P}}

\newcommand{\hf}{\frac{_1}{^2}}

\usepackage{xparse}

\def\bi{\begin{itemize}}
\def\ei{\end{itemize}}

\def\bnum{\begin{enumerate}}
\def\enum{\end{enumerate}}
\def\<#1{\langle #1 \rangle}

\def\z{\mathbf{z}}

\def\x{\mathbf{x}}

\newcommand{\caA}{{\mathcal A}}

\newcommand{\caF}{{\mathcal F}}

\newcommand{\caO}{{\mathcal O}}

\newcommand{\caT}{{\mathcal T}}

\begin{document}

\title[Local Conformal Structure of Liouville Quantum Gravity]{Local Conformal Structure of Liouville Quantum Gravity}

\author[Antti Kupiainen]{Antti Kupiainen$^{1}$}
\address{University of Helsinki, Department of Mathematics and Statistics,
         P.O. Box 68 , FIN-00014 University of Helsinki, Finland}
\email{antti.kupiainen@helsinki.fi}

\author[R\'emi Rhodes]{R\'emi Rhodes$^{2}$}
\address{Universit{\'e} Paris-Est Marne la Vall\'ee, LAMA, Champs sur Marne, France}
\email{remi.rhodes@wanadoo.fr}

\author[Vincent Vargas]{ Vincent Vargas$^{2}$}
\address{ENS Ulm, DMA, 45 rue d'Ulm,  75005 Paris, France}
\email{Vincent.Vargas@ens.fr}

\footnotetext[1]{Supported by the Academy of Finland,$^2$Research supported in part by ANR grant Liouville (ANR-15-CE40-0013)}



 \begin{abstract}
Liouville Conformal Field Theory (LCFT) is an essential building block of Polyakov's formulation of non critical string theory. Moreover, scaling limits of statistical mechanics models on planar maps are believed by physicists to be described by LCFT. A rigorous probabilistic formulation of LCFT based on a path integral formulation was recently given by the present authors and F. David in  \cite{DKRV}.  
In the present work, we  prove the validity of the conformal Ward identities and the Belavin-Polyakov-Zamolodchikov (BPZ) differential equations (of order $2$) for the correlation functions of LCFT. This initiates the program started in the seminal work of Belavin-Polyakov-Zamolodchikov \cite{BPZ} in a probabilistic setup for a non-trivial Conformal Field Theory.  We also prove several celebrated results on LCFT, in particular an explicit formula for the  4 point correlation functions (with insertion of a second order degenerate field)  leading to a rigorous  proof of a non trivial functional relation on the 3 point structure constants derived earlier in the physics literature by Teschner \cite{Tesc}. The proofs are based on exact identities which rely on the underlying Gaussian structure of LCFT  combined with estimates from the theory of critical Gaussian Multiplicative Chaos and a careful analysis of singular integrals (Beurling transforms and generalizations). As a by-product, we give bounds on the correlation functions when two points collide making rigorous certain predictions from physics on the so-called ``operator product expansion" of LCFT.  
 \end{abstract}

\maketitle
\tableofcontents

\section{Introduction}

Ever since the ground-breaking work of Belavin, Polyakov and Zamolodchikov (BPZ) in 1984 \cite{BPZ}, the precise mathematical structure of local conformal symmetry uncovered in that work has been a challenge to mathematicians. The Conformal Field Theories (CFT) studied in \cite{BPZ} are believed to be  limits of  probabilistic objects, namely 
scaling  limits of Gibbs measures of statistical mechanics models defined on two dimensional grids (or graphs). However, the full continuum formalism of CFT in the sense of BPZ has proven to be mathematically elusive except for a few cases among which
 the {\it Gaussian Free Field} (GFF)  \cite{KM} and (partially) the Ising model at critical temperature (see \cite{BD,BD1,CGS,CHI,Dub2,HS} for the latest developments on this). 

One of the most intriguing CFT's is the Liouville CFT (LCFT hereafter\footnote{In our previous works, we also used the terminology Liouville quantum field theory with associated abbreviation LQFT. Both terminologies (and theories) are completely equivalent since LQFT   is in fact a CFT; we have decided to the use the abbreviation LCFT in this article to stress that it is a CFT.}). It first appeared (in the context of String Theory) in Polyakov's Liouville quantum gravity theory \cite{Pol} of summation of random metrics  and then  in the 1988 work of Knizhnik, Polyakov and Zamolodchikov (KPZ) \cite{cf:KPZ} on the relations between CFT's on deterministic and random surfaces or in other terms on the relationship between statistical mechanics models on fixed grids and on random grids (random planar maps). An exact mathematical formulation of the so-called KPZ relation that appears in \cite{cf:KPZ} is given in \cite{DKRV2}. 

KPZ viewed random surfaces as a two dimensional manifold $\Sigma$ equipped with a random Riemannian metric $G$ whose law should be invariant under the group of diffeomorphisms ${\rm Diff}(\Sigma)$ of the surface. Guided by the fact that  the space of smooth metrics on $\Sigma$ is obtained as  ${\rm Diff}(\Sigma)$ orbits of  metrics of the form $e^{\sigma}\hat G$ where $\sigma$ is a real valued function on $\Sigma$ called conformal factor and $\hat G$ belongs to a finite dimensional moduli space of metrics they ended up  looking for a law for the random field  $\sigma$. They argued that the law of $\sigma$ is described by LCFT. In what follows, we will only consider the case of the Riemann sphere $\Sigma= \hat{\C}= \C \cup \lbrace \infty \rbrace$ (one could also consider other topologies like the disk where one must also take into account non trivial specific issues linked to the presence of a boundary).  In the case of the Riemann sphere,  the metric is written as $e^{\gamma\phi(z)}|dz|^2$ and the law of $\phi$ is
\begin{equation}\label{liouvlaw}
\nu(d\phi)=e^{-S_L(\phi)}\nu_0(d\phi)
 \end{equation}   
where $S_L$ is the {\it Liouville Action functional}
\begin{equation}\label{liouvlaw1}
S_L(\phi)=\frac{_1}{^{\pi}}\int_\C(|\partial_z\phi(z)|^2+\pi\mu e^{\gamma\phi(z)})d^2z
 \end{equation} 
 with $d^2z$ the standard Lebesgue measure and $\nu_0$ a putative "flat" measure on some space of maps $\phi:\C\to\R$ (in fact, there is an infinite constant hidden behind expression \eqref{liouvlaw1} that we omit in the introduction for the sake of clarity: see Section \ref{sec:backgr} for the precise formulation). One can notice that LCFT has two parameters $\gamma$ and $\mu$. According to KPZ \cite{cf:KPZ}, the parameter $\gamma$ is determined by  the particular CFT (e.g. model of statistical mechanics) which lives on the random surface. Such a CFT is characterized by its central charge $c_M$ (the $M$ in the notation stands for matter since in Liouville quantum gravity this CFT is called a matter field) and the relation between $\gamma$ and $c_M$ is $c_M=25-6Q^2$ where
 \begin{equation}\label{Qdef}
Q=\frac{2}{\gamma}+\frac{\gamma}{2}.  \footnote{The central charge $c_M$ of the CFT living on the random surface is not to be confused with the central charge $c_L$ of LCFT with parameter $\gamma$, which is also a CFT. These two CFTs are coupled independently with $c_L= 1+6 Q^2$: equivalently, one has the relation $c_M+c_L-26=0$ discovered by Polyakov in \cite{Pol}.}
 \end{equation} 
 For instance, uniform random planar maps correspond to a CFT with central charge $c_M=0$ and therefore to $\gamma=\sqrt{8/3}$, the Ising model on random planar maps corresponds to $c_M=\frac{1}{2}$ hence to $\gamma=\sqrt{3}$ and the GFF on random planar maps corresponds to $c_M=1$ hence to $\gamma=2$. The parameter $\mu>0$ is called the {\it cosmological constant} and it makes the law \eqref{liouvlaw1} non Gaussian. It turns out that various quantities in LCFT have a simple scaling behaviour in $\mu$ (called ``KPZ-scaling" in the physics literature) but we may not take $\mu$ to zero in LCFT.
 
 LCFT is supposed to be   as its name suggests also a Conformal Field Theory (CFT). Recall that this means in particular that there should exist {\it primary conformal fields} $V_\alpha(z)$ defined for $z \in \C$, i.e. random fields whose expectations in the Liouville law \eqref{liouvlaw} are conformal tensors. More precisely, if $z_1, \cdots, z_N$ are $N$ distinct points in $\C$  then for a M\"obius map $\psi(z)= \frac{az+b}{cz+d}$ (with $a,b,c,d \in \C$ and $ad-bc=1$) 
 \begin{equation}\label{KPZformula}
\langle \prod_{k=1}^N V_{\alpha_k}(\psi(z_k))    \rangle=  \prod_{k=1}^N |\psi'(z_k)|^{-2 \Delta_{\alpha_k}}     \langle \prod_{k=1}^N V_{\alpha_k}(z_k)     \rangle
\end{equation}  
where we use the physicists' notation $\langle\cdot\rangle$ for the average with respect to the measure $\nu$. The exponents $\Delta_{\alpha_k}$ are called {\it conformal weights}. In LCFT the primary fields are called {\it vertex operators} in the physics jargon and are given by
 \begin{equation}\label{KPZformula1}
V_{\alpha}(z)  = e^{\alpha\phi(z)}
\end{equation}  
for suitable $\alpha\in\C$. As $\phi$ turns out to be a distribution valued random field, this definition requires a regularization and renormalization procedure (Section \ref{sec:backgr}). In LCFT, the correlations in \eqref{KPZformula} were defined in \cite{DKRV} for $N \geq 3$ and under certain assumptions on the $(\alpha_k)_{1 \leq k \leq N}$ called the Seiberg bounds:
\begin{equation}\label{TheSeibergbounds}
 \sum_{k=1}^N\alpha_k>2Q, \quad \forall k, \; \alpha_k<Q
\end{equation}

The crucial property of a CFT with central charge $c$ is however not the global conformal transformation property \eqref{KPZformula} but rather {\it local conformal invariance}. To define this in the probabilistic setup requires considering variation of the measure $\nu$ under a change of the geometry of the surface where the fields are defined. In general in local Quantum Field Theory one expects that such a variation is encoded in a random field, the {\it stress-energy tensor} (abbreviated SE tensor hereafter). Briefly,  spelling this out in our 2d set-up, suppose the measure $\nu$ of the CFT can be defined in a set-up where the surface carries a smooth Riemannian metric $G=\sum_{i,j=1}^2 g_{ij}dx^i \otimes dx^{j}$; in this context, we will denote by $\langle \cdot  \rangle_{G}$  averages with respect to the CFT in the background metric $G$. Let $g^{ij}$ be the inverse matrix  $\sum_{j=1}^2 g^{ij}g_{jk}=\delta^i_k$. 
Consider a one parameter family $G_\epsilon$ where  $g_\epsilon^{ij}=g^{ij}+\epsilon f^{ij}$ whith $f$  a smooth function with support in $\C\setminus \cup_{k=1}^N z_k$.
 Then one expects 
\begin{equation}\label{defgen}  
 \frac{_d}{^{d\epsilon}}\mid_{\epsilon=0} \langle   \prod_{k=1}^N V_{\alpha_k}(z_k)   \rangle_{G_\epsilon}= \sum_{i,j=1}^2 \frac{1}{4\pi}   \int_{\C} f^{ij}(z)\langle T_{ij}(z) \prod_{k=1}^N V_{\alpha_k}(z_k)    \rangle_G  \, \text{vol}_G(d ^2z).
 \end{equation}
where $\text{vol}_G(d ^2z)$ is the volume form of $G$ and $T_{ij}$ is by definition the SE tensor. In CFT, the SE tensor has two nontrivial components: in the complex coordinates they are $T(z):=T_{zz}(z)$ and $\bar T(z):=T_{\bar z\bar z}(z)$. Then, BPZ argued that 
 $T(z)$ encodes { local conformal symmetries} through  the {\it Conformal Ward Identities}. The first Ward identity says that the correlation function is meromorphic in the argument $z$ of $T(z)$ with prescribed singularities: 
\begin{equation}
  \langle T(z) \prod_{k=1}^N V_{\alpha_k}(z_k)   \rangle= \sum_{i=1}^N \frac{\Delta_{\alpha_i} }{(z-z_i)^2} \langle  \prod_{k=1}^N V_{\alpha_k}(z_k)   \rangle   +\sum_{i=1}^N \frac{1}{z-z_i} \partial_{z_i}\langle  \prod_{k=1}^N V_{\alpha_k}(z_k)   \rangle  \quad\label{wardid1}.
 \end {equation}
 Note in particular that the $T$ insertion  is holomorphic. The second Ward  identity controls the singularity when two $T$-insertions come close  \begin{align}  \nonumber
    \langle T(z)T(z') \prod_{k=1}^N V_{\alpha_k}(z_k)   \rangle
&=\frac{1}{2}\frac{c }{(z-z')^4}  \langle  \prod_{k=1}^N V_{\alpha_k}(z_k)   \rangle +\frac{2 }{(z-z')^2} \langle T(z') \prod_{k=1}^N V_{\alpha_k}(z_k)   \rangle\\& +\frac{1 }{z-z'} \partial_{z'} \langle T(z') \prod_{k=1}^N V_{\alpha_k}(z_k)   \rangle+\dots\label{wardid2}
 \end {align}
where the dots refer to terms that are bounded as $z\to z'$. Recall that here $c$ is the central charge of the CFT. $\bar T(z)$ satisfies the complex conjugate  identities. 

As stressed by BPZ, these Ward identities are the fundamental structural property of a CFT and they are the starting point in showing that the CFT gives rise to a representation of the Virasoro Algebra with central charge $c$. In this paper, we define the SE tensor rigorously for LCFT and prove the Ward identities: see Theorem \ref{wardth}\footnote{In the context of LCFT, the SE tensor defined formally by \eqref {defforus} below is nothing but the quantum analog of the classical SE tensor $T(\phi_\star)= - (\partial_z \phi_\star)^2+ \frac{2}{\gamma} \partial_{zz}^2 \phi_\star$ where $\phi_\star$ minimizes  the Liouville action $S_L$ given by \eqref{liouvlaw1}. The classical SE tensor was  introduced more than a century ago by Poincar\'e in his theory of the unifomisation of Riemann surfaces: one can read the nice introduction of \cite{TaZo} on this point.}. As an output of these Ward identities, we will adress the construction of the representation of the Virasoro Algebra in a forthcoming work. 

The second set of fundamental identities for a CFT uncovered by BPZ goes under the name {\it BPZ-equations}. These are differential equations for correlation functions of the CFT that can be used to actually determine some of the correlation functions. BPZ uncovered a set of  {\it degenerate fields} 
whose insertions to correlation functions lead to differential  relations as in the case of Ward identities. In LCFT the two simplest degenerate fields are given by the vertex operators 
$V_{-\frac{\gamma}{2}}$ and $V_{-\frac{2}{\gamma}}$: we will show that they satisfy the following second order linear differential equations:
\begin{align}\label{bpzeq}
 (\frac{1}{\alpha^{2}}\partial_{z}^2   + \sum_{i=1}^N \frac{\Delta_{\alpha_i}}{(z-z_i)^2}   +  \sum_{i=1}^N \frac{1}{z-z_i}  \partial_{z_i}  )\langle   V_\alpha(z) \prod_{k=1}^N V_{\alpha_k}(z_k)   \rangle    =  0.
\end{align}
where $\alpha=-\frac{\gamma}{2}$ and  $\alpha=-\frac{2}{\gamma}$ respectively.

Furthermore, we prove that equation \eqref{bpzeq} can be uniquely solved for $N=3$ in terms of hypergeometric functions and the so-called three point structure constants of LCFT: this is because the global conformal invariance property \eqref{KPZformula} enables one to map equation \eqref{bpzeq} to a standard hypergeometric partial differential equation when $N=3$. Let us mention that a key point in our analysis of equation \eqref{bpzeq} for $N=3$ is the novel observation that assuming global conformal invariance \eqref{KPZformula} the real valued solution space to \eqref{bpzeq} is in fact (for most values of $\gamma$ and the $(\Delta_{\alpha_i})_{1 \leq i \leq N}$) a real one dimensional space: see lemma \ref{lemmaPDEsbis} in the Appendix for a precise statement on the hypergeometric partial differential equation. Now, one can notice that in a CFT the global conformal symmetry \eqref{KPZformula} fixes the three point correlation functions up to some fundamental constant $C_\gamma(\alpha_1, \alpha_2, \alpha_3)$: more precisely, one has
\begin{equation*}
 \langle      \prod_{k=1}^3 V_{\alpha_k}(z_k)  \rangle 
 =  |z_1-z_2|^{ 2 \Delta_{12}}  |z_2-z_3|^{ 2 \Delta_{23}} |z_1-z_3|^{ 2 \Delta_{13}}C_\gamma(\alpha_1,\alpha_2,\alpha_3)
\end{equation*}
with $\Delta_{12}= \Delta_{\alpha_3}-\Delta_{\alpha_1}-\Delta_{\alpha_2}$, etc. These constants $C_\gamma(\alpha_1, \alpha_2, \alpha_3)$ are called the three point structure constants and are building blocks of LCFT in the so-called conformal bootstrap approach. The bootstrap approach (see the review \cite{Rib}) aims to give a construction of LCFT based on an exact formula for $C_\gamma(\alpha_1, \alpha_2, \alpha_3)$, the celebrated DOZZ-formula (after Dorn-Otto-Zamolodchikov-Zamolodchikov \cite{DoOt,ZaZa}) and certain recursive rules to deduce from the DOZZ formula the n-point correlations for $n>3$.\footnote{In the bootstrap approach, these correlations are expressed as sums involving holomorphic (and anti holomorphic) functions and the constants $C_\gamma(\alpha_1, \alpha_2, \alpha_3)$.} Using the solution of \eqref{bpzeq} and setting $l(x)=\Gamma(x)/\Gamma(1-x)$ we prove for $\alpha_1,\alpha_2,\alpha_3$ satisfying the Seiberg bounds \eqref{TheSeibergbounds}  that
\begin{equation}\label{3pointconstanteqintro}
\frac{C_\gamma(\alpha_1+\frac{\gamma}{2},\alpha_2,\alpha_3)}{C_\gamma(\alpha_1-\frac{\gamma}{2},\alpha_2,\alpha_3)}  
= - \frac{1}{\pi \mu}\frac{l(-\frac{\gamma^2}{4})  l(\frac{\gamma \alpha_1}{2}) l(\frac{\alpha_1\gamma}{2}  -\frac{\gamma^2}{4})  l(\frac{\gamma}{4} (\bar{\alpha}-2\alpha_1- \frac{\gamma}{2}) )   }{l( \frac{\gamma}{4} (\bar{\alpha}-\frac{\gamma}{2} - 2Q)  ) l( \frac{\gamma}{4} (\bar{\alpha}-2\alpha_3-\frac{\gamma}{2} ))  l( \frac{\gamma}{4} (\bar{\alpha}-2\alpha_2-\frac{\gamma}{2} )) }. 
\end{equation}
where $\bar{\alpha}=\alpha_1+\alpha_2+\alpha_3$. This is the content of Corollary \ref{3pointconstant}. Relation \eqref{3pointconstanteqintro} was obtained earlier in the physics literature by Teschner  \cite{Tesc} using clever  but non rigorous conformal bootstrap techniques. Relation \eqref{3pointconstanteqintro} is a major first step in a rigorous proof of the DOZZ formula: see next subsection. 

Finally, the correlation functions of primary fields in CFT are singular as two points $z_i,z_j$ get together: 
\begin{align}\label{ope}
\langle   \prod_{k=1}^N V_{\alpha_k}(z_k)   \rangle\sim |z_i-z_j|^{-\eta_{ij}}
\end{align}
with $\eta_{ij} \in \R$ and $\sim$ denotes equivalence (up to logarithmic corrections) as $|z_i-z_j| \to 0$. The study of these singularities is important as it gives information about the {\it fusion rules}, another of the structures of CFT uncovered by BPZ. More generally, the so-called operator product expansion in the physics literature corresponds to studying equivalence \eqref{ope} at order 1 and higher in $|z_i-z_j|$. In Liouville theory,  the operator product expansion is quite subtle and it is expected that the asymptotics in \eqref{ope} have logarithmic corrections in $ |z_i-z_j|$\footnote{This does not mean that LCFT is a logarithmic CFT, a variant of a classical CFT.}.   In this paper, as technical building blocks in the proof of the Ward identities and the BPZ equations, we prove detailed estimates on these singularities and in particular we prove that the logarithmic corrections indeed are present: see section \ref{sectionestimates} where are stated general fusion estimates.

Let us next describe briefly the mathematical content of the paper. The Liouville term $\int_{\C} e^{\gamma\phi(z)}d^2z$ in the action functional \eqref{liouvlaw1} is an example of {\it Gaussian Multiplicative Chaos} (GMC), a random multifractal measure on $\C$. The Liouville correlation functions turn out to be nonlinear functionals of the GMC measure and the study of their regularity boils down to a careful analysis of the GMC measure. In particular a variation of the {\it freezing} phenomenon familiar in the theory of GMC enters the analysis in a essential way. The detailed regularity properties of the correlation functions are a necessary input for the proof of the Ward and BPZ identities which are based on certain exact identities involving the correlation functions and non trivial integral transforms of these correlation functions (these exact identities are obtained through Gaussian integration by parts). For instance the insertion of the SE tensor $T(z)$ in the correlation functions leads to expressions involving Beurling and more singular integral transforms of them. The Ward and BPZ identities are then the consequence of subtle cancellations of not absolutely convergent integrals and require great care.

\subsection{Perspectives}\label{Perspectives}

This paper is the first in a series of papers aiming at unifying  two approaches of LCFT in the physics literature: the path integral approach and the conformal bootstrap approach. More precisely, the above results (Ward identities, BPZ equations, relation \eqref{3pointconstanteqintro}) are essential in this direction. We believe that both approaches are equivalent though the status of their relation is still controversial in the physics literature. Indeed, there are numerous reviews and papers within the physics literature on the path integral approach of LCFT and its relation with the bootstrap approach but they offer different perspectives and conclusions (see \cite{HaMaWi,OPS,seiberg} for instance). One major step towards this unification would be the proof that $C_\gamma(\alpha_1,\alpha_2,\alpha_3)$ indeed satisfies the DOZZ formula, namely that one can analytically continue $(\alpha_1,\alpha_2,\alpha_3) \mapsto C_\gamma(\alpha_1,\alpha_2,\alpha_3)$ to a set $\C^3 \setminus \mathcal{P}_\gamma$ where $\mathcal{P}_\gamma$ is a (rather complicated) set of poles depending on $\gamma$ and such that on $\C^3 \setminus \mathcal{P}_\gamma$ we have (recall that $l(x)=\Gamma(x)/\Gamma(1-x)$)    
\begin{equation}\label{DOZZformula}
C_\gamma(\alpha_1,\alpha_2,\alpha_3)= (\pi \: \mu \:  l(\frac{\gamma^2}{4})  \: (\frac{\gamma}{2})^{2 -\gamma^2/2} )^{\frac{2 Q -\sum_i \alpha_i}{\gamma}}   \frac{\Upsilon_{\frac{\gamma}{2}}'(0)\Upsilon_{\frac{\gamma}{2}}(\alpha_1) \Upsilon_{\frac{\gamma}{2}}(\alpha_2) \Upsilon_{\frac{\gamma}{2}}(\alpha_3)}{\Upsilon_{\frac{\gamma}{2}}(\frac{\bar{\alpha}-2Q}{2}) 
\Upsilon_{\frac{\gamma}{2}}(\frac{\bar{\alpha}-\alpha_1}{2}) \Upsilon_{\frac{\gamma}{2}}(\frac{\bar{\alpha}-\alpha_2}{2}) \Upsilon_{\frac{\gamma}{2}}(\frac{\bar{\alpha}-\alpha_3}{2})   } 
\end{equation}  
%
(recall   that $\bar{\alpha}= \alpha_1+\alpha_2+\alpha_3$) and $\Upsilon_{\frac{\gamma}{2}}$ is a  special function (depending on $\gamma$): see subsection \ref{expressionUpsilon}  for an analytic expression. 
The function $\Upsilon_{\frac{\gamma}{2}}$ has zeros but no poles and so the poles $\mathcal{P}_\gamma$ of $C_\gamma(\alpha_1,\alpha_2,\alpha_3)$ can be read off formula \eqref{DOZZformula} and correspond to the zeros of the denumerator. Assuming some regularity (in $\gamma$ and $\alpha_1,\alpha_2,\alpha_3$) on $C_\gamma(\alpha_1,\alpha_2,\alpha_3)$ the DOZZ formula is in fact the only solution on $\mathcal{P}_\gamma$ of shift equation \eqref{3pointconstanteqintro} and the following dual shift equation 
\begin{equation}\label{3pointconstanteqintrodual}
\frac{C_\gamma(\alpha_1+\frac{2}{\gamma},\alpha_2,\alpha_3)}{C_\gamma(\alpha_1-\frac{2}{\gamma},\alpha_2,\alpha_3)}  
= - \frac{1}{\pi \tilde{\mu}}\frac{l(-\frac{4}{\gamma^2})  l(\frac{2 \alpha_1}{\gamma}) l(\frac{2\alpha_1}{\gamma}  -\frac{4}{\gamma^2})  l(\frac{1}{\gamma} (\bar{\alpha}-2\alpha_1- \frac{2}{\gamma}) )   }{l( \frac{1}{\gamma} (\bar{\alpha}-\frac{2}{\gamma} - 2Q)  ) l( \frac{1}{\gamma} (\bar{\alpha}-2\alpha_3-\frac{2}{\gamma} ))  l( \frac{1}{\gamma} (\bar{\alpha}-2\alpha_2-\frac{2}{\gamma} )) }. 
\end{equation}  
\vspace{0.1 cm}
with $\tilde{\mu}= \frac{(\mu \pi l(\frac{\gamma^2}{2})  )^{\frac{4}{\gamma^2} }}{ \pi l(\frac{4}{\gamma^2})}$.
Our work proves that $C_\gamma(\alpha_1,\alpha_2,\alpha_3)$ satisfies the first shift equation for restricted values of $\alpha_1,\alpha_2,\alpha_3$. In a forthcoming work \cite{KRV}, we address the problem of the analytic continuation of $C_\gamma(\alpha_1,\alpha_2,\alpha_3)$ as a function of $\alpha_1,\alpha_2,\alpha_3$ (hence showing that one can lift the restriction on the values of $\alpha_1,\alpha_2,\alpha_3$) and the problem of proving the second shift equation \eqref{3pointconstanteqintrodual}.

Let us also emphasize that the functional relations \eqref{3pointconstanteqintro} and \eqref{3pointconstanteqintrodual}  are very important relations per se for all $\gamma \in \C$ (and not just real $\gamma \in ]0,2]$)\footnote{Recall that our approach, based on a path integral, is presently restricted to the case of real $\gamma$ in the interval $]0,2]$}.  Indeed, for each fixed $\gamma \in \C$, they are used in the physics literature to construct three point structure constants and CFTs in the conformal bootstrap approach: see for instance the very recent bootstrap construction of Liouville theory by Ribault-Santachiara with purely imaginary $\gamma$ \cite{Rib1} based on the exact solution to \eqref{3pointconstanteqintro} and \eqref{3pointconstanteqintrodual} discovered in \cite{KoPe,Za}\footnote{The exact solution to \eqref{3pointconstanteqintro} and \eqref{3pointconstanteqintrodual} for purely imaginary $\gamma$ is not the analytic continuation in $\gamma$ of the DOZZ formula \eqref{DOZZformula} valid for $\gamma \in ]0,2]$}.  

Finally, let us mention that the Liouville three point structure constants $C_\gamma(\alpha_1, \alpha_2, \alpha_3)$ are also particularly interesting as they seem to have deep connections to other  mathematical objects, e.g. the Nekrasov partition functions: this is the basis of the celebrated AGT conjecture \cite{AGT}. This brings yet additional motivation to study the structure constants. On the mathematical side, the AGT conjecture has been explored recently by Maulik and Okounkov in \cite{MO}.

\subsection{History on LCFT and probabilistic approaches to CFT}

Finally we want to make some comments about  other mathematical studies of LCFT and CFTs. There is a huge physics literature on LCFT for which we refer the reader to the reviews \cite{Tesc1,nakayama,Rib} and the previous references. Takhtajan et al.  \cite{TT} studied Liouville theory in the setup of  a formal  power series in $\gamma$  (the so-called semiclassical expansion). In particular  Ward identities were established in this formal power series context. Also, independently from the work \cite{DKRV}, Duplantier-Miller-Sheffield developped a theory of quantum surfaces in \cite{DMS} which lies in some sense at the boundary of LCFT. In the case of the sphere, they work with two marked points $0$ and $\infty$ and they consider random measures defined up to dilation and rotation since two points are not sufficient to determine conformal embeddings in the sphere. From the point view of LCFT, this corresponds to the construction of the two point correlation functions which exist only in a generalized sense: this point will be explained in more detail in the forthcoming work \cite{KRV}. Their theory is based on a coupling between variants of SLE curves joining the two marked points and the full plane GFF; their framework is interesting in the study of the relation between  LCFT and random planar maps.   

On the constructive probabilistic side there are very few results on  conformal invariance and Ward identities in other CFTs. 
For the {\it Gaussian Free Field} (GFF) a complete descripition of the CFT  is developed in Kang and Makarov's monograph \cite{KM}. In particular, they derive Ward and BPZ identities for this particular CFT with central charge $c_{\mathrm{GFF}}=1$. Essentially, the GFF  corresponds to setting the cosmological constant $\mu$ to $0$. The resulting CFT has a very different structure from the Liouville case. 

For interacting field theories, following the breakthrough papers by Smirnov \cite{S1} and then Chelkak-Smirnov \cite{CS}, the scaling limit of the discrete {\it Ising model} in a domain at critical temperature is very well understood mathematically. The scaling limit of the model when the mesh size goes to zero is proven to be described by a unitary CFT, the Ising CFT, which belongs to the class of the so-called minimal models. The Ising CFT is composed of two non trivial primary fields: the spin and the energy density. The rescaled correlations of the spin were proven to converge to an explicit expression in Chelkak-Hongler-Izyurov \cite{CHI} (see also the independent work by Dub\'edat \cite{Dub2} for convergence in the plane). In fact, one can also construct the scaling limit of the spin as a random distribution: see Camia-Garban-Newman \cite{CGN}. The rescaled correlations of the energy density were proven to converge to an explicit expression in Hongler-Smirnov \cite{HS} (see also the independent work by Boutillier-De Tili\`ere \cite{BD,BD1} for convergence in the plane). Finally, a recent work \cite{CGS} shows convergence of a discrete SE tensor to a continuum limit which satisfies the associated Ward identities.

In the context of {SLE and CLE}, which are random conformally invariant curves,  one can construct partition functions or correlation functions as probabilities of events related to several SLEs or CLE. The correlation functions correspond to CFTs with central charge $c \leq 1$. In this context, the understanding of the correlation functions is quite precise though not complete: see Dub\'edat \cite{Dub1}, Bauer-Bernard-Kytol\"a \cite{BBK} for multiple SLEs and more recently Camia- Gandolfi-Kleban \cite{CGK} for a construction of primary fields in the context of CLE. In the context of CLE, Doyon \cite{Do} has proposed a construction of the stress-energy tensor but it relies on assumptions which have not been proved yet. 


\medskip

\subsection{Organization of the paper}

The rest of the paper is organized as follows. In the next section, we set the notations, recall some results of \cite{DKRV} and state the main results of the paper: the Ward and BPZ identities. In Section 3, we prove results on differentiability properties of Liouville correlation functions and based on these in Section 4 we prove the Ward and BPZ identities. Finally in Section 5 we prove detailed bounds on the correlation functions when two or three points get together. This section is the technical backbone of our paper and can be seen as a first step towards proving fusion rules and operator product expansions for LCFT.

\subsubsection*{Acknowledgements} The authors wish to thank Francois David and Sylvain Ribault for fruitful discussions on Liouville Field theory. The authors also wish to thank Colin Guillarmou for helping them handle the hypergeometric equation which arises in the study of the four point correlation function.


\section{Main results}

\subsection{Background and notations}\label{sec:backgr}
In this section, we recall the precise definition of the Liouville correlation functions  as given in  \cite{DKRV}.

\subsubsection{LCFT on $\hat\C$. } 
Recall that we denote by $\hat\C$ the Riemann sphere. A smooth conformal metric on $\hat\C$ is given by $G=g(z)|dz|^2$ where $g(z)=\caO(|z|^{-4})$ as $z\to\infty$.  It is then natural to write the metric in LCFT as $e^{\gamma\phi(z)}|dz|^2=e^{\gamma X(z)}g(z)|dz|^2$ and define  the Liouville action functional  on $\hat\C$ by
\begin{equation}\label{actionLiouville}
S_c(X,{g}):= \frac{1}{\pi}\int_{\C}\big(|\partial_zX (z) |^2+ \frac{{Q_c}}{4}g(z)R_{{g}}(z) X(z) +\pi \mu e^{\gamma X(z) }g(z)\big)\,d^2 z   .
\end{equation}
Here $R_g(z)=-4g^{-1}\partial_z\partial_{\bar z}\ln g(z)$ is the scalar curvature of the metric $g$ and $Q_c=\frac{2}{\gamma}$. 
Note that formally \eqref{actionLiouville}  agrees with \eqref{liouvlaw1} with $\phi(z)= X(z)+\frac{1}{\gamma} \ln g(z)$ up to an infinite constant which stems from the fact that  $\C$ and $\hat\C$ are not conformally equivalent. To have conformal invariance of the probabilistic theory it turns out one needs to change the parameter  $Q_c=\frac{2}{\gamma}$ in the classical action  \eqref{actionLiouville} to the "quantum" value  \eqref{Qdef}. This yields the following definition for the quantum action
\begin{equation}\label{actionLiouvilleQ}
S(X,{g}):= \frac{1}{\pi}\int_{\C}\big(|\partial_zX (z) |^2+ \frac{{Q}}{4}g(z)R_{{g}}(z) X(z) +\pi \mu e^{\gamma X(z) }g(z)\big)\,d^2 z   .
\end{equation}

We work in this paper with the round metric
 $$\hat{g}(z)=\frac{4}{(1+\bar zz)^2}
$$
in which case the scalar curvature is constant $R_{\hat g}=2$. The smooth metrics on $\hat\C$ are then given by  $e^{\varphi(z)}\hat g(z) |dz|^2$ with $\varphi(z)$ and   $\varphi(1/z)$ smooth on $\C$. In \cite{DKRV} it was shown that the change  of the Liouville correlation functions under change of $\varphi$ is explicit through the so-called Weyl anomaly formula and thus there is no loss working with $\hat g$. In the sequel, we will therefore work in the background metric $\hat{g}$.

\subsubsection{Convention and notations.} In what follows, in addition to the complex variable $z$, we will also consider variables $x,y$ in $\C$ and for integer $N \geq 3$ variables  $z_1, \cdots, z_N$ which also belong to $\C$. For these complex variables, we will use the standard notation for complex derivatives: for $x=x_1+i x_2$  we set $\partial_x= \frac{1}{2}(\partial_{x_1}- i \partial_{x_2})$ and $\partial_{\bar{x}}= \frac{1}{2} (\partial_{x_1}+ i \partial_{x_2})$. The variables $x,y$ will typically  be variables of integration: we will denote by $d^2x$ and $d^2y$ the corresponding Lebesgue measure on $\C$ (seen as $\R^2$). We will also denote $|\cdot|$ the norm in $\C$ of the standard Euclidean (flat) metric and for all $r>0$ we will denote by $B(x,r)$ the Euclidean ball of center $x$ and radius $r$. 

\subsubsection{Gaussian Free Field. } To define the measure  \eqref{liouvlaw} it is natural to start with the quadratic part of the action functional  \eqref{actionLiouville}  which naturally gives rise to a Gaussian measure, the Gaussian Free Field (GFF) (we refer to section 4 in \cite{Dub0} or \cite{She07} for an introduction to the GFF). As is well known the GFF on the plane is defined modulo a constant but in LCFT this constant has to be included as an integration variable in the measure \eqref{liouvlaw}. The way to proceed is to replace $X$ in \eqref{actionLiouville}  by $c+X$  where $c\in\R$ and $X$ is the Gaussian Free Field on $\C$ where the additive constant is fixed by $\int_\C X(x)\hat{g}(x)d^2x=0$. The covariance of $X$\footnote{The field $X$ was denoted $X_{\hat{g}}$ in the article \cite{DKRV} or the lecture notes \cite{RV}.} is given explicitely for $x,y \in \C$ by 
\begin{equation}\label{hatGformula}
\E [ X(x)X(y)] = G(x,y)=\ln\frac{1}{|x-y|}-\frac{1}{4}(\ln\hat g(x)+\ln\hat g(y))+\chi
\end{equation}
where $\chi:=\ln 2-\frac{1}{2}$.

\subsubsection{Gaussian multiplicative chaos. }
The field $X$ is distribution valued and to define its exponential a renormalization procedure is needed. We will work with a mollified regularization of the GFF, namely $X_{\epsilon}=\rho_\epsilon \ast  X   $ with $\rho_\epsilon(x)= \frac{1}{\epsilon^2}\rho (\frac{\bar xx}{\epsilon^2})$ where $\rho$ is $C^{\infty}$ non-negative with compact support in $[0,\infty[$ and such that $\pi \int_0^\infty \rho(t) dt=1$.  The variance of  $X_{\epsilon}(x)$ satisfies
 \begin{equation}\label{circlegreen}
\lim_{\epsilon\to 0} \: (\E[X_{\epsilon} (x)^2]+\ln (a\epsilon))
=-\hf\ln \hat{g}(x)
\end{equation} 
uniformly on $\C$ where the constant $a$ depends on the regularization function $\rho$. 
Define the measure\footnote{This normalization is chosen to match with the standard  physics literature.} 
\begin{equation}\label{Meps1}
M_{\gamma,\epsilon}(d^2x):=
e^{\frac{\gamma^2}{2}\chi}
e^{\gamma X_{\epsilon}(x)-\frac{\gamma^2}{2} \E[ X_{\epsilon}(x)^2] } \hat{g}(x) d^2x.
\end{equation}
Then, 
for $\gamma\in [0,2[$,   we have the convergence  in probability
\begin{equation}\label{law}
M_{\gamma}=\lim_{\epsilon\to 0}M_{\gamma,\epsilon}
 \end{equation}
 and convergence is in the sense of weak convergence of measures. This limiting measure is non trivial and is a (up to a multiplicative constant) Gaussian multiplicative chaos (GMC) of the field $X$ with respect to the measure ${\hat{g}}(x)dx$ (see \cite{Ber} for an elementary approach and references). By \eqref{circlegreen} we may also write 
 \begin{equation}\label{Meps}
M_{\gamma}(d^2x)=\lim_{\epsilon\to 0}(A\epsilon)^{\frac{\gamma^2}{2}}e^{\gamma \phi_\epsilon(x)}  \, d^2x
\end{equation}
where the constant $A=ae^\chi$, $\phi_\epsilon=\rho_\epsilon\ast\phi$  and $\phi$ is  the {\it Liouville field}
\begin{equation}\label{Liouville field}
\phi=X+\frac{_Q}{^2}\ln\hat{g}.   \footnote{This is a slight abuse of terminology as we use a slightly different convention than the article \cite{DKRV} or the lecture notes \cite{RV} where the Liouville field is rather $\phi+c$.}
\end{equation}
Note that $Q$ here is given by \eqref{Qdef} and not $Q_c$. The change is due to the renormalization in \eqref{Meps1}.

\subsubsection{Liouville measure } The Liouville measure $e^{-S(X,\hat{g})}DX$ with $S$ given by \eqref{actionLiouvilleQ}  is now defined as follows. Since $R_{\hat g}=2$ and $\int_\C X(x)\hat g(x) d^2x=0$ the linear term becomes
$
\int_\C(c+X(x)) \hat g (x) d^2x=4\pi c
$.  This leads to the
definition 
\begin{equation}\label{Liouville measure}
 \nu(dX,dc):=e^{-2Q c}     e^{- \mu e^{\gamma c} M_{\gamma}(\C)    }\P(dX)   \: dc.
\end{equation}
Here $\P(dX)$ is the measure on $H^{-1}(\hat \C)$ induced by the GFF $X$. Note that the random variable $M_{\gamma}(\C) $ is almost surely finite because $\E M_{\gamma}(\C) =e^{\frac{\gamma^2}{2}\chi}\int_\C \hat g(x) dx<\infty
$. This implies that $\nu$ is an infinite measure on $H^{-1}(\hat \C)$: $\int d\nu=\infty$ since for $M_{\gamma}(\C) <\infty$ the $c$-integral diverges at $-\infty$. However, suitable correlation functions still exist as we see next.

\subsubsection{Liouville correlation functions.} The vertex operators  \eqref{KPZformula1}  need to be regularized as well: we set
\begin{equation}\label{Vdefi}
V_{\alpha, \epsilon}(z)= (A\epsilon)^{\frac{\alpha^2}{2}}e^{\alpha (\phi_\epsilon(z)+c)}=e^{\alpha c} e^{\frac{\alpha^2 \chi}{2}}e^{\alpha X_{\epsilon}(z)-\frac{\alpha^2}{2} \E[ X_{\epsilon}(z)^2] }\hat g_\epsilon (z) ^{\Delta_\alpha}
\end{equation}
where $\hat g_\epsilon (z) = e^{(\rho_\epsilon \ast \ln \hat{g} )(z)}$. Let us denote 
averages with respect to $\nu$ by $\langle\,\cdot\,\rangle
$ and 
\begin{equation}\label{defU}
U_N= \lbrace \z=(z_1, \dots, z_N) \in \mathbb{C}^N, \:  \: z_i \not = z_j \:  \forall i \neq j \rbrace.
\end{equation}
Fix $\z\in U_N$ and 
 weights $\alpha_1,\dots,\alpha_N$. Here and below we use the notation $\langle\,\cdot\,\rangle_\epsilon$ for the regularized Liouville measure where in \eqref{Liouville measure} we use the measure $V_{\gamma, \epsilon} (x) d^2x$ instead of $M_{\gamma}$. Similarly, standard GMC theory ensures that $V_{\gamma, \epsilon} (x) d^2x$ converges in probability to $M_\gamma$ as $\epsilon\to 0$.  Now, it was shown in \cite{DKRV} that the  limit
\begin{equation}\label{lcorre}
 \langle \prod_{k=1}^N  V_{\alpha_k}(z_k) \rangle :=  4 e^{-2\chi Q^2}\: \underset{\epsilon \to 0}{\lim}  \:  \langle \prod_{k=1}^N  V_{\alpha_k,\epsilon}(z_k)  \rangle_\epsilon\footnote{The global constant $4 e^{-2\chi Q^2}$ which depends on $\gamma$ plays no role but it is included to match with the standard  physics literature which is based on the DOZZ formula \eqref{DOZZformula}. This constant was not included in the definitions in the article \cite{DKRV} or the lecture notes \cite{RV}.}
\end{equation}
exists and is finite if and only if $ \sum_{k=1}^N\alpha_k>2Q$.  Moreover, under this condition, the limit is non zero  if and only if $\alpha_k<Q$ for all $k$. These conditions are the Seiberg bounds  \eqref{TheSeibergbounds} originally introduced in \cite{seiberg}.
 {\it We assume throughout this paper that the Seiberg bounds \eqref{TheSeibergbounds} are satisfied and that $\alpha_k \neq 0$ for all $k$.}\footnote{This is no restriction since the case $\alpha_k=0$ corresponds to setting $V_{\alpha_k}(z_k)=1$.} Note that these bounds imply that for a nontrivial correlation we need at least {\it three} vertex operators; therefore, we have $N \geq 3$ in the sequel.  
The correlation function \eqref{lcorre}  satisfies the    conformal invariance property \eqref{KPZformula} and the conformal weights of the vertex operators are given by
$$\Delta_{\alpha}= \frac{\alpha}{2} (Q-\frac{\alpha}{2}).$$

\subsubsection{Reduction to Multiplicative Chaos.} In order to keep this work as self contained as possible, we remind the basics of the construction of the Liouville correlations. The main idea is that one can express these correlations as functions of GMC measures with log singularities. More precisely, using the explicit expression  \eqref{Vdefi}, we may  first integrate over the $c$-variable in the vertex operator correlation functions which yields (see  \cite{DKRV}):
  \begin{align*}
   \langle \prod_{k=1}^N  V_{\alpha_k}(z_k)  \rangle& = 4 \: e^{-2\chi Q^2}\:  \lim_{\epsilon\to 0} \int_{\R}   e^{-2Q c}  \:   \E \left [  \prod_{k=1}^N  V_{\alpha_k,\epsilon}(z_k)        e^{- \mu e^{\gamma c}  \int_{\C}  V_{\gamma, \epsilon} (x) d^2x    } \right ]   \: dc\\
 &=  \mu^{-s} 4 \: e^{\frac{\chi}{2}\sum_{k=1}^N\alpha_k^2-2\chi Q^2}\gamma^{-1}\Gamma(s)\lim_{\epsilon\to 0}\E  \left [  \prod_{k=1}^N e^{{\alpha_k} X_\epsilon(z_k)- \frac{\alpha_k^2}{2}\E X_\epsilon(z_k)^2}   \hat{g}_\epsilon (z_k)^{\Delta_{\alpha_k}}    \left (   \int_{\C} V_{\gamma, \epsilon} (x) d^2x \right )  ^{-s}    \right ].
\end{align*}
Now, using Girsanov's theorem  (see  \cite{DKRV}) we may trade the vertex operators to a shift of $X$ to obtain
\begin{equation}\label{funamental}
 \langle    \prod_{k=1}^N V_{\alpha_k}(z_k)  \rangle   =   4 \: e^{\frac{\chi}{2}\sum_{k=1}^N\alpha_k^2-2\chi Q^2}e^{\frac{1}{2}\sum_{i \not = j}^N \alpha_i \alpha_j G(z_i,z_j)}(\prod_{k=1}^N\hat{g} (z_k)^{\Delta_{\alpha_k}})\mu^{-s} \gamma^{-1}\Gamma(s)\E \left [ \left ( \int_{\C} \, e^{\sum_{k=1}^N \alpha_k G(x,z_k)}M_\gamma(d^2x) \right )^{-s}  \right ]
\end{equation}
where $s=\frac{\sum_{k=1}^N \alpha_k-2Q}{\gamma}$. Note that the Seiberg bounds ensure $s>0$.
Using formula \eqref{hatGformula} we can write this as
\begin{equation}
 \langle    \prod_{k=1}^N V_{\alpha_k}(z_k)  \rangle =B(\pmb{\alpha})\prod_{i < j} \frac{1}{|z_i-z_j|^{\alpha_i \alpha_j}} \mu^{-s} \gamma^{-1}\Gamma(s)\E \left [  \left (  \int_{\C}  F(x,{\bf z}) M_\gamma(d^2x)  \right )^{-s}  \right ]\label{Z1}
\end{equation}
where we set $\pmb{\alpha}=(\alpha_1, \cdots, \alpha_N), $
\begin{equation*}
F(x,{\bf z})=\prod_{k=1}^N \frac{ 1}{|x-z_k|^{\gamma \alpha_k}}   \hat{g}(x)^{  - \frac{\gamma}{4} \sum_{k=1}^N \alpha_k } 
\end{equation*}
and
\begin{equation}\label{Adefi}
B(\pmb{\alpha})=4 \: e^{-\frac{\chi}{2}(\sum_{k=1}^N\alpha_k-2Q)^2}.
\end{equation}
Thus, up to explicit factors the Liouville correlations are reduced to the study of the random variable $\int_{\C} F(x, {\bf z}) M_\gamma(d^2x) $. In particular, the Seiberg bounds $\alpha_k<Q$ for all $k$ are the condition of integrability of $F$ against the chaos measure  $M_\gamma$                  (see  \cite{DKRV}). 

\subsection{Ward and BPZ identities}
A formal calculation using the definition \eqref{defgen}  of the SE tensor yields the  following result in the case of LCFT 
\begin{equation}\label{defforus}
 T(z)=  Q \partial_{z}^2 \phi(z)- ( \partial_{z}\phi(z))^2+\E ((\partial_{z}X(z))^2).
 \end{equation}
 where $\phi$ is the {Liouville field}
  \eqref{Liouville field}.
  We will define this  via a regularized  version:
\begin{equation}\label{SETdefi}
T_\epsilon(z):=
Q \partial_{z}^2 \phi_\epsilon(z)-(\partial_z \phi_\epsilon(z))^2+\E( (\partial_{z} X_\epsilon(z))^2)
\end{equation}
(note that $X_\epsilon$ is smooth (a.s.)).
Here is the main theorem on the Ward identities:
\begin{theorem}\label{wardth}
{\rm (a)} The correlation functions $(z_1, \dots, z_N) \mapsto  \langle \prod_{k=1}^N  V_{\alpha_k}(z_k)  \rangle$ are $C^2$ in the set $U_N$.
\vskip 2mm
 
\noindent {\rm (b)} The limits
\begin{align}\label{Wardidentity1}
\lim_{\epsilon\to 0} \: \langle   T_{\epsilon}(z) \prod_{k=1}^N V_{\alpha_k}(z_k)   \rangle&:=\langle   T(z) \prod_{k=1}^N V_{\alpha_k}(z_k)   \rangle\\
\underset{\epsilon' \to 0}{\lim}\; \underset{\epsilon \to 0}{\lim}\;  \langle T_\epsilon(z) T_{\epsilon'}(z') \prod_{k=1}^N  V_{\alpha_k}(z_k)  \rangle&:=     \langle T(z) T(z') \prod_{k=1}^N  V_{\alpha_k}(z_k)  \rangle \label{Wardidentity2}
\end{align}
exist for all distinct $z, z', z_1, \cdots, z_N$. \eqref{Wardidentity1} is given by the first Ward identity \eqref{wardid1} and the second Ward identity \eqref{wardid2} holds in the form
\begin{align}
\langle T(z) T(z') \prod_{k=1}^N  V_{\alpha_k}(z_k)  \rangle& =\frac{1}{2} \frac{{c_L}}{(z-z')^4}  \langle \prod_{k=1}^N  V_{\alpha_k}(z_k)  \rangle+ \frac{2}{(z-z')^2} \langle T(z') \prod_{k=1}^N  V_{\alpha_k}(z_k)  \rangle
 +  \frac{1}{z-z'}\partial_{z'} \langle T(z') \prod_{k=1}^N V_{\alpha_k}(z_k)  \rangle \nonumber \\
 &+\sum_{i=1}^N\frac{\Delta_{\alpha_i}}{(z-z_i)^2} \langle T(z') \prod_{k=1}^N  V_{\alpha_k}(z_k)  \rangle+\sum_{i=1}^N\frac{1}{z-z_i} \partial_{z_i}\langle T(z') \prod_{k=1}^N V_{\alpha_k}(z_k)  \rangle  \label{Ward} 
\end{align}
where $c_L=1+6Q^2$ is the central charge of the Liouville theory.
\end{theorem}



The result on the BPZ equations is the following:
\begin{theorem}\label{BPZTH} Let  $\chi=-\frac{\gamma}{2}$ or  $\chi=-\frac{2}{\gamma}$ and suppose $\chi+\sum_{k=1}^N\alpha_k>2Q$. Then the BPZ equation \eqref{bpzeq} holds 
 in $\mathbb{C} \setminus  \lbrace  z_1, \cdots, z_N \rbrace$.
\end{theorem}



\subsection{Relations on  the 3 point structure constant}
We will use the BPZ equations to deduce an exact expression for the 4 point correlation function with a degenerate field  $\langle    V_{-\frac{\gamma}{2}}  (z)  \prod_{k=1}^3 V_{\alpha_k}(z_k)  \rangle$ and deduce from this expression a non trivial relation on the 3 point structure constants. This relation is usually referred to as Teschner's trick in the physics literature since it was shown by Teschner to lead to a simple heuristic derivation of the celebrated DOZZ formula for the 3 point structure constants (see section \ref{Perspectives} for the importance of these constants). 

Let us first use  M\"obius invariance \eqref{KPZformula} to simplify the three and four point functions. This fixes the three point function up to a constant
\begin{align}
  \langle      \prod_{k=1}^3 V_{\alpha_k}(z_k)  \rangle 
 & =  |z_1-z_2|^{ 2 \Delta_{12}}  |z_2-z_3|^{ 2 \Delta_{23}} |z_1-z_3|^{ 2 \Delta_{13}}C_\gamma(\alpha_1,\alpha_2,\alpha_3) \label{confinv3}
\end{align}
where we denoted $\Delta_{12}=\Delta_{\alpha_3}-\Delta_{\alpha_1}-\Delta_{\alpha_2}$ etc... Similarly, the
 four point function is fixed up to a single function depending on the cross ratio of the points. Specializing to the case we are interested in, we get
\begin{align}
  \langle    V_{-\frac{\gamma}{2}}  (z)  \prod_{k=1}^3 V_{\alpha_k}(z_k)  \rangle 
 & = |z_3-z|^{- 4 \Delta_{-\frac{\gamma}{2}}}  |z_2-z_1|^{ 2 (\Delta_3-\Delta_2-\Delta_1-\Delta_{-\frac{\gamma}{2}})  } |z_3-z_1|^{2(\Delta_2+\Delta_{-\frac{\gamma}{2}} -\Delta_3 -\Delta_1  )} \\ &\times |z_3-z_2|^{2 (\Delta_1+\Delta_{-\frac{\gamma}{2}}-\Delta_3-\Delta_2)} G\left ( \frac{(z-z_1)(z_2-z_3)}{ (z-z_3) (z_2-z_1)}  \right )  \label{confinv}
\end{align}
where $\Delta_{-\frac{\gamma}{2}}= -\frac{\gamma}{4} (Q+\frac{\gamma}{4})$ and $\Delta_k= \frac{\alpha_k}{2} (Q- \frac{\alpha_k}{2})$.
We can recover $C_\gamma$ and $G(z)$ as the following limits
\begin{align}
C_\gamma(\alpha_1,\alpha_2,\alpha_3)&=\lim_{z_3\to\infty} |z_3|^{4 \Delta_3} \langle       V_{\alpha_1}(0)  V_{\alpha_2}(1) V_{\alpha_3}(z_3) \rangle \label{Climit}\\
G(z)&= 
\lim_{ z_3 \to\infty }|z_3|^{4 \Delta_3} \langle    V_{-\frac{\gamma}{2}}  (z)  V_{\alpha_1}(0)  V_{\alpha_2}(1) V_{\alpha_3}(z_3)  \rangle   \label{Glimit}.
\end{align}

In order to state the result, we introduce
 \begin{equation}\label{Fpmdef}
F_{-}(z)= {}_2F_1({ a},{  b},{ c},z), \quad F_{+}(z)= z^{1-{ c}} {}_2F_1(1+{ a}-{ c},1+{ b}-{ c},2-{ c},z)
\end{equation}
where $_2F_1(a,b,c,z)$ are the standard hypergeometric series extended to $\mathbb{C} \setminus ]1,\infty[$ and the real parameters ${ a},{  b},{ c}$ have the following expression
\begin{equation}\label{defab}
{ a}= \frac{\gamma}{2} (\frac{\alpha_1}{2}-\frac{Q}{2}) + \frac{\gamma}{2} (\frac{\alpha_2}{2}+\frac{\alpha_3}{2}-\frac{\gamma}{2})-\frac{1}{2} \quad { b}=\frac{\gamma}{2} (\frac{\alpha_1}{2}-\frac{Q}{2}) + \frac{\gamma}{2} (\frac{\alpha_2}{2}-\frac{\alpha_3}{2})+\frac{1}{2}
\end{equation}
and
\begin{equation}\label{defc}
{ c}= 1+ \frac{\gamma}{2} (\alpha_1-Q).
\end{equation}
As always we assume the Seiberg bounds for $(-\frac{\gamma}{2},\alpha_1,\alpha_2,\alpha_3)$, i.e.  $\sum_{k=1}^3 \alpha_k>2Q+\frac{\gamma}{2}$ and $\alpha_k<Q$ for all $k$. Then:
\begin{theorem}\label{theo4point}
 Let 
 $Q-\frac{1}{\gamma}< \alpha_1< Q-\frac{\gamma}{2}$ and write $G(z)=|z|^{\frac{\gamma \alpha_1}{2}}   |z-1|^{\frac{\gamma \alpha_2}{2}}\tilde G(z)$. Then  
 \begin{equation*}\label{theo4pointexpression}
\tilde G(z)= C_\gamma(\alpha_1-\frac{\gamma}{2}, \alpha_2,\alpha_3)  |F_{-}(z)|^2  - \mu \frac{\pi}{  l(-\frac{\gamma^2}{4}) l(\frac{\gamma \alpha_1}{2})  l(2+\frac{\gamma^2}{4}- \frac{\gamma \alpha_1}{2}) } C_\gamma(\alpha_1+\frac{\gamma}{2}, \alpha_2,\alpha_3)  |F_{+}(z)|^2
\end{equation*}
where $l(x)=\frac{\Gamma (x)}{\Gamma(1-x)}$.

\end{theorem}
\begin{remark}
The set of $(\alpha_1,\alpha_2,\alpha_3)$ that satisfy the assumptions of Theorem \ref{theo4point} is non empty for $\gamma< \sqrt{2}$. 
The lower bound on $\alpha_1$ can be relaxed with more work hence leading to a larger set of $\gamma$ ensuring that the assumptions of Theorem \ref{theo4point} are non empty.  Since this adds additional technicalities  we will adress it in a forthcoming work \cite{KRV}. The upper bound  comes from the necessity of $\alpha_1+\frac{\gamma}{2}$ to satisfy the Seiberg bound $\alpha_1+\frac{\gamma}{2}<Q$: one can also relax this condition to some extent with additional work  \cite{KRV}.   
\end{remark}
From this we can deduce the following corollary on the 3 point structure constants:
\begin{corollary}\label{3pointconstant}
 Let $(\alpha_1,\alpha_2,\alpha_3)$ be such that we have $Q-\frac{1}{\gamma}< \alpha_1< Q-\frac{\gamma}{2}$ and such that $(-\frac{\gamma}{2},\alpha_1,\alpha_2,\alpha_3)$ satisfy the Seiberg bounds. Then we have the following relation
\begin{equation}\label{3pointconstanteq}
\frac{C_\gamma(\alpha_1+\frac{\gamma}{2},\alpha_2,\alpha_3)}{C_\gamma(\alpha_1-\frac{\gamma}{2},\alpha_2,\alpha_3)}  
= - \frac{1}{\pi \mu}\frac{l(-\frac{\gamma^2}{4})  l(\frac{\gamma \alpha_1}{2}) l(\frac{\alpha_1\gamma}{2}  -\frac{\gamma^2}{4})  l(\frac{\gamma}{4} (\bar{\alpha}-2\alpha_1- \frac{\gamma}{2}) )   }{l( \frac{\gamma}{4} (\bar{\alpha}-\frac{\gamma}{2} - 2Q)  ) l( \frac{\gamma}{4} (\bar{\alpha}-2\alpha_3-\frac{\gamma}{2} ))  l( \frac{\gamma}{4} (\bar{\alpha}-2\alpha_2-\frac{\gamma}{2} )) }
\end{equation}
where $\bar{\alpha}= \alpha_1+\alpha_2+\alpha_3$.
\end{corollary}

\section{Properties of Liouville Correlation Functions}
\subsection{Integrability  Properties}

In Section 5, we prove detailed estimates for the Liouville correlations as some of the points get together. For the proof of the Ward identities we need the following special cases. 
Let the weights $\alpha_1,\dots,\alpha_N$ satisfy the Seiberg bounds \eqref{TheSeibergbounds}  and $\z=(z_1,\dots,z_N)\in U_N$. From now on we use the notation
$G_\epsilon(\z)= \langle  \prod_{k=1}^N  V_{\alpha_k ,\epsilon }(z_k)\rangle_{\epsilon}$ and more generally for $\x= x_1, \cdots, x_n \in\C$
\begin{equation}\label{defgeneralG}
G_\epsilon(\x; \z):= \langle \prod_{i=1}^n  V_{\gamma ,\epsilon }(x_i) \prod_{k=1}^N  V_{\alpha_k ,\epsilon }(z_k)\rangle_{\epsilon}.
\end{equation}
This is a slight abuse of notation since we use the same notation $G_\epsilon$ for different functions: these functions depend on the number of variables $n$ but it should be clear from the context how many variables we are considering. Similarly $G(\z)$ and $G(\x; \z)$ stand respectively for the limits of $G_\epsilon(\z)$ and $G_\epsilon(\x; \z)$ as $\epsilon$ goes to $0$. 
In the following we fix $\z$ and study $G_\epsilon(\x; \z)$ as a function of $\x$. We will mainly be interested in the case of $G_\epsilon(\cdot; \z)$ as a function of one variable that we will denote $x$ or of two variables that we will denote $x,y$.  Define also 
 \begin{equation*}\bar G(x; \z):=\sup_\epsilon G_\epsilon(x; \z),\ \ \ \bar G(x,y; \z):=\sup_\epsilon G_\epsilon(x,y; \z)
  \end{equation*}
 \begin{proposition}\label{1and2point} 
 {\rm a)} Let $\epsilon>0$. Then $G_\epsilon(x; \z)$ and $G_\epsilon(x,y; \z)$ are smooth with $\|G_\epsilon(\cdot; \z)\|_\infty<\infty$ and for some $p>2$ and some constant $C_\epsilon>0$
 $$G_\epsilon(x,y; \z)\leq C_\epsilon (1+|x|)^{-p}(1+|y|)^{-p}.
 $$  
\noindent {\rm (b)}  The functions 
$\bar G(x; \z)$ and $\bar G(x,y; \z)$ belong to $L^p(\C)$ and $L^p(\C^2)$  respectively for $p\in [1,r]$ for some $r>1$ and
uniformly in $\z$ on compact subsets of $U_N$. The same holds for the functions
 \begin{equation*}|\ln|x-z_i||^k\bar G(x; \z),\ \  \ | \ln|x-y||^k\bar G(x,y; \z)
  \end{equation*}
   for all $i$ and $k>0$.

\noindent {\rm (c)}  Let $x,y\in 
\C\setminus\{z_1,\dots,z_N\}$. Then 
$$
\bar G(x,y; \z)\leq C|x-y|^{-2+\zeta}
$$
where the constant $C$ depends on $|x-z_i|, |y-z_i|$, $i=1\dots N$.
 \end{proposition}

We will also need the fact that the regularized correlations are translation invariant 
\begin{lemma}\label{trinv} 
For all $y\in\C$, $G_\epsilon(z_1+y,\dots,z_N+y)=G_\epsilon(\z)$ and thus
$
\sum_{i=1}^N\partial_{z_i}G_\epsilon(\z)=0
$.
\end{lemma}
For the proof see Section \ref{regcor}.
\subsection{Integration by Parts Formula and a KPZ identity}\label{ippkpz}
Let $f\in C_0^\infty(\C)$ and set $X(f)=\int_{\C} X(z)f(z)d^2z$. The following identity follows by integration by parts in the Gaussian measure\footnote{Recall that for a centered Gaussian vector $(X,Y_1,\dots, Y_N)$ and a smooth function $f$ on $\R^N$, the Gaussian integration by parts yields $\E[X \,f(Y_1,\dots,Y_N)]=\sum_{k=1}^N\E[XY_k]\E[\partial_{Y_k}f(Y_1,\dots,Y_N)] $.}:
\begin{align} \nonumber
 \langle X(f)  \prod_{k=1}^N  V_{\alpha_k,\epsilon  }(z_k)\rangle_{\epsilon}  =& \sum_{i=1}^N\alpha_i\E(X(f)X_\epsilon(z_i))   \langle  \prod_{k=1}^N  V_{\alpha_k ,\epsilon }(z_k)\rangle_{\epsilon}\\  -&\mu\gamma\int_{\C} \E(X(f)X_\epsilon(x))\langle  V_{\gamma,\epsilon }(x)\prod_{k=1}^N  V_{\alpha_k ,\epsilon }(z_k)\rangle_{\epsilon}d^2x\label{ipp}.
\end{align}
By Proposition \ref{1and2point},  the $x$ integral converges since  $|\E(X(f)X_\epsilon(x))|\leq C\ln(2+|x|)$. We will need to use this formula for $f$ satisfying $\int_{\C} f(z)d^2z=0$. Then 
$$
\E(X(f)X_\epsilon(z))= (C_\epsilon\ast f)(z)-\frac{_1}{^4} \int_{\C} \ln\hat g(u)f(u)d^2u
$$
where $C_\epsilon=\rho_\epsilon\ast \ln|z|^{-1}$. 
  Recalling the definition of the Liouville field \eqref{Liouville field} we then get
\begin{align*} 
 \langle \phi(f)  \prod_{k=1}^N  V_{\alpha_k,\epsilon  }(z_k)\rangle_{\epsilon}  &= \sum_{k=1}^N\alpha_k (C_\epsilon\ast f)(z_k) G_\epsilon( \z)
 -\mu\gamma\int_{\C} (C_\epsilon\ast f)(x)G_\epsilon(x; \z)
 d^2x\\&
 +\frac{1}{4}\int_{\C}\ln\hat g(u)f(u)d^2u\big((2Q-\sum_{k=1}^N\alpha_k)G_\epsilon( \z)+\mu\gamma\int_{\C} G_\epsilon(x; \z)d^2x\big) 
\end{align*}
The metric dependent term actually vanishes due to the following remarkable identity
\begin{lemma}\label{kpzid} For all $\epsilon\geq 0$ (KPZ-identity)
\begin{equation} 
\mu\gamma\int_{\C} 
G_\epsilon(x; \z)d^2x = (\sum_{k=1}^N \alpha_k-2Q)G_\epsilon( \z).
\label{kpzid}
\end{equation}
\end{lemma}

\proof Recalling the $c$-dependence in \eqref{Vdefi} we get by a simple change of variables $\gamma^{-1}\ln \mu+c=c'$ that
\begin{align*} 
  \langle  \prod_{k=1}^N  V_{\alpha_k,\epsilon  }(z_k)\rangle_{\epsilon,\mu}  = &\int_{\R}   e^{-2Q c}  \:   \E \left [   \prod_{k=1}^N  V_{\alpha_k, \epsilon}(z_k)        e^{- \mu   \int_{\C}V_{\gamma,\epsilon}(x)\,d^2x    } \right ]   \: dc\\
  = &\mu^{-\frac{\sum_{k=1}^N\alpha_k-2Q}{\gamma}} \int_{\R}   e^{-2Q c'}  \:   \E \left [   \prod_{k=1}^N  V_{\alpha_k, \epsilon}(z_k)        e^{-   \int_{\C}V_{\gamma,\epsilon}(x)\,d^2x    } \right ]   \: dc'.
\end{align*}
The identity follows by differentiating in $\mu$. The limit as $\epsilon\to 0$ follows in virtue of Proposition \ref{1and2point}.
\qed

\medskip
Now, we give some definitions which will be used throughout the rest of the paper. Recall that $\rho_\epsilon(x)= \frac{1}{\epsilon^2} \rho (\frac{\bar{x}x}{\epsilon^2}) $ where $\rho$ is $C^{\infty}$ non-negative with compact support in $[0,\infty[$ and such that $\pi \int_0^\infty \rho(t) dt=1$.Then we set 
for $\epsilon, \epsilon' \geq 0$ \begin{equation}\label{defregula1}
\frac{1}{(z)}_{\epsilon,\epsilon'}:=\rho_\epsilon\ast\rho_{\epsilon'}\ast \frac{1}{z}
\end{equation}
and more generally for $n \geq 0$
\begin{equation}\label{defregula2}
\frac{1}{(z)_{\epsilon,\epsilon'}^{n+1}}:=  \partial_{z}^{n} \:   \frac{1}{(z)}_{\epsilon,\epsilon'} 
\end{equation}
If $\epsilon'=0$, we simply set $\frac{1}{(z)}_{\epsilon}: = \frac{1}{(z)}_{\epsilon,0}$ and similarly for $\frac{1}{(z)_{\epsilon}^{n+1}}$.

We then get the integration by parts formula that will be used repeatedly in this paper:
\begin{lemma}\label{ippid} For all $\epsilon\geq 0$ and $\epsilon'>0$ and $n\geq 0$
\begin{align} 
 \langle \partial_z^{n+1}\phi_{\epsilon'}(z)  \prod_{k=1}^N  V_{\alpha_k,\epsilon  }(z_k)\rangle_{\epsilon}  &=-\frac{1}{2} \sum_{i=1}^N\alpha_i \partial_z^{n}\frac{1}{(z-z_i)}_{\epsilon,\epsilon'}G_\epsilon( \z) 
 +\frac{\mu\gamma}{2}\int_{\C}  \partial_z^{n}\frac{1}{(z-x)}_{\epsilon,\epsilon'}G_\epsilon(x; \z)
 d^2x.\label{ipp1}
\end{align}
\end{lemma}

As an application, let us compute the derivative of the regularized correlations:
\begin{equation}
 \partial_{z_i}G_\epsilon( \z) =
  \alpha_i\langle \partial_{z_i} \phi_\epsilon(z_i)\prod_{k=1}^N  V_{\alpha_k,\epsilon  }(z_k)\rangle_{\epsilon}=-\frac{1}{2}\sum_{j; j\neq i}^N \frac{\alpha_i \alpha_j}{(z_i-z_j)_{\epsilon,\epsilon}}G_\epsilon( \z) 
  + \frac{\alpha_i\mu\gamma}{2} Y_{i,\epsilon}(\z) \label{derivativecorr1}
\end{equation}
with
\begin{equation}
 Y_{i,\epsilon}(\z) =\int_\C\frac{1}{(z_i-x)_{\epsilon,\epsilon}}G_\epsilon(x; \z)d^2x.
 \label{Ydefi}
\end{equation}
\vskip 2mm
\noindent{\bf Remark}. {\it Note that  Proposition \ref{1and2point} does not allow us to control the limit as  $\epsilon$ tends to zero of expression \eqref{Ydefi} since it guarantees only integrability of a logarithmic singularity in $z_i-x$. Indeed, depending on $\gamma$ and $\alpha_i$ the singularity of correlation function $G_\epsilon(x; \z)$ is $|z_i-x|^{-2+\zeta}$ with $\zeta$ arbitrarily close to $0$. Hence the integral in \eqref{Ydefi} is not always absolutely convergent and its analysis is rather subtle. }

\subsection{Differentiability of  Correlation Functions}

This Section is devoted to the proof of Theorem \ref{wardth} (a).
We start with a convenient representation of the first derivative $ \partial_{z_i}G_\epsilon( \z)$; indeed, as mentioned in the previous remark, the $Y_{i,\epsilon}(\z) $ term in \eqref{derivativecorr1} is difficult to study directly hence we give an expression for the $Y_{i,\epsilon}(\z) $ term which is easier to study. Indeed, we get the following lemma:

\begin{lemma}\label{1stder} 
We can write $\partial_{z_i}G_\epsilon (\z)$ under the following form
\begin{equation} 
\partial_{z_i}G_\epsilon (\z)=
-\frac{1}{2}\sum_{j;j\neq i}^N \frac{\alpha_i \alpha_j}{(z_i-z_j)_{\epsilon,\epsilon}}G_\epsilon (\z)+\int_{\C} f_{i,\epsilon}(x,\z)G_\epsilon (x;\z)dx+ \int_{\C} \int_{\C} F_{\epsilon}(x,y,z_i)G_\epsilon (x,y;\z)d^2x d^2y
\end{equation}
where $f_{i,\epsilon}$ and $F_{\epsilon}$ are bounded smooth functions which converge in $C^1$ as $\epsilon\to 0$.  Moreover
\begin{equation} 
\|\partial^nf_{i,\epsilon}(\cdot, \z)\|_\infty, \|\partial^nF_{\epsilon} (\cdot, z_i)\|_\infty\leq C\delta_i^{-n-1}\label{deribnd}
\end{equation}
for  $n=0,1$
uniformly in $\epsilon$ where  $\delta_i:=\min_{j; j\neq i}|z_i-z_j|$ and $C$ is some global constant (which depends only on $\mu,\gamma$ and $\alpha_i$). 
\end{lemma}
%
%
%

\begin{proof}
The key idea behind this proof is to use an exact identity (i.e. identity \eqref{Theexactidt} below) to deduce a new and easier to study expression for $Y_{i,\epsilon}(\z) $. Recall that $x \mapsto G_\epsilon (x;\z)$ is smooth. Hence ($\pi$ times) the Beurling transform
\begin{align*}
A_{\epsilon}(z;\z)=-\int_{\C} \frac{1}{(z-x)^2}G_\epsilon (x;\z)d^2x:=-\lim_{\epsilon' \to 0}\int_{\C} \frac{1}{(z-x)^2}1_{|z-x|\geq\epsilon'}G_\epsilon (x;\z)d^2x
\end{align*}
is defined and satisfies the following key identity
\begin{equation}\label{Theexactidt}
A_{\epsilon}(z;\z)=
\int_{\C} \frac{1}{z-x} \partial_x G_\epsilon (x;\z)d^2x=B_{\epsilon}(z;\z)+C_{\epsilon}(z;\z)    
\end{equation}
where  by \eqref{derivativecorr1}
 \begin{align*}
 B_{\epsilon}(z;\z)&=  \frac{\gamma}{2} \sum_{j=1}^N \alpha_j  \int_{\C} \frac{1}{z-x} \frac{1}{(z_j-x)_{\epsilon,\epsilon}} G_\epsilon (x;\z)
\,dx   \\
C_{\epsilon}(z;\z)&=-
\frac{\mu\gamma^2}{2}
 \int_{\C} \frac{1}{z-x}   \int_{\C} \frac{1}{(x-y)_{\epsilon,\epsilon}} 
G_\epsilon (x,y;\z) d^2x d^2y.
\end{align*}
Using  $\frac{1}{z-x}=\frac{1}{z-z_j}\Big(\frac{x-z_j}{z-x}+1\Big) $ we write
\begin{equation}\label{decompB}
 \int_{\C} \frac{1}{z-x} \frac{1}{(z_j-x)_{\epsilon,\epsilon}} G_\epsilon (x;\z)dx=
 \frac{1}{z-z_j} (Y_{j,\epsilon}(\z)+D_{j,\epsilon}(z;\z))
\end{equation}
where 
\begin{align*}
D_{j,\epsilon}(z;\z) =&-  \int_{\C} \frac{z_j-x}{z-x}  \frac{1}{(z_j-x)_{\epsilon,\epsilon}}G_\epsilon (x;\z) \,d^2x.
\end{align*}

We want to deduce from expression \eqref{Theexactidt} (valid if $z \neq z_j$ for all $j$) an expression for $Y_{i,\epsilon}(\z)$. A natural way to do so is to integrate expression \eqref{Theexactidt} with respect to a small contour  around $z_i$ to get rid of the  $Y_{j,\epsilon}(\z)$ terms for $j \neq i$. However, contour integrals lack smoothness so we will use an equivalent but more smooth procedure which  boils down to taking an average over contour integrals of  \eqref{Theexactidt}. 
More precisely, let  $\theta(x)$ be a smooth bump with support on $]\frac{\delta_i}{4},\frac{\delta_i}{2}[$  and $\chi(z)=\theta(|z|)\frac{z}{|z|}$. So $\chi$ is supported inside the annulus around the origin with radii $\frac{\delta_i}{4}$ and $ \frac{\delta_i}{2}$. We normalize it so that  $\frac{\gamma \alpha_i}{2} \int_{\C}  \chi(z)\frac{1}{z}d^2z=1$. Then, for all continuous functions $f$
\begin{equation*}
\int_{\C} \chi(z) f(z) d^2z= \frac{1}{\sqrt{-1} }\int_{\frac{\delta_i}{4}}^{\frac{\delta_i}{2}}  \theta(r)\: \left (\oint_{|z|=r}   f(z)dz \right ) dr.
\end{equation*}
Therefore, for all $j \neq i$ we have $\int_{\C} \frac{\chi(z)}{z-z_j} d^2z=0$ and hence we get 
 $$
 \int_{\C} \chi(z-z_i)B_{\epsilon}(z,\z)d^2z=Y_{i,\epsilon}(\z)+\frac{\gamma}{2} \int_{\C} \chi(z-z_i) \Big ( \sum_{j=1}^N  \frac{\alpha_j}{z-z_j} D_{j,\epsilon}(z,\z) \Big ) d^2z.
 $$
 Using the key identity \eqref{Theexactidt}, this leads by Fubini (applied to the measure $d^2x d^2z$) to 
\begin{equation}\label{Ysol}
 Y_{i,\epsilon}(\z)=\int_{\C}  f_{i,\epsilon}(x;\z) G_\epsilon (x;\z) \,d^2x+ \int_{\C}  \int_{\C}  F_{\epsilon}(x,y,z_i) G_\epsilon (x,y;\z) \,d^2x d^2y 
\end{equation}
where
\begin{equation}\label{frep}
 f_{i,\epsilon}(x;\z)=\int_{\C}\chi(z-z_i)(-\frac{1}{(z-x)^2}+  \frac{\gamma}{2} \sum_{j=1}^N  \alpha_j \frac{1}{z-z_j} \frac{z_j-x}{z-x}  \frac{1}{(z_j-x)_{\epsilon,\epsilon}})d^2z
\end{equation}
and 
$ F_{\epsilon}$ is given after symmetrizing by 
\begin{equation}\label{Frep}
 F_{\epsilon}(x,y,z_i)= -
\frac{\mu\gamma^2}{4}
  (x-y) \frac{1}{(x-y)_{\epsilon,\epsilon}} \int_{\C}  \chi(z-z_i) \frac{1 }{z-x} \frac{1 }{z-y}d^2z=C \cdot(x-y) \frac{1}{(x-y)_{\epsilon, \epsilon}} \int_{|x-z_i|}^{|y-z_i|} \theta(r)dr
\end{equation}
for $|x-z_i|\leq|y-z_i|$. The last  integral is $\caO(|x-y|)$. Hence $ F_{\epsilon}(x,y,z_i)$ converges in $C^1$ and satisfies the bounds \eqref{deribnd} .
Now, we take care of $f_{i,\epsilon}$. The integral 
$$
\int_{\C}  \chi(z-z_i)\frac{1}{z-z_j} \frac{1}{z-x} d^2z=\frac{1}{\sqrt{-1}}\int_{\frac{\delta_i}{4}}^{\frac{\delta_i}{2}}  \theta(r) \left (\oint_{|z|=r}\frac{1}{z-(z_j-z_i)} \frac{1}{z-(x-z_i)}dz \right ) dr
$$
vanishes if $|z_j-x|<\frac{\delta_i}{4}$ (for all $j$) by Cauchy. Indeed, if $j \neq i $ and  $|z_j-x|<\frac{\delta_i}{4}$ then for all $r \in ]\frac{\delta_i}{4},\frac{\delta_i}{2}[$ the function $z \mapsto \frac{1}{z-(z_j-z_i)} \frac{1}{z-(x-z_i)}$ has no poles inside the circle of radius $r$ and if $|z_i-x|<\frac{\delta_i}{4}$ then for all $r \in ]\frac{\delta_i}{4},\frac{\delta_i}{2}[$ the function $z \mapsto \frac{1}{z} \frac{1}{z-(x-z_i)}$ has two poles which compensate in the residue formula for the contour integral $\oint_{|z|=r}dz$ (except for the case $x=z_i$: in this case, it is obvious that  for all $r \in ]\frac{\delta_i}{4},\frac{\delta_i}{2}[$ we have $\oint_{|z|=r}   \frac{1}{z^2}dz=0$).

From the above considerations, we can conclude that $ f_{i,\epsilon}(x;\z)$ converges together with all its derivatives. Since  $ \frac{1}{2} \gamma \alpha_i\int_{\C}  \chi(z)\frac{1}{z}d^2z=1$ we may assume 
$$
\|\theta\|_\infty\leq C\delta_i^{-1},\ \ \ \|\partial\theta\|_\infty\leq C\delta_i^{-2}.
$$
Hence the bounds \eqref{deribnd} hold for $f_{i,\epsilon}$.
%
\end{proof}
This lemma leads to the following corollary which gives a control of the $\partial_{z_i}$ derivative with respect to $\delta_i=\min_{j; j\neq i}|z_i-z_j|$:

\begin{corollary}\label{derest} $G_\epsilon (\z)$ is $C^1$ in $U_N$ and for all
$\epsilon\geq 0$
$$
|\partial_{z_i}G_\epsilon (\z)|\leq C(\min_{j; j\neq i}|z_i-z_j|)^{-1}G_\epsilon (\z).
$$
\end{corollary}
\begin{proof}
We use the expression for the derivative  given by Lemma \ref{1stder}. By the KPZ-identity \eqref{kpzid} we get
\begin{equation} 
|\int_{\C} f_{i,\epsilon}(x;\z)G_\epsilon (x;\z)d^2x |\leq C\delta_i^{-1}\int_{\C} G_\epsilon (x;\z)d^2x=C\delta_i^{-1}G_\epsilon (\z).
\end{equation}
and similarly for the $F_{\epsilon}$ term, completing the proof.
\end{proof}

Define
\begin{equation} 
G(f;\z):=\int_{\C} f(x)G (x;\z)d^2x,\ \ \ G(F;\z):=\int_{\C} \int_{\C} F(x,y)G (x,y;\z)d^2xd^2y\label{phiins}
\end{equation}
and $G_\epsilon(f;\z)$ and $G_\epsilon(F;\z)$ similarly.  Note that by the KPZ-identity \eqref{kpzid} these are well defined if $f,F$ are bounded:
$$|G(f;\z)|\leq C\|f\|_\infty G(\z),\ \ \ |G(F;\z)|\leq C\|F\|_\infty G(\z)
$$ 
The following Lemma will be used repeatedly in what follows:

\begin{lemma}\label{derconv} Let $f,F$ be bounded, $C^1$ and with bounded derivatives. Then
$G(f;\z)$ and $G(F;\z)$ are $C^1$ on $U_N$ and
\begin{align*}
\partial_{z_i}G(f;\z)&=\int_{\C} f(x)\partial_{z_i}G (x;\z)d^2x=\lim_{\epsilon\to 0}\int_{\C} f(x)\partial_{z_i}G_\epsilon (x;\z)d^2x\\
\partial_{z_i}G(F;\z)&=\int_{\C} \int_{\C} F(x,y)\partial_{z_i}G (x,y;\z)d^2x d^2y=
\lim_{\epsilon\to 0}\int_{\C} \int_{\C}  F(x,y)\partial_{z_i}G_\epsilon (x,y;\z)d^2xd^2y,
\end{align*}
where the notations $f \mapsto \int_{\C} f(x)\partial_{z_i}G (x;\z)d^2x$ and  $F \mapsto \int_{\C} \int_{\C} F(x,y)\partial_{z_i}G (x,y;\z)d^2xd^2y$ stand for continuous linear functionals on $C^1$. 
Furthermore if $f_n\to f$
and $F_n\to F$ together with their derivatives uniformly on $\C$ and $\C^2$  then $\partial_{z_i}G(f_n;\z)$ and $\partial_{z_i}G(F_n;\z)$ converge to $\partial_{z_i}G(f;\z)$ and $\partial_{z_i}G(F;\z)$.
\end{lemma}


\begin{remark}
The functions $x \mapsto \partial_{z_i}G (x;\z)$ and $(x,y)  \mapsto G (x,y;\z)$ do not necessarily belong to $L^1$ hence one can not necessarily consider the integrals $\int_{\C} f(x)\partial_{z_i}G (x;\z)d^2x$ and $\int_{\C} \int_{\C} F(x,y)\partial_{z_i}G (x,y;\z)d^2xd^2y$ in the classical sense. This is actually why the analysis of $G(f;\z)$ and $G(F;\z)$ is delicate and non trivial.     
\end{remark}
\begin{proof}Let $\rho_i$ be a smooth bump supported in a small enough ball $B_i$ around $z_i$  and $\rho_i=1$ in a neighborhood of $z_i$: typically, $\rho_i(x)=\tilde{\rho}(x-z_i)$ where  $\tilde{\rho}$ is a smooth bump supported in a small neighborhood of origin with $\tilde{\rho}(x)=1$ for $|x|$ small. We decompose
\begin{equation}\label{decompose}
\int_{\C} f(x)\partial_{z_i}G_\epsilon (x;\z)d^2x=\int_{\C} f(x)\partial_{z_i}G_\epsilon (x;\z)\rho_i(x)d^2x+\int_{\C} f(x)\partial_{z_i}G_\epsilon (x;\z)(1-\rho_i(x))d^2x.
\end{equation}
By corollary \ref{derest} (applied to $(x;\z)$), we have for all $x$
\begin{equation}\label{boundwithy}
|\partial_{z_i}G_\epsilon (x;\z) |\leq C (\min (|x-z_i|, \delta_i))^{-1}G_\epsilon (x;\z)
\end{equation}
where $\delta_i= \min_{j; j\neq i}|z_i-z_j|$. By Proposition \ref{1and2point}  the function $ x \mapsto G_\epsilon (x;\z)$ is dominated by an $L^1$ function hence so is $\partial_{z_i}G_\epsilon (x;\z)(1-\rho_i(x))$ by inequality \eqref{boundwithy}. Therefore, by the dominated convergence theorem, the integral $\int_{\C} f(x)\partial_{z_i}G_\epsilon (x;\z)(1-\rho_i(x))d^2x$ converges to $\int_{\C} f(x)\partial_{z_i}G (x;\z)(1-\rho_i(x))d^2x$ as $\epsilon$ goes to $0$.

\noindent
We now consider the first term on the right hand side of \eqref{decompose}.  By Lemma \ref{trinv} (translation invariance), we get
\begin{align}\label{decamp}
\partial_{z_i}G_\epsilon(x;\z)&=-\sum_{j; j\neq i}^N\partial_{z_j}G_\epsilon(x;\z)-\partial_{x}G_\epsilon(x;\z)
\end{align}
hence the first term becomes
\begin{align*}
\int_{\C}  f(x)\partial_{z_i}G_\epsilon (x;\z)\rho_i(x)d^2x=\int_{\C} \partial_x (f(x) \rho_i(x)  )G_\epsilon (x;\z) d^2x-\sum_{j; j\neq i}^N\int_{\C} f(x) \partial_{z_j}G_\epsilon (x;\z)\rho_i(x)d^2x.
\end{align*}
The first term on the right hand side of the above equality converges (by using Proposition \ref{1and2point} and the dominated convergence theorem) hence we focus on the $\sum_{j; j\neq i}$ term.
We may now use Lemma \ref{1stder} for each $j \neq i$ 
\begin{align}
\partial_{z_j}G_\epsilon (x;\z)&=-\frac{\alpha_j}{2}(\sum_{k; k\neq j}^N\frac{\alpha_k }{(z_j-z_k)_{\epsilon,\epsilon}}+\frac{\gamma }{(z_j-x)_{\epsilon,\epsilon}} )G_\epsilon (x;\z) \nonumber\\
&+\int_{\C} \bar{f}_{j,\epsilon}(x,y,\z)G_\epsilon (x,y;\z)d^2y+ \int_{\C} \int_{\C} \bar{F}_{j,\epsilon}(x, y,y',\z)G_\epsilon (x,y,y';\z)d^2yd^2y'  \label{usefulldecomp} 
\end{align}
where the $C^1$ norms of the functions $y \mapsto \bar{f}_{j,\epsilon}(x,y,\z)$ and $(y,y') \mapsto \bar{F}_{j,\epsilon}(x,y,y',\z)$ depend on $\min(\delta_j,|x-z_j|)$ with $\delta_j=\min_{k; k \neq j}|z_j-z_k|$. There exists $\delta>0$ such that for $x\in B_i$ we have $\min(\delta_j,|x-z_j|) \geq \delta$ and therefore using the KPZ identity \eqref{kpzid} we get from decomposition \eqref{usefulldecomp} that
\begin{equation*}
\rho_i(x) \: |\partial_{z_j}G_\epsilon (x;\z)  |  \leq C  \rho_i(x) \: G_\epsilon (x;\z),
\end{equation*}
for some  constant $C>0$ (which depends on $B_i$ and $\z$). We then get convergence of $\int_{\C}  f(x)\partial_{z_i}G_\epsilon (x;\z)\rho_i(x)d^2x$ again by Proposition \ref{1and2point}  along with the dominated convergence theorem.

%



\vspace{0.3 cm}

Consider the second claim. We localize the $x,y$-integrals again by inserting $1=\rho_{i}+(1-\rho_{i})$. 
There are three cases to consider. 
\vskip 2mm

\noindent (a) The first case is
\begin{align*}
 \int_{\C} \int_{\C} F(x,y)\partial_{z_i}G_\epsilon (x,y;\z)\rho_{i}(x)\rho_{i}(y)d^2xd^2y&= \int_{\C} \int_{\C} G_\epsilon (x,y;\z)(\partial_{x}+\partial_{y})(F(x,y)\rho_{i}(x)\rho_{i}(y))d^2xd^2y\\&-\sum_{j; j\neq i}^N \int_{\C} \int_{\C} F(x,y)\partial_{z_j}G_\epsilon (x,y;\z)\rho_{i}(x)\rho_{i}(y)dxdy
\end{align*}
where we used translation invariance.
The terms with $\partial_{z_j}G_\epsilon$ can be treated in a similar way than the terms $\int_{\C} f(x)\partial_{z_j}G_\epsilon (x;\z)\rho_i(x)d^2x$ in the previous proof by combining Lemma  \ref{1stder} and Proposition \ref{1and2point} .
\vskip 2mm

\noindent (b)
The second case with insertion $(1-\rho_{i}(x))(1-\rho_{i}(y))$ can also be dealt with in a similar way as the $\int_{\C} f(x)\partial_{z_i}G_\epsilon (x;\z)(1-\rho_i(x))d^2x$ term in the previous proof by using Proposition \ref{1and2point}.
 \vskip 2mm

\noindent (c)
The third case is
\begin{align*}
 \int_{\C} \int_{\C} F(x,y)\partial_{z_i}G_\epsilon (x,y;\z)\rho_{i}(x)(1-\rho_{i}(y))d^2xd^2y.
 \end{align*}
Write $F(x,y)=F(z_i,y)+F(x,y)-F(z_i,y)$. Then by the KPZ identity \eqref{kpzid}
\begin{align*}
 \int_{\C} \int_{\C} F(z_i,y)\partial_{z_i}&G_\epsilon (x,y;\z)\rho_{i}(x)(1-\rho_{i}(y))d^2xd^2y=C\int_{\C}  F(z_i,y)\partial_{z_i}G_\epsilon (y;\z)(1-\rho_{i}(y)))d^2y\\&- \int_{\C} \int_{\C} F(z_i,y)\partial_{z_i}G_\epsilon (x,y;\z)(1-\rho_{i}(x))(1-\rho_{i}(y))d^2xd^2y
 \end{align*}
so we are left with estimating
\begin{align}\label{final}
 \int_{\C} \int_{\C} (F(x,y)-F(z_i,y))\partial_{z_i}G_\epsilon (x,y;\z)\rho_{i}(x)(1-\rho_{i}(y))d^2xd^2y.
 \end{align}
Since $|F(x,y)-F(z_i,y)|\leq C|z_i-x|$ and by Corollary \ref{derest} 
$$
|\partial_{z_i}G_\epsilon (x,y;\z)|\leq C|z_i-x|^{-1}G_\epsilon (x,y;\z)
$$
for some constant (depending on $B_i$ and $\z$) we obtain
 \begin{equation*}
 |F(x,y)-F(z_i,y)||\partial_{z_i}G_\epsilon (x,y;\z) |\rho_{i}(x)(1-\rho_{i}(y))  \leq C G_\epsilon (x,y;\z) \rho_{i}(x)(1-\rho_{i}(y).
 \end{equation*}
 We then conclude by using Proposition \ref{1and2point}. 
  \end{proof}
  By combining Lemmas  \ref{1stder}  and \ref{derconv} and one can easily show
{\begin{corollary}
 $G  (\z)$ is $C^2$ in $U_N$.
\end{corollary}}
\noindent{\it Proof of  Theorem \ref{wardth} (a).}
 We have
\begin{align}
\partial_{z_i}\partial_{z_j}G_\epsilon (\z)&=\partial_{z_i}(-\frac{1}{2}\sum_{k; k \neq j}^N \frac{\alpha_j \alpha_k}{(z_j-z_k)_{\epsilon,\epsilon}}G_\epsilon (\z))+\int_{\C} \partial_{z_i}f_{j,\epsilon}(x;\z)G_\epsilon (x;\z)d^2x+ \int_{\C} \int_{\C} \partial_{z_i}F_{\epsilon}(x,y,{z_j})G_\epsilon (x,y;\z)d^2xd^2y\nonumber\\&+\int_{\C} f_{j,\epsilon}(x;\z)\partial_{z_i}G_\epsilon (x;\z)d^2x+ \int_{\C} \int_{\C} F_{\epsilon}(x,y,{z_j})\partial_{z_i}G_\epsilon (x,y;\z)d^2xd^2y.\label{2ndder}
\end{align}
Convergence of the last two terms follows from Lemma \ref{derconv}  and  the first three from Lemma \ref{1stder}.
\qed

\vspace{0.2 cm}

Now, we set notations which will be used in the sequel of the paper. As a consequence of the previous considerations, we have shown that $ Y_{i,\epsilon}(\z)$ given by \eqref{Ydefi} converges as $\epsilon$ goes to $0$ and we will naturally set the following notation
\begin{equation}\label{TheYtermnotation}  
\int_\C\frac{1}{(z_i-x)}G(x; \z)d^2x:= \underset{\epsilon \to 0}{\lim} \int_\C\frac{1}{(z_i-x)_{\epsilon,\epsilon}}G_\epsilon(x; \z)d^2x
\end{equation}

We have also shown that $\langle \prod_{k=1}^N  V_{\alpha_k} (z_k)\rangle$ is $C^1$ and the $\partial_{z_i}$ derivative has the following expression
\begin{equation}\label{notationderivsum}
 \partial_{z_i} \langle \prod_{k=1}^N  V_{\alpha_k} (z_k)\rangle =-\frac{1}{2}\sum_{j;j\neq i}^N\frac{\alpha_i \alpha_j}{z_i-z_j}G( \z) 
  + \frac{\alpha_i\mu\gamma}{2} \int_\C\frac{1}{z_i-x}G(x; \z)d^2x.
\end{equation}

\section{Proof of the Ward and BPZ Identities}

\subsection{Proof of the first Ward Identity}

We prove the first statement of Theorem \ref{wardth} (b). Recall that $\frac{1}{(z)_\epsilon}$ denotes $\rho_\epsilon\ast \frac{1}{z}$. 
Using the integration by parts formula \eqref{ipp1}  we get
 \begin{align*}
 \langle\partial_{z}^2\phi_\epsilon(z) \prod_{k=1}^N V_{\alpha_k}(z_k)   \rangle&=-
\frac{1}{2}\sum_{i=1}^N\alpha_i \partial_{z} \frac{1}{(z-z_i)_\epsilon}G(\z)+\frac{\mu\gamma}{2}
   \int_{\C} \partial_{z} \frac{1}{(z-x)_\epsilon} G(x;\z)d^2x.
 \end{align*}
Along the same lines as the proof of Lemma \ref{ippid} (using  Gaussian integration by parts) one can 
prove
 \begin{align*}
  \langle( (\partial_z \phi_{\epsilon}(z))^2-& \E[ (\partial_z X_{\epsilon}(z))^2  ]) \prod_{k=1}^N V_{\alpha_k}(z_k)   \rangle
=
\frac{1}{4} \sum_{i,j=1}^N \alpha_i \alpha_j \frac{1}{(z-z_i)_\epsilon} \frac{1}{(z-z_j)_\epsilon}G(\z)
 \\& -\frac{\mu \gamma }{2}     \sum_{i=1}^N \alpha_i  \frac{1}{(z-z_i)_\epsilon} \int_{\C} \frac{1}{(z-x)_\epsilon}G(x;\z)d^2x -\frac{\mu \gamma^2  }{4}      \int_{\C}  (  \frac{1}{(z-x)_\epsilon})^2G(x;\z)d^2x\\&+\frac{\mu^2 \gamma^2}{4}  \int_{\C} \int_{\C} \frac{1}{(z-x)_\epsilon} \frac{1}{(z-y)_\epsilon}G(x,y;\z)d^2x d^2y.
 \end{align*}
 Combining we get (recall \eqref{SETdefi})
\begin{align}
\langle T_\epsilon(z)\prod_{k=1}^N V_{\alpha_k}(z_k)\rangle&=
(-\frac{Q}{2}\sum_{i=1}^N\alpha_i \partial_z \frac{1}{(z-z_i)_\epsilon} -\frac{1}{4} \sum_{i,j} \alpha_i \alpha_j \frac{1}{(z-z_i)_\epsilon} \frac{1}{(z-z_j)_\epsilon})G(\z)
\nonumber \\& + \frac{\mu \gamma}{2}     \sum_{i=1}^N \alpha_i  \frac{1}{(z-z_i)_\epsilon} \int_{\C} \frac{1}{(z-x)_\epsilon}G(x;\z)d^2x +R_\epsilon(z;\z)\label{tcontra1}
 \end{align}
 with
  \begin{align}
R_\epsilon(z;\z)&:=  \frac{\mu\gamma}{2} \int_{\C} (Q \partial_{z} \frac{1}{(z-x)_\epsilon} + \frac{\gamma}{2}(  \frac{1}{(z-x)_\epsilon})^2)G(x;\z)d^2x   
 -\frac{\mu^2 \gamma^2}{4} \int_{\C} \int_{\C} \frac{1}{(z-x)_\epsilon} \frac{1}{(z-y)_\epsilon}G(x,y;\z)d^2x d^2y.\label{tcontra2}
 \end{align} 
Integrating $\partial_z=-\partial_x$ by parts we have
 \begin{equation*}
 \int_{\C} \partial_{z} \frac{1}{(z-x)_\epsilon} G(x;\z)d^2x =-\frac{\gamma}{2}\sum_{i=1}^N \alpha_i\int_{\C}  \frac{1}{(z-x)_\epsilon}  \frac{1}{x-z_i} G(x;\z)d^2x
 + \frac{\mu \gamma^2}{2} \int_{\C} \int_{\C} \frac{1}{(z-x)_\epsilon}\frac{1}{x-y}G(x,y;\z)d^2x d^2y
 \end{equation*} 
where  the last integral is defined as the limit  
$$\lim_{\epsilon\to 0}\int_{\C} \int_{\C} \frac{1}{(z-x)_\epsilon}\frac{1}{(x-y)_{\epsilon'}}G_{\epsilon'}(x,y;\z)d^2x d^2y.
$$  By exploiting the fact that the function $(x,y) \mapsto G_{\epsilon'}(x,y;\z)$ is symmetric we get 
\begin{equation*}
\frac{\mu \gamma^2}{2}  \int_{\C} \int_{\C} \frac{1}{(z-x)_\epsilon}\frac{1}{x-y}G(x,y;\z)d^2x d^2y= \frac{\mu \gamma^2}{4}  \int_{\C} \int_{\C} (\frac{1}{(z-x)_\epsilon}- \frac{1}{(z-y)_\epsilon})\frac{1}{x-y}G(x,y;\z)d^2x d^2y
\end{equation*}
 Therefore, we get
  \begin{align*}
 \int_{\C} \partial_{z} \frac{1}{(z-x)_\epsilon} G(x;\z)d^2x&=-\frac{\gamma}{2} \sum_{i=1}^N \alpha_i \int_{\C}  \frac{1}{(z-x)_\epsilon}  \frac{1}{x-z_i} G(x;\z)d^2x
 +\frac{\mu \gamma^2}{2}  \int_{\C} \int_{\C} \frac{1}{(z-x)_\epsilon}\frac{1}{x-y}G(x,y;\z)d^2x d^2y\\
 &=-\frac{\gamma}{2}\sum_{i=1}^N \alpha_i \frac{1}{z-z_i} \left ( \int_{\C}  \frac{1}{(z-x)_\epsilon} G(x;\z)d^2x+ \int_{\C} h_\epsilon(z-x) \frac{1}{x-z_i} G(x;\z)d^2x \right )\\
 & + \frac{\mu \gamma^2}{4}  \int_{\C} \int_{\C} (\frac{1}{(z-x)_\epsilon}- \frac{1}{(z-y)_\epsilon})\frac{1}{x-y}G(x,y;\z)d^2x d^2y
\end{align*} 
where   $ h_\epsilon(x)=x\frac{1}{(x)_\epsilon} $. Thus recalling that $Q=\frac{\gamma}{2}+\frac{2}{\gamma}$ we arrive at
 \begin{align}
 R_\epsilon(z;\z)&= -\frac{\mu \gamma}{2}\sum_{j=1}^N \alpha_j \frac{1}{(z-z_j)} \left ( \int_{\C}  \frac{1}{(z-x)_\epsilon} G(x;\z)d^2x+ \int_{\C} h_\epsilon(z-x) \frac{1}{x-z_j} G(x;\z)d^2x \right ) \nonumber\\
 &+ \frac{\mu\gamma^2}{4}\int_{\C} ( \partial_{z} \frac{1}{(z-x)_\epsilon} + (  \frac{1}{(z-x)_\epsilon})^2)G(x;\z)d^2x\nonumber\\
 &
  -\frac{\mu^2 \gamma^2}{4}  \int_{\C} \int_{\C} ( \frac{1}{(z-x)_\epsilon} \frac{1}{(z-y)_\epsilon}-(\frac{1}{(z-x)_\epsilon}- \frac{1}{(z-y)_\epsilon})\frac{1}{x-y})G(x,y;\z)d^2x d^2y\label{Rfinal}. 
 \end{align}
 In fact one can show that the last two terms in the right-hand side of the above expression converge to $0$ as $\epsilon\to 0$. In the proof of Lemma \ref{dT'} below, we will prove a more technical result involving the derivatives of this term. We leave as an exercise for the reader to adapt the proof and show this convergence to $0$.

Finally we have
 \begin{align}
  \int_{\C} h_\epsilon(z-x) \frac{1}{x-z_j} G(x;\z)d^2x=  -h_\epsilon(z-z_j)Y_j(\z)
  +  \int_{\C} (h_\epsilon(z-x)-h_\epsilon(z-z_j)) \frac{1}{x-z_j} G(x;\z)d^2x\to -Y_j(\z)\label{yjterm}
  \end{align}
 since $ h_\epsilon(x)\to 1$ for $x\neq 0$ and since $(h_\epsilon(z-x)-h_\epsilon(z-z_j)) \frac{1}{x-z_j}$ is bounded, uniformly in $\epsilon$, and converges to zero for $x\neq z$.
 To summarize
 \begin{align}
 \lim_{\epsilon\to 0}R_\epsilon(z;\z)&= -\frac{\mu\gamma}{2}\sum_{j=1}^N \frac{\alpha_j }{z-z_j}  (\int_{\C}  \frac{1}{z-x} G(x;\z)d^2x-Y_j(\z))\nonumber
 \end{align}
 and then
 \begin{align}
\lim_{\epsilon\to 0}\langle T_\epsilon(z)\prod_{k=1}^N V_{\alpha_k}(z_k)\rangle&=
\Big ( \frac{Q}{2}\sum_{i=1}^N \frac{\alpha_i }{(z-z_i)^2} -\frac{1}{4} \sum_{i,j} \frac{ \alpha_i \alpha_j}{(z-z_i)(z-z_j)} \Big )G(\z)
 +\frac{\mu\gamma}{2} \sum_{i=1}^N  \frac{\alpha_i}{z-z_i} Y_i(\z)\nonumber \\& 
 =\sum_{i=1}^N \frac{\Delta_{\alpha_i }}{(z-z_i)^2}G(\z)+\sum_{i=1}^N \frac{1}{z-z_i}\partial_{z_i}G(\z)\label{w1conv}
 \end{align}
 by \eqref{derivativecorr1}. \qed

\vspace{0.2 cm}

Now, we need the following elaboration of this result in the proof of the second Ward identity:

\begin{lemma}\label{dT'} The convergence in \eqref{w1conv} takes place in $C^1$ i.e.
\begin{align*}
\lim_{\epsilon\to 0}\partial_{z_i}\langle T_\epsilon(z)\prod_{k=1}^N V_{\alpha_k}(z_k)\rangle=\partial_{z_i}\langle T(z)\prod_{k=1}^N V_{\alpha_k}(z_k)\rangle
 \end{align*}
 \end{lemma}
 \begin{proof}

It suffices to study the convergence of  $\partial_{z_i}R_\epsilon(z;\z)$ where $R_\epsilon(z;\z)$ is given by \eqref{Rfinal}. By Lemma  \ref{derconv},
 $R_\epsilon(z;\z)$ is $C^1$ (as a function of the $\z$ variable).  To study the $\epsilon\to 0$ limit let us consider the   terms individually. First consider the second term i.e. \eqref{yjterm}.
 The first term in  \eqref{yjterm} converges in  $C^1$ since  $ h_\epsilon$ converges in $C^1$ in the complement of the origin and since $Y_i$ is $C^1$. For the second term in  \eqref{yjterm} let $\tilde{\rho}$ be a smooth bump supported in a small neighborhood of origin with $\tilde{\rho}(x)=1$ for $|x|$ small. Insert $1=\tilde{\rho}(z-x)+1-\tilde{\rho}(z-x)$ in the integral. By  Lemma  \ref{derconv}
  \begin{align*}
   \partial_{z_i}\int_{\C} &(h_\epsilon(z-x)-h_\epsilon(z-z_j)) \frac{1}{x-z_j} (1-\tilde{\rho}(z-x))G(x;\z)d^2x=\\&   \int_{\C} (h_\epsilon(z-x)-h_\epsilon(z-z_j)) \frac{1}{x-z_j} (1-\tilde{\rho}(z-x)) \partial_{z_i}G(x;\z)d^2x
\\&+ \int_{\C}  \partial_{z_i}((h_\epsilon(z-x)-h_\epsilon(z-z_j)) \frac{1}{x-z_j}) (1-\tilde{\rho}(z-x))G(x;\z)d^2x\to 0
   \end{align*}
as $\epsilon\to 0$ since $(h_\epsilon(z-x)-h_\epsilon(z-z_j)) \frac{1}{x-z_j}(1-\tilde{\rho}(z-x))$ is smooth and bounded and converges to $0$ in $C^1$.
 
 For the $\tilde{\rho}(z-x)$ case we note that $x \mapsto G(x;\z)$ is $C^1$ on the support of $\tilde{\rho}(z-x)$ (provided the support of $\tilde{\rho}$ is small enough) hence this term too vanishes  as $\epsilon\to 0$. Summarizing, we obtained
  \begin{align*}
 \lim_{\epsilon\to 0} \partial_{z_i} \int_{\C} h_\epsilon(z-x) \frac{1}{x-z_j} G(x;\z)d^2x=  -\partial_{z_i}Y_j(\z).
  \end{align*}
 The convergence of $ \int_{\C}  \frac{1}{(z-x)_\epsilon}\partial_{z_i} G(x;\z)d^2x$  follows in the same way by localizing at $z$ and its complement and in the same way we also get
 \begin{align}
 \lim_{\epsilon\to 0}  \int_{\C} &( \partial_{z} \frac{1}{(z-x)_\epsilon} + (  \frac{1}{(z-x)_\epsilon})^2) \partial_{z_i}G(x;\z)d^2x\nonumber \\=&\lim_{\epsilon\to 0}  \int_{\C} ( \partial_{z} \frac{1}{(z-x)_\epsilon} + (  \frac{1}{(z-x)_\epsilon})^2) (\partial_{z_i}G(x;\z)-\partial_{z_i}G(z;\z))d^2x =0\label{conve1}
  \end{align} 
where we have added the term involving $\partial_{z_i}G(z;\z)$ above as its contribution vanishes by isotropy and then used the fact that  $\partial_{z_i}G(x;\z)$ is of class $C^2$ in the neighborhood of $z$.

The proof will be completed provided we show
\begin{equation}\label{basicconv}
\lim_{\epsilon\to 0}   \partial_{z_i}  \int_{\C} \int_{\C} \chi_\epsilon(x,y,z)
G(x,y;\z)d^2x d^2y=\lim_{\epsilon\to 0}    \int_{\C} \int_{\C} \chi_\epsilon(x,y,z)
 \partial_{z_i}G(x,y;\z)d^2x d^2y=0
\end{equation}
where
\begin{equation*}
\chi_\epsilon(x,y,z):=
 \frac{1}{(z-x)_\epsilon} \frac{1}{(z-y)_\epsilon}-(\frac{1}{(z-x)_\epsilon}- \frac{1}{(z-y)_\epsilon})\frac{1}{x-y}.
\end{equation*}
We used  Lemma \ref{derconv} in the first equality in \eqref{basicconv}. \eqref{basicconv} is proved  by localizing the $x,y$ integrals by inserting $1= \tilde{\rho} (x-z)+1- \tilde{\rho} (x-z)$ or $1= \tilde{\rho} (x-z_i)+1- \tilde{\rho} (x-z_i)$ and similarly with the $y$ variable. This results to insertion of $f (x-u)  g (y-v)$ in the integrand where $f,g\in\{\tilde{\rho}, ,1-\tilde{\rho}\} $ and $u,v\in \{z_i,z\}$. We consider in detail the cases where $f=g=\tilde\rho$ the others being similar but easier. 
There are three cases to consider taking into account symmetry in $x$ and $y$: 

\vskip 2mm

\noindent (a) The first case is the insertion of $  \tilde{\rho} (x-z_i)  \tilde{\rho} (y-z_i)$.
By translation invariance 
\begin{align*}
 &  \int_{\C} \int_{\C} \chi_\epsilon(x,y,z) 
  \tilde{\rho} (x-z_i)  \tilde{\rho} (y-z_i)  \partial_{z_i}G(x,y;\z)d^2x d^2y  \\
& =-\sum_{j; j\neq i}^N  \int_{\C} \int_{\C} \chi_\epsilon(x,y,z) 
  \tilde{\rho} (x-z_i)  \tilde{\rho} (y-z_i) \partial_{z_j}  G(x,y;\z)d^2x d^2y   \\ 
& - \int_{\C} \int_{\C} \chi_\epsilon(x,y,z) 
  \tilde{\rho} (x-z_i)  \tilde{\rho} (y-z_i)  (\partial_x+\partial_y) G(x,y;\z)d^2x d^2y.
\end{align*}
For the first term we note that 
$\chi_\epsilon$ tends to zero in sup-norm  on the support of the integrand. The limit vanishes  using  
Corollary \ref{derest} and then Proposition \ref{1and2point} along with the dominated convergence theorem.  Integrating by parts  the second term equals
$$
 \int_{\C} \int_{\C} (\partial_x+\partial_y)(\chi_\epsilon(x,y,z) 
  \tilde{\rho} (x-z_i)  \tilde{\rho} (y-z_i) )  G(x,y;\z)d^2x d^2y.
$$
Also $(\partial_x+\partial_y)\chi_\epsilon$  tends to zero in sup-norm  on the support of the integrand and then Proposition \ref{1and2point} along with the dominated convergence theorem gives the claim.

\vskip 2mm

\noindent (b) As the second case we consider the insertion $  \tilde{\rho} (x-z)  \tilde{\rho} (y-z_i) $.
Proceeding as in (a) we need to study the convergence of
\begin{align*}
& -\sum_{j; j\neq i}^N  \int_{\C} \int_{\C} \chi_\epsilon(x,y,z) 
  \tilde{\rho} (x-z)  \tilde{\rho} (y-z_i) \partial_{z_j}  G(x,y;\z)d^2x d^2y   \\ 
& - \int_{\C} \int_{\C} \chi_\epsilon(x,y,z)   \tilde{\rho} (x-z)  \tilde{\rho} (y-z_i)  (\partial_x+\partial_y) G(x,y;\z)d^2x d^2y.
\end{align*}
The  first term is dealt as in (a) case. For the second  term we use the fact that $ x,y \mapsto G(x,y;\z)$ is $C^1$ on the support of $  x,y \mapsto \tilde{\rho} (x-z)  \tilde{\rho} (y-z_i)$ and conclude by dominated convergence theorem.

\vskip 2mm

\noindent (c) The third case is given by
\begin{align*}
   \int_{\C} \int_{\C} \chi_\epsilon(x,y,z)   \tilde{\rho} (x-z)  \tilde{\rho} (y-z)  \partial_{z_i}G(x,y;\z)d^2x d^2y   \end{align*}
 This tends to zero thanks to Corollary \ref{derest} and then Proposition \ref{1and2point} along with the dominated convergence theorem.

\end{proof}

\subsection{Proof of the Second Ward Identity}

We will study the limit 
\begin{equation}\label{Wardidentity}
  \langle T(z) T(z') \prod_{k=1}^N  V_{\alpha_k}(z_k)  \rangle:= \underset{\epsilon' \to 0}{\lim}\; \underset{\epsilon \to 0}{\lim}\;  \langle T_\epsilon(z) T_{\epsilon'}(z') \prod_{k=1}^N  V_{\alpha_k}(z_k)  \rangle     .
\end{equation}
The proof consists of lengthy but straightforward calculations using the integration by parts formula  \eqref{ipp1} together with a careful analysis of the $\epsilon,\epsilon'$  limits as many integrals will not be absolutely convergent.

The first step consists of using the integration by parts formula twice to get rid of the $T_\epsilon(z)$
\begin{align}
& \langle T_\epsilon(z) T_{\epsilon'}(z') \prod_{k=1}^N  V_{\alpha_k}(z_k)  \rangle = ( -\frac{1}{2}  ( \frac{1}{(z'-z)_{\epsilon,\epsilon'}^2}  )^2+3Q^2\frac{1}{(z'-z)_{\epsilon,\epsilon'}^4} )G(\z)\nonumber\\&+2Q\frac{1}{(z'-z)_{\epsilon,\epsilon'}^3}\langle(\partial_{z'}\phi_{\epsilon'}(z')-\partial_z \phi_\epsilon(z))\prod_{k=1}^N  V_{\alpha_k}(z_k)  \rangle -2\frac{1}{(z'-z)_{\epsilon,\epsilon'}^2}\langle(\partial_{z'}\phi_{\epsilon'}(z')\partial_z\phi_\epsilon(z)\prod_{k=1}^N  V_{\alpha_k}(z_k)  \rangle \nonumber\\&+
F_{\epsilon,\epsilon'}(z,z';\z)\label{Feps}
  \end{align} 
  where the contractions to the $V_{\alpha_k}$ can be read from \eqref{tcontra1},  \eqref{tcontra2}:
 \begin{align}
F_{\epsilon,\epsilon'}(z,z';\z)&=
(\frac{Q}{2}\sum_{i=1}^N\alpha_i  \frac{1}{(z-z_i)^2_\epsilon}  -\frac{1}{4} \sum_{i,j=1}^N \alpha_i \alpha_j \frac{1}{(z-z_i)_\epsilon} \frac{1}{(z-z_j)_\epsilon}) \langle  T_{\epsilon'}(z') \prod_{k=1}^N  V_{\alpha_k}(z_k)  \rangle \nonumber
 \\& +\frac{ \mu \gamma}{2}     \sum_{i=1}^N \alpha_i  \frac{1}{(z-z_i)_\epsilon} \int_{\C}  \frac{1}{(z-x)_\epsilon} \langle  T_{\epsilon'}(z')V_\gamma(x) \prod_{k=1}^N  V_{\alpha_k}(z_k)  \rangle 
 d^2x\nonumber \\&+  \frac{\mu \gamma }{2}  \int_{\C} ( \frac{\gamma}{2}   (  \frac{1}{(z-x)_\epsilon})^2-Q  \frac{1}{(z-x)^2_\epsilon} ) \langle  T_{\epsilon'}(z')V_\gamma(x) \prod_{k=1}^N V_{\alpha_k}(z_k)  \rangle 
 d^2x \nonumber
  \\&
 -\frac{\mu^2 \gamma^2}{4}  \int_{\C} \int_{\C} \frac{1}{(z-y)_\epsilon} \frac{1}{(z-x)_\epsilon} \langle  T_{\epsilon'}(z') V_\gamma(x)V_\gamma(y) \prod_{k=1}^N  V_{\alpha_k}(z_k)  \rangle \label{epward}
 d^2x d^2y . 
 \end{align}
 Note that the correlations appearing in $F_{\epsilon,\epsilon'}(z,z';\z)$ are of the form appearing in Lemma \ref{derconv}. This allows us to take the $\epsilon\to 0$ limit in \eqref{epward}. Consider in particular the Beurling transform type terms, i.e. the third term in the above right-hand side. Since $x\mapsto \langle  T_{\epsilon'}(z')V_\gamma(x) \prod_{k=1}^NV_{\alpha_k}(z_k)  \rangle $ is $C^1$ in a neighborhood of $z$ these terms converge to
 $$ - \mu \int_{\C} \frac{1}{(z-x)^2}   \langle  T_{\epsilon'}(z') V_\gamma(x) \prod_{k=1}^N V_{\alpha_k}(z_k)  \rangle 
 d^2x. $$
 Writing $\frac{1}{(z-x)^2}=\partial_x\frac{1}{(z-x)}$ and then using twice  integration by parts  (first with respect to $\partial_x$ and then Gaussian integration by parts), we get
 \begin{align*}
 - \mu \int_{\C} \frac{1}{(z-x)^2} & \langle  T_{\epsilon'}(z') V_\gamma(x) \prod_{k=1}^N V_{\alpha_k}(z_k)  \rangle 
 d^2x =-\frac{\mu \gamma}{2} \sum_{i=1}^N \alpha_i  \frac{1}{z-z_i} \int_{\C} \frac{1}{z-x} \langle  T_{\epsilon'}(z')V_\gamma(x) \prod_{k=1}^N  V_{\alpha_k}(z_k)  \rangle d^2x \\&-\frac{\mu \gamma}{2} \sum_{i=1}^N \alpha_i  \frac{1}{z-z_i} \int_{\C} \frac{1}{x-z_i} \langle  T_{\epsilon'}(z')  V_\gamma(x) \prod_{k=1}^N V_{\alpha_k}(z_k)  \rangle d^2x   \\&+\frac{ \mu^2 \gamma^2}{4}  \int_{\C} \int_{\C} \frac{1}{z-y} \frac{1}{z-x} \langle  T_{\epsilon'}(z')V_\gamma(x)V_\gamma(y)  \prod_{k=1}^N  V_{\alpha_k}(z_k)  \rangle 
 d^2x d^2y \\&+\mu\gamma Q \int_{\C}  \frac{1}{(z'-x)^3_{\epsilon'}} \frac{1}{z-x} \langle V_\gamma(x)  \prod_{k=1}^N V_{\alpha_k}(z_k)  \rangle d^2x \\&+\mu\gamma\int_{\C}  \frac{1}{(z'-x)^2_{\epsilon'}} \frac{1}{z-x} \langle\partial_{z'}\phi_{\epsilon'}(z') V_\gamma(x)  \prod_{k=1}^N V_{\alpha_k}(z_k)  \rangle  d^2x.
\end{align*}
 As in the proof of the 1st Ward identity, we next compare these expressions with
\begin{align*}
\partial_{z_i}\langle  T_{\epsilon'}(z') \prod_{k=1}^N  V_{\alpha_k}(z_k)  \rangle&=-\frac{\alpha_i}{2} \sum_{j; j\neq i}^N\frac{\alpha_j}{z_i-z_j} \langle  T_{\epsilon'}(z') \prod_{k=1}^N  V_{\alpha_k}(z_k)  \rangle+\alpha_iQ \frac{1}{(z'-z_i)^3_{\epsilon'}}\langle  \prod_{k=1}^N V_{\alpha_k}(z_k)  \rangle\\&+\frac{\alpha_i\mu\gamma}{2} \int_{\C} \frac{1}{z_i-x} \langle  T_{\epsilon'}(z') V_\gamma(x) \prod_{k=1}^N V_{\alpha_k}(z_k)  \rangle d^2x
+\alpha_i \frac{1}{(z'-z_i)^2_{\epsilon'}}\langle \partial_{z'} \phi_{\epsilon'}(z')\prod_{k=1}^N V_{\alpha_k}(z_k)  \rangle.
\end{align*}
 Some calculation gives 
\begin{align}
\lim_{\epsilon\to 0}F_{\epsilon,\epsilon'}(z,z',\z)&=
\sum_{i=1}^N \frac{\Delta_{\alpha_i} }{(z-z_i)^2}  \langle  T_{\epsilon'}(z') \prod_{k=1}^N  V_{\alpha_k}(z_k)  \rangle +\sum_{i=1}^N  \frac{1}{z-z_i}\partial_{z_i}  \langle  T_{\epsilon'}(z') \prod_{k=1}^N  V_{\alpha_k}(z_k)  \rangle\nonumber
 \\& -\sum_{i=1}^N \frac{Q\alpha_i }{z-z_i}  \frac{1}{(z'-z_i)^3_{\epsilon'}} \langle \prod_{k=1}^N  V_{\alpha_k}(z_k)  \rangle -\sum_{i=1}^N \frac{{\alpha_i} }{z-z_i}  \frac{1}{(z'-z_i)^2_{\epsilon'}} \langle\partial_{z'}\phi_{\epsilon'}(z') \prod_{k=1}^N  V_{\alpha_k}(z_k)  \rangle\nonumber
 \\&+\mu\gamma Q \int_{\C}  \frac{1}{(z'-x)^3_{\epsilon'}} \frac{1}{z-x} \langle V_\gamma(x)  \prod_{k=1}^N V_{\alpha_k}(z_k)  \rangle d^2 x
 \nonumber \\&+\mu\gamma\int_{\C}  \frac{1}{(z'-x)^2_{\epsilon'}} \frac{1}{z-x} \langle\partial_{z'}\phi_{\epsilon'}(z') V_\gamma(x)  \prod_{k=1}^N V_{\alpha_k}(z_k)  \rangle d^2x,\label{Flimit}
\end{align}
 after recalling that $\Delta_{\alpha_i}= \frac{\alpha_i}{2}(Q-\frac{\alpha_i}{2})$. Let us now address the $\epsilon'\to 0$ limit. Lemma \ref{dT'}  takes care of the second term in \eqref{Flimit}.  The $\partial\phi$ terms  in \eqref{Feps}  converge as  $\epsilon'\to 0$ respectively  to
 \begin{align*}
\frac{2Q}{(z'-z)^3}\langle(\partial_{z'}\phi(z')-\partial_z\phi(z))\prod_l  V_{\alpha_l}(z_l)  \rangle&= \frac{Q}{(z'-z)^2}\sum_{i=1}^N \frac{{\alpha_i} }{(z'-z_i)(z-z_i)}\langle\prod_{k=1}^N  V_{\alpha_k}(z_k)  \rangle\\&-\mu\gamma\frac{Q}{(z'-z)^2}\int_{\C} \frac{1}{z'-x}\frac{1}{z-x}\langle V_\gamma(x)\prod_{k=1}^N  V_{\alpha_k}(z_k)  \rangle d^2x\\
\frac{-2}{(z'-z)^2}\langle \partial_{z'}\phi(z')\partial_{z}\phi(z) \prod_{k=1}^N  V_{\alpha_k}(z_k)  \rangle&=  \frac{1}{(z'-z)^4} \langle \prod_{k=1}^N  V_{\alpha_k}(z_k)  \rangle 
-\frac{1}{2}\frac{1}{(z'-z)^2}\sum_{i,j} \frac{{\alpha_i\alpha_j} }{(z'-z_i)(z-z_j)}\langle\prod_{k=1}^N  V_{\alpha_k}(z_k)  \rangle\\&+\frac{\mu\gamma}{2(z'-z)^2}\sum_{i=1}^N \int_{\C} (\frac{{\alpha_i} }{z'-z_i}\frac{1}{z-x}+\frac{{\alpha_i} }{z-z_i}\frac{1}{z'-x})\langle V_\gamma(x)\prod_{k=1}^N  V_{\alpha_k}(z_k)  \rangle d^2x\\&
-\frac{\mu^2\gamma^2}{2}\frac{1}{(z'-z)^2}\int_{\C} \int_{\C} \frac{1}{z'-y}\frac{1}{z-x}\langle V_\gamma(x)V_\gamma(y)\prod_{k=1}^N V_{\alpha_k}(z_k)  \rangle d^2xd^2y\\&
+\frac{\mu\gamma^2}{2} \frac{1}{(z'-z)^2}\int_{\C} \frac{1}{z'-x}\frac{1}{z-x}\langle V_\gamma(x)\prod_{k=1}^N  V_{\alpha_k}(z_k)  \rangle d^2x.
  \end{align*} 
Next, consider the last two terms in \eqref{Flimit}. First we integrate by parts the last one:
  \begin{align}
\mu\gamma &\int_{\C}  \frac{1}{(z'-x)^2_{\epsilon'}} \frac{1}{z-x}
 \langle\partial\phi_{\epsilon'}(z') V_\gamma(x)  \prod_{k=1}^N V_{\alpha_k}(z_k)  \rangle d^2x =\\&-\frac{\mu\gamma}{2} \sum_{i=1}^N \frac{\alpha_i }{(z'-z_i)_{\epsilon'}}\int_{\C}  \frac{1}{(z'-x)^2_{\epsilon'}} \frac{1}{z-x}
\langle V_\gamma(x)\prod_{k=1}^N  V_{\alpha_k}(z_k)  \rangle d^2x\\&
-\frac{\mu\gamma^2}{2} \int_{\C}  \frac{1}{(z'-x)^2_{\epsilon'}}\frac{1}{(z'-x)_{\epsilon'}} \frac{1}{z-x}\langle V_\gamma(x)\prod_{k=1}^N V_{\alpha_k}(z_k)  \rangle d^2x\nonumber \\&+\frac{\mu^2\gamma^2}{2}
\int_{\C} \int_{\C} \frac{1}{(z'-y)_{\epsilon'}} \frac{1}{(z'-x)^2_{\epsilon'}} \frac{1}{z-x}\langle  V_\gamma(x)V_\gamma(y)\prod_{k=1}^N  V_{\alpha_k}(z_k)  \rangle d^2xd^2y
\label{Flimit1}.
  \end{align}  
Since   $x \mapsto \frac{1}{z-x}\langle V_\gamma(x)\prod_{k=1}^N  V_{\alpha_k}(z_k)  \rangle$ is $C^2$ in a neighborhood of $z'$  terms in \eqref{Flimit} and \eqref{Flimit1} combine to the following limit:
   \begin{align*}
 \lim_{\epsilon'\to 0}&\int_{\C}  (\mu\gamma Q\frac{1}{(z'-x)^3_{\epsilon'}}   -\hf\mu\gamma^2\frac{1}{(z'-x)^2_{\epsilon'}}\frac{1}{(z'-x)_{\epsilon'}})\frac{1}{z-x}\langle V_\gamma(x) \prod_{k=1}^N V_{\alpha_k}(z_k)  \rangle  d^2x=\\&
2\mu \int_{\C} \frac{1}{(z'-x)^3} \frac{1}{z-x}\langle V_\gamma(x) \prod_{k=1}^N V_{\alpha_k}(z_k)  \rangle d^2x
\end{align*}
where for a $C^2$ function $f$ we denote 
$$\int_{\C} \frac{1}{(z-x)^3}f(x)d^2x=\lim_{\epsilon\to 0}\int_{\C} \frac{1}{(z-x)^3}1_{|z-x|\geq \epsilon}f(x)d^2x.
$$
The proof of existence of the limit \eqref{Wardidentity} is now completed and we proceed to analyze the result. We have obtained:
\begin{align}%
  \langle T(z) T(z') \prod_l  V_{\alpha_l}(z_l)  \rangle&=\frac{1}{2}\frac{ c_L}{(z'-z)^4}G(\z)+
 \sum_{i=1}^N \frac{\Delta_{\alpha_i}}{(z-z_i)^2}\langle T(z') \prod_{k=1}^N V_{\alpha_k}(z_k)  \rangle\nonumber\\& +\sum_{i=1}^N\frac{1}{z-z_i}\partial_{z_i}\langle T(z') \prod_{k=1}^N V_{\alpha_k}(z_k)  \rangle+R_1G(\z)+R_2+R_3\label{Wardidentity0}
\end{align}
with
\begin{align}%
 R_1&=\frac{1}{(z'-z)^2}( \sum_{i=1}^N\frac{Q\alpha_i}{(z'-z_i)(z-z_i)}-\frac{1}{2}\sum_{i,j=1}^N \frac{\alpha_i\alpha_j}{(z'-z_i)(z-z_j)})\nonumber\\&+\frac{1}{2}\sum_{i,j=1}^N\frac{\alpha_i\alpha_j}{(z'-z_i)(z'-z_j)^2(z-z_j)}-\sum_{i=1}^N\frac{Q\alpha_i}{(z'-z_i)^3(z-z_i)}\nonumber \\&
 =(\frac{1}{(z'-z)^2}-\frac{1}{2}\frac{1}{z'-z}\partial_{z'})(Q\sum_{i=1}^N\frac{\alpha_i}{(z'-z_i)^2}-\frac{1}{2}(\sum_{i=1}^N\frac{\alpha_i}{z'-z_i})^2)\label{Wardidentity00}
\end{align}
and the terms involving integrals of correlations
\begin{align*}
 R_2&=\frac{\mu\gamma}{2}\frac{1}{(z'-z)^2}\sum_{i=1}^N \int_{\C}(\frac{{\alpha_i} }{z'-z_i}\frac{1}{z-x}+\frac{{\alpha_i} }{z-z_i}\frac{1}{z'-x}) G(x;\z) d^2x\\& - \frac{\mu \gamma}{2} \sum_{i=1}^N \frac{\alpha_i}{z-z_i}  \frac{1}{(z'-z_i)^2}  \int_{\C} \frac{1}{z'-x} G(x;\z)d^2x
 \end{align*}
and
\begin{align*}
 R_3&=-\frac{2 \mu}{(z'-z)^2}\int_{\C} \frac{1}{(z'-x)(z-x)}G(x;\z)d^2x-\frac{\mu^2 \gamma^2}{2} \frac{1}{(z'-z)^2}\int_{\C} \int_{\C} \frac{1}{(z'-y)(z-x)}G(x,y;\z)d^2x d^2y \\
 & + 2 \mu \int_{\C} \frac{1}{(z'-x)^3 (z-x)}G(x;\z)d^2x  -\frac{\mu \gamma}{2}   \sum_{i=1}^N  \frac{\alpha_i}{z'-z_i} \int_{\C} \frac{1}{(z'-x)^2 (z-x)}G(x;\z)d^2x \\
 & +\frac{\mu^2 \gamma^2}{2}\int_{\C} \int_{\C} \frac{1}{(z'-y)(z-x)(z'-x)^2}G(x,y;\z)  d^2x d^2y. 
 \end{align*}

Now, we use the identity 
\begin{equation*}
\frac{1}{(z'-y)(z-x)}  (\frac{1}{(z'-x)^2}  - \frac{1}{(z'-z)^2})= - \frac{1}{(z'-z)} \frac{1}{(z'-x)^2(z'-y)}-\frac{1}{(z'-z)^2} \frac{1}{(z'-x)(z'-y)} 
\end{equation*}
to get
\begin{align*}
 R_3&=-\frac{2 \mu}{(z'-z)^2}\int_{\C} \frac{1}{(z'-x)(z-x)}G(x;\z)d^2x-\frac{\mu^2 \gamma^2}{2} \frac{1}{(z'-z)}\int_{\C} \int_{\C} \frac{1}{(z'-x)^2(z'-y)}G(x,y;\z)d^2x d^2y \\
 & + 2 \mu \int_{\C} \frac{1}{(z'-x)^3 (z-x)}G(x;\z)d^2x  -\frac{\mu \gamma}{2}   \sum_{i=1}^N  \frac{\alpha_i}{z'-z_i} \int_{\C} \frac{1}{(z'-x)^2 (z-x)}G(x;\z)d^2x \\
 & -\frac{\mu^2 \gamma^2}{2} \frac{1}{(z'-z)^2}\int_{\C} \int_{\C} \frac{1}{(z'-x)(z'-y)}G(x,y;\z) d^2x d^2y .
 \end{align*}
We use the following identity to get rid of the $\int_{\C} \int_{\C} \frac{1}{(z'-x)^2(z'-y)}G(x,y;\z)  d^2x d^2y$ term 
\begin{align*}
 \frac{2 \mu}{z'-z} \int_{\C} \frac{1}{(z'-x)^3}G(x;\z)d^2x & =  \frac{\mu \gamma}{2}  \frac{1}{z'-z} \sum_{i=1}^N \alpha_i  \int_{\C} \frac{1}{(z'-x)^2 (x-z_i)}   G(x;\z)d^2x \\
 & -  \frac{\mu^2 \gamma^2}{2}  \frac{1}{z'-z}   \int_{\C} \int_{\C} \frac{1}{(z'-x)^2(z'-y)} G(x,y;\z) d^2x d^2y \\
\end{align*}
This leads to 
\begin{align*}
 R_3&=-\frac{2 \mu}{(z'-z)^2}\int_{\C} \frac{1}{(z'-x)(z-x)}G(x;\z)d^2x-  \frac{\mu \gamma}{2}  \frac{1}{z'-z} \sum_i \alpha_i  \int_{\C} \frac{1}{(z'-x)^2 (x-z_i)}   G(x;\z)d^2x \\
 & + 2 \mu \frac{1}{z'-z} \int_{\C} \frac{1}{(z'-x)^2 (z-x)}G(x;\z)d^2x  -\frac{\mu \gamma}{2}   \sum_{i=1}^N  \frac{\alpha_i}{z'-z_i} \int_{\C} \frac{1}{(z'-x)^2 (z-x)}G(x;\z)d^2x \\
 & -\frac{\mu^2 \gamma^2}{2} \frac{1}{(z'-z)^2}\int_{\C} \int_{\C} \frac{1}{(z'-x)(z'-y)}G(x,y;\z) d^2x d^2y.
 \end{align*}
Since
\begin{equation*}
2 \mu \frac{1}{z'-z} \int_{\C} \frac{1}{(z'-x)^2 (z-x)}G(x;\z)d^2x -\frac{2 \mu}{(z'-z)^2}\int_{\C} \frac{1}{(z'-x)(z-x)}G(x;\z)d^2x= -\frac{2 \mu }{(z'-z)^2}   \int_{\C} \frac{1}{(z'-x)^2}G(x;\z)d^2x
\end{equation*}
we finally get that 
\begin{align*}
 R_3&=-\frac{2 \mu }{(z'-z)^2}   \int_{\C} \frac{1}{(z'-x)^2}G(x;\z)d^2x-  \frac{\mu \gamma}{2}  \frac{1}{z'-z} \sum_{i=1}^N \alpha_i  \int_{\C} \frac{1}{(z'-x)^2 (x-z_i)}   G(x;\z)d^2x \\
 &   -\frac{\mu \gamma}{2}   \sum_{i=1}^N  \frac{\alpha_i}{z'-z_i} \int_{\C} \frac{1}{(z'-x)^2 (z-x)}G(x;\z)d^2x  -\frac{\mu^2 \gamma^2}{2} \frac{1}{(z'-z)^2}\int_{\C} \int_{\C} \frac{1}{(z'-x)(z'-y)}G(x,y;\z) d^2x d^2y.
 \end{align*}
The key identity \eqref{Theexactidt} gives us once again that
\begin{align*}
& -\frac{2 \mu }{(z'-z)^2}   \int_{\C} \frac{1}{(z'-x)^2}G(x;\z)d^2x -\frac{\mu^2 \gamma^2}{2} \frac{1}{(z'-z)^2}\int_{\C} \int_{\C} \frac{1}{(z'-x)(z'-y)}G(x,y;\z) d^2x d^2y  \\
& = -\mu \gamma  \frac{1}{(z'-z)^2} \sum_{i=1}^N \alpha_i \int_{\C} \frac{1}{(z'-x)(x-z_i)} G(x;\z)d^2x.
\end{align*}
This enables to simplify $R_3$ and therefore at the end we get the following expression for $R_2+R_3$
\begin{align*}
 R_2 +R_3& =\frac{\mu\gamma}{2}\frac{1}{(z'-z)^2}\sum_{i=1}^N \int_{\C} (\frac{{\alpha_i} }{z'-z_i}\frac{1}{z-x}+\frac{{\alpha_i} }{z-z_i}\frac{1}{z'-x}) G(x;\z) d^2x  \\
&    -  \frac{\mu \gamma}{2}  \frac{1}{z'-z} \sum_{i=1}^N \alpha_i  \int_{\C} \frac{1}{(z'-x)^2 (x-z_i)}   G(x;\z)d^2x   -\frac{\mu \gamma}{2}   \sum_{i=1}^N  \frac{\alpha_i}{z'-z_i} \int_{\C} \frac{1}{(z'-x)^2 (z-x)}G(x;\z)d^2x \\
 &    -\mu \gamma  \frac{1}{(z'-z)^2} \sum_{i=1}^N \alpha_i \int_{\C} \frac{1}{(z'-x)(x-z_i)} G(x;\z)d^2x - \frac{\mu \gamma}{2} \sum_{i=1}^N \frac{\alpha_i}{z-z_i}  \frac{1}{(z'-z_i)^2}  \int_{\C} \frac{1}{z'-x} G(x;\z)d^2x.
 \end{align*}
In the above expression, by a tedious (but rather straightforward) computation, we can simplify the fractions that appear in the integrands: this yields that
\begin{equation*}
 R_2 +R_3= \frac{\mu \gamma}{(z'-z)^2}  \sum_{i=1}^N  \frac{\alpha_i}{z'-z_i}  \int_{\C} \frac{1}{z_i-x} G(x;\z)d^2x+  \frac{\mu \gamma}{2(z'-z)}  \sum_{i=1}^N  \frac{\alpha_i}{(z'-z_i)^2}  \int_{\C} \frac{1}{z_i-x} G(x;\z)d^2x.
 \end{equation*} 

We can now get the final result by using identity \eqref{notationderivsum}.
\qed

\subsection{Holomorphic identity for BPZ}

To prepare for the proof of Theorem  \ref{BPZTH}, we will consider once again the key identity \eqref{Theexactidt} when one of the  insertion points  has  weight $-\frac{\gamma}{2}$ and is evaluated at $z$. In order to highlight the special role played by the insertion at point $z$, we will not use in this proof the shorthand  notation $G$ for the correlations. 
Also, we will slightly change the definitions behind  equality \eqref{Theexactidt} by adding in $A_\epsilon$ the term in $B_\epsilon$ which depends on the insertion $-(\frac{\gamma}{2},z)$. 
This leads to the following definitions: 
\begin{align} 
\bar{A}_\epsilon(z) & = - \int_{\C} \frac{1}{(z-x)^2}  \Big(  \langle V_{-\frac{\gamma}{2},\epsilon}(z)  V_{\gamma,\epsilon}(x)  \prod_{k=1}^N  V_{\alpha_k, \epsilon}(z_k)  \rangle_{ \epsilon} - 1_{|x-z| \leq 1}\langle  V_{-\frac{\gamma}{2},\epsilon}(z) V_{\gamma,\epsilon}(z)\prod_{k=1}^N  V_{\alpha_k, \epsilon}(z_k)  \rangle_{ \epsilon}\Big) d^2x \nonumber  \\
& +\frac{\gamma^2}{4} \int_{\C}  \frac{1}{(z-x)} \frac{1}{(z- x)_\epsilon}  \langle V_{-\frac{\gamma}{2},\epsilon}(z) V_{\gamma,\epsilon}(x)\prod_{k=1}^N  V_{\alpha_k,\epsilon}(z_k)  \rangle_{\epsilon}\,d^2x ,\label{defAbar}\\
\bar{B}_\epsilon(z)  
&=  \frac{\gamma}{2} \sum_{i=1}^N \alpha_i \:  \int_{\C} \frac{1}{(z-x)} \frac{1}{(x-z_i)_\epsilon}  \langle  V_{-\frac{\gamma}{2},\epsilon}(z) V_{\gamma,\epsilon}(x)\prod_{k=1}^N  V_{\alpha_k,\epsilon}(z_k)  \rangle_{\epsilon}\,d^2x, \\
\bar{C}_\epsilon (z)&= - \frac{\mu \gamma^2}{4}  \int_{\C} \int_{\C}   \frac{x-y}{(z-y)(z-x)}   \frac{1}{(x-y)_\epsilon} \langle V_{-\frac{\gamma}{2},\epsilon}(z)  V_{\gamma,\epsilon }(x)V_{\gamma,\epsilon}(y)\prod_{k=1}^N  V_{\alpha_k,\epsilon }(z_k)  \rangle_{\epsilon}\,d^2x  d^2y.
\end{align}
In \eqref{defAbar} we have written the Beurling transform for the smooth function directly without writing explicitly the limiting procedure. Notice that the first integral in the definition of $\bar{A}_\epsilon$ is the sum of two terms and that the contribution of the term involving $1_{|x-z| \leq 1}\dots $ is null by isotropy. 
 The key identity \eqref{Theexactidt} now reads
\begin{equation}\label{Thesameexactidt}
\bar{A}_{\epsilon}(z)+\bar{B}_{\epsilon}(z)+\bar{C}_{\epsilon}(z)=0.
\end{equation}  
 The $\epsilon\to 0$ limits are covered by the following Lemmas proven in Section \ref{sectionestimates}:

\begin{lemma}\label{appliBPZ1}
The function $x \mapsto \frac{1}{(z-x)^2}  \langle V_{-\frac{\gamma}{2}}(z)  V_{\gamma}(x)  \prod_{k=1}^N  V_{\alpha_k}(z_k)  \rangle   $ is integrable and we have
\begin{equation*}
\lim_{\epsilon\to 0}\bar{A}_{\epsilon}(z)=\bar{A}(z):=(\frac{\gamma^2}{4}-1) \int_{\C} \frac{1}{(z-x)^2}   \langle V_{-\frac{\gamma}{2}}(z)  V_{\gamma}(x)  \prod_{k=1}^N  V_{\alpha_k}(z_k)  \rangle d^2x   
\end{equation*}
uniformly  on compact subsets of $\mathbb{C} \setminus \lbrace z_1, \cdots z_n \rbrace$. \end{lemma}
\medskip

\begin{lemma}\label{appliBPZ2}
The function $(x,y) \mapsto \frac{1}{(z-y)(z-x)}  \langle V_{-\frac{\gamma}{2}}(z)  V_{\gamma}(x) V_{\gamma}(y)  \prod_{k=1}^N  V_{\alpha_k}(z_k)  \rangle  $ is  integrable and
\begin{equation*}
\lim_{\epsilon\to 0}\bar{C}_{\epsilon}(z)=\bar{C}(z):=- \frac{\mu \gamma^2}{4} \int_{\C} \frac{1}{(z-y)(z-x)}    \langle V_{-\frac{\gamma}{2}}(z)  V_{\gamma}(x) V_{\gamma}(y) \prod_{k=1}^N  V_{\alpha_k}(z_k)  \rangle d^2xd^2y   
\end{equation*}
uniformly  on compact subsets of $\mathbb{C} \setminus \lbrace z_1, \cdots z_n \rbrace$.

\end{lemma}

We can now take the limit in \eqref{Thesameexactidt}: this yields for all $z\in \mathbb{C} \setminus \lbrace z_1, \cdots z_n \rbrace$
\begin{equation*}
\bar{A}(z)+\bar{B}(z)+\bar{C}(z)=0
\end{equation*}  
where $\bar{A}$, $\bar{C}$ have been defined in the previous two lemmas and $\bar{B}$ is given by the following expression
\begin{equation}
\bar B(z) =  \frac{\gamma}{2\pi} \sum_{i=1}^N \frac{\alpha_i}{z-z_i} (  \int_{\C} \frac{1}{z-x}   \langle V_{-\frac{\gamma}{2}}(z)  V_{\gamma}(x) \prod_{k=1}^N  V_{\alpha_k}(z_k)  \rangle d^2x   -\int_{\C} \frac{1}{z_i-x}   \langle V_{-\frac{\gamma}{2}}(z)  V_{\gamma}(x) \prod_{k=1}^N  V_{\alpha_k}(z_k)  \rangle d^2x)\label{barB}
\end{equation}
where recall that
\begin{equation*}
\int_{\C} \frac{1}{z_i-x}   \langle V_{-\frac{\gamma}{2}}(z)  V_{\gamma}(x) \prod_{k=1}^N  V_{\alpha_k}(z_k)  \rangle d^2x= \underset{\epsilon \to 0}{\lim} \int_{\C} \frac{1}{(z_i-x)_{\epsilon,\epsilon}}   \langle V_{-\frac{\gamma}{2},\epsilon}(z)  V_{\gamma,\epsilon}(x) \prod_{k=1}^N  V_{\alpha_k,\epsilon}(z_k)  \rangle_\epsilon d^2x
\end{equation*}
and that the above limit exists by lemma \ref{1stder}. Also recall that by lemma \ref{1stder} we get  
\begin{align}
\partial_{z_i}   \langle  V_{-\frac{\gamma}{2}} (z) \prod_{k=1}^N  V_{\alpha_k  }(z_k)\rangle&=\frac{\gamma\alpha_i}{4} \frac{1}{z_i-z} - \frac{1}{2}\sum_{j; j \not = i}^N  \frac{ \alpha_i \alpha_j}{z_i-z_j} \langle V_{-\frac{\gamma}{2}} (z) \prod_{k=1}^N  V_{\alpha_k}(z_k)  \rangle\nonumber\\
&+ \frac{\mu \gamma \alpha_i}{2}   \int_{\C} \frac{1}{z_i-x}  \langle V_{-\frac{\gamma}{2}}(z)  V_{\gamma}(x) \prod_{k=1}^N  V_{\alpha_k}(z_k)  \rangle d^2x  \label{remindderiv}.
\end{align}

\subsection{Proof of Theorem \ref{BPZTH}}

From eq. \eqref{notationderivsum} we get for $z\in  \mathbb{C} \setminus \lbrace z_1, \cdots z_n \rbrace$
\begin{align}
 \partial_{z}\langle   V_{- \frac{\gamma}{2}}(z) \prod_{k=1}^N V_{\alpha_k}(z_k)   \rangle &=\caF_1(z)   - \frac{\mu \gamma^2}{4}\caF_2 (z)\label{derco}
 \end{align}
 with
\begin{align*}  &\caF_1(z)  := \frac{\gamma}{4}\sum_{i=1}^N  \frac{\alpha_i}{(z-z_i)}  \langle   V_{- \frac{\gamma}{2}}(z) \prod_{k=1}^N V_{\alpha_k}(z_k)\rangle,
\\
 &\caF_2 (z)  :=
  \int_{\C}   \frac{1}{z- x}   \langle   V_{- \frac{\gamma}{2}}(z) V_{\gamma} (x)\prod_{k=1}^N V_{\alpha_k}(z_k)   \rangle d^2x 
\end{align*} 
$\caF_1$  is $C^1$ in $\mathbb{C} \setminus \lbrace z_1, \cdots z_n \rbrace$ and using \eqref{notationderivsum}  its derivative is given by
\begin{align*}
\partial_z\caF_1(z)&=  -\frac{\gamma}{4}\sum_{i=1}^N  \frac{\alpha_i}{(z-z_i)^2}  \langle   V_{- \frac{\gamma}{2}}(z) \prod_{k=1}^N V_{\alpha_k}(z_k)   \rangle ,\\
& +  \frac{\gamma^2}{16}  \sum_{i,j=1}^N  \frac{\alpha_i\alpha_j}{(z-z_i)(z-z_j)}\langle  V_{- \frac{\gamma}{2}}(z) \prod_{k=1}^N V_{\alpha_k}(z_k)   \rangle \\
& - \frac{\mu \gamma^3}{16} \sum_{i=1}^N  \frac{\alpha_i}{(z-z_i)}  \int_{\C} \frac{1}{z-x}   \langle   V_{- \frac{\gamma}{2}}(z)V_\gamma(x) \prod_{k=1}^N V_{\alpha_k}(z_k)   \rangle d^2x.
\end{align*}
Let
\begin{equation*}
 \caF_{2,\epsilon} (z) =
  \int_{\C}   \frac{1}{(z- x)_{\epsilon,\epsilon}}   \langle   V_{- \frac{\gamma}{2},\epsilon}(z) V_{\gamma,\epsilon} (x)\prod_{k=1}^N V_{\alpha_k,\epsilon}(z_k)   \rangle_\epsilon d^2x
\end{equation*}  
Then    $\caF_{2,\epsilon}$ and  $\partial_z\caF_{2,\epsilon}$ converge to  $\caF_2$ and $\partial_z\caF_{2}$  uniformly in compact subsets of $ \mathbb{C} \setminus \lbrace z_1, \cdots z_n \rbrace$.  Differentiating and using \eqref{derivativecorr1} we obtain
\begin{align*}
\partial_z  \caF_{2,\epsilon} (z)&=     \int_{\C} \partial_z \frac{1}{(z-x)_{\epsilon,\epsilon}}     \langle    V_{- \frac{\gamma}{2},\epsilon}(z)  V_{\gamma,\epsilon} (x) \prod_{k=1}^N V_{\alpha_k,\epsilon}(z_k)   \rangle_\epsilon  d^2x\\
&+ \frac{ \gamma^2}{4}   \int_{\C} (\frac{1}{(z-x)_{\epsilon,\epsilon}})^2     \langle    V_{- \frac{\gamma}{2},\epsilon}(z)  V_{\gamma,\epsilon} (x) \prod_{k=1}^N V_{\alpha_k,\epsilon}(z_k)   \rangle_\epsilon  d^2x\\
&+\frac{ \gamma}{4}\sum_{i=1}^N\frac{\alpha_i}{(z-z_i)_{\epsilon,\epsilon}}\int_{\C} \frac{1}{(z-x)_{\epsilon,\epsilon}}     \langle    V_{- \frac{\gamma}{2},\epsilon}(z)  V_{\gamma,\epsilon} (x) \prod_{k=1}^N V_{\alpha_k,\epsilon}(z_k)   \rangle_\epsilon  d^2x\\
&  - \frac{\mu  \gamma^2}{4} \int_{\C}   \int_{\C}  \frac{1}{(z-y)_{\epsilon,\epsilon}}  \frac{1}{(z-x)_{\epsilon,\epsilon}}    \langle    V_{- \frac{\gamma}{2},\epsilon}(z)  V_{\gamma,\epsilon} (x)V_{\gamma,\epsilon} (y) \prod_{k=1}^N V_{\alpha_k,\epsilon}(z_k)   \rangle_\epsilon  d^2x  d^2y.
\end{align*}
In the same way as in Lemmas \ref{appliBPZ1} and \ref{appliBPZ2} we get 
\begin{equation*}
\lim_{\epsilon\to 0}\partial_z  \caF_{2,\epsilon}(z)= \bar A(z)+\bar C(z)+\frac{ \gamma}{4}\sum_{i=1}^N\frac{\alpha_i}{z-z_i}\int_{\C} \frac{1}{z-x}     \langle    V_{- \frac{\gamma}{2}}(z)  V_{\gamma} (x) \prod_{k=1}^N V_{\alpha_k}(z_k)   \rangle  d^2x.
\end{equation*}
 In conclusion, we get the following for  $ z\in\mathbb{C} \setminus \lbrace z_1, \cdots z_n \rbrace$:
 \begin{align*}
&  \frac{4}{\gamma^2}   \partial_{zz}^2 \langle    V_{- \frac{\gamma}{2}}(z)   \prod_{k=1}^N V_{\alpha_k}(z_k)   \rangle  \\
& =  -\frac{1}{\gamma} \sum_{i=1}^N \frac{\alpha_i}{(z-z_i)^2} \langle    V_{- \frac{\gamma}{2}}(z)   \prod_{k=1}^N V_{\alpha_k}(z_k)   \rangle  +\frac{1}{4}\sum_{i,j=1}^N  \frac{\alpha_i \alpha_j}{(z-z_i) (z-z_j) }  \langle  V_{- \frac{\gamma}{2}}(z) \prod_{k=1}^N V_{\alpha_k}(z_k)   \rangle  \\ 
& -\frac{\mu \gamma}{2}  \sum_{i=1}^N  \frac{\alpha_i}{z-z_i}\int_{\C} \frac{1}{z-x} \langle  V_{- \frac{\gamma}{2}}(z) V_{\gamma}(x) \prod_{k=1}^N V_{\alpha_k}(z_k)   \rangle    d^2x  -\mu (\bar A(z)+\bar C(z))
\end{align*} 
Now using $ -\mu(\bar A(z)+\bar C(z))=\mu\bar B(z)$ and recalling \eqref{barB} and \eqref{remindderiv} and collecting terms the BPZ equation holds. 

.\qed

\subsection{Proof of Theorem \ref{theo4point}}

Combining \eqref{Z1}  with  \eqref{Climit} we get the following expression for the three point structure constant:
\begin{equation*}
C_\gamma(\alpha_1,\alpha_2,\alpha_3)=B(\alpha_1,\alpha_2,\alpha_3)\mu^{-s}  \gamma^{-1}\Gamma(s)  \E (\rho(\alpha_1,\alpha_2,\alpha_3)^{-s})
\end{equation*}
where
\begin{equation*}
\rho(\alpha_1,\alpha_2,\alpha_3)= e^{\frac{\chi \gamma^2}{2}} \int_{\C} e^{\gamma X(x)- \frac{\gamma^2}{2} \E[X(x)^2] } \frac{1}{ |x|^{\gamma \alpha_1}  |x-1|^{\gamma \alpha_2}  }  \hat{g}(x)^{1 - \frac{\gamma}{4} \sum_{k=1}^3 \alpha_k } d^2x.
\end{equation*}
and $B(\alpha_1,\alpha_2,\alpha_3)$ is given by \eqref{Adefi}. Similarly, the function $G$ defined in  \eqref{Glimit} becomes
 $$
  G(z)=|z|^{\frac{\gamma \alpha_1}{2}}   |z-1|^{\frac{\gamma \alpha_2}{2}}  \mathcal{T}(z) 
  $$ 
  where $ \mathcal{T}(z)$ is given by 
\begin{align}\label{Tdefi}
  \mathcal{T}(z) &=B(-\frac{_\gamma}{^2},\alpha_1,\alpha_2,\alpha_3) \mu^{-s+\frac{1}{2}}  \gamma^{-1}\Gamma(s-\frac{1}{2})  \E [R(z)^{\frac{1}{2}-s}]
 \end{align}
and
\begin{equation*}
R(z)= e^{\frac{\chi \gamma^2}{2}} \int_{\C} e^{\gamma X(x)- \frac{\gamma^2}{2} \E[X(x)^2] } \frac{|x-z|^{\frac{\gamma^2}{2}}}{ |x|^{\gamma \alpha_1}  |x-1|^{\gamma \alpha_2}  }  \hat{g}(x)^{1+\frac{\gamma^2}{8} - \frac{\gamma}{4} \sum_{k=1}^3 \alpha_k } d^2x.
\end{equation*}
The function $\caT$ is positive and $C^2$ on $\C\setminus\{0,1\}$.
Note also that 
\begin{equation}\label{T(0)}
\mathcal{T}(0)=C_\gamma(\alpha_1-\frac{_\gamma}{^2},\alpha_2,\alpha_3).
\end{equation}
The BPZ equation  applied to the expression \eqref{confinv} yields the following equation for $G$
\begin{equation*}
\frac{4}{\gamma^2} \partial_{z}^2 G(z)  - (\frac{1}{z}+\frac{1}{z-1})\partial_z G(z) + \left ( \frac{\Delta_1}{z^2}+\frac{\Delta_2}{(z-1)^2} + \frac{\Delta_3-\Delta_2-\Delta_1-\Delta_{-\frac{\gamma}{2}}}{z(z-1)}   \right ) G(z)=0.
\end{equation*}
 and, after a bit of calculus, we see that the function $\mathcal{T}$ is a   solution of a PDE version of the Gauss hypergeometric equation
\begin{equation}\label{hypergeo}
\partial_{z}^2 \caT(z)+  \frac{({ c}-z({ a}+{ b}+1))}{z(1-z)}\partial_z \caT(z) -\frac{{ a}{ b} }{z(1-z)}\caT(z)=0
\end{equation}
where ${ a},{ b},{ c}$ are given by \eqref{defab} and \eqref{defc}. In the Appendix we prove
\begin{lemma}\label{lemmaPDEsbis}
Let $a,b,c$ be real parameters such that
\begin{equation}\label{Condparameters}
c \in \R \setminus \Z, \quad c-a-b \in  \R \setminus \Z
\end{equation}
Then, the every solution of equation \eqref{hypergeo} in $\C \setminus \lbrace 0,1 \rbrace$ 
can be written as
\begin{equation*}
\caT(z)= \lambda_1  | F_{-}(z) |^2+ \lambda_2  | F_{+}(z) |^2
\end{equation*}
where $\lambda_1, \lambda_2$ are related by the following relation
\begin{equation}\label{Fundrelation}
\lambda_1   \frac{\Gamma(c)^2}{   \Gamma(c-a)  \Gamma(c-b)  \Gamma (a)  \Gamma(b) }+ \lambda_2   \frac{\Gamma(2-c)^2}{   \Gamma(1-a)  \Gamma(1-b)  \Gamma(a-c+1)  \Gamma(b-c+1) }=0
\end{equation}
In particular, the solution space is one dimensional if one of the coefficients in front of the $\lambda_i$ is non zero.
\end{lemma}
In order to determine  $\lambda_1, \lambda_2$ we study $\caT(z)$ near the origin. From the behaviour of  $F_\pm$ we get
\begin{equation*}
 \lambda_1 |F_{-}(z)|^2 +\lambda_2  |F_{+}(z)|^2  = \lambda_1+ \lambda_2 |z|^{2(1-c)} + o(|z|^{2(1-c)}  ) .
\end{equation*}
On the other hand we have
\begin{lemma}\label{asyex}
The function $\mathcal{T}(z)$ has the following expansion around $z=0$
\begin{equation}
\mathcal{T}(z)= C_\gamma(\alpha_1-\frac{\gamma}{2},\alpha_2,\alpha_3) -\mu  \frac{\pi}{  l(-\frac{\gamma^2}{4}) l(\frac{\gamma \alpha_1}{2})  l(2+\frac{\gamma^2}{4}- \frac{\gamma \alpha_1}{2}) } C_\gamma(\alpha_1+\frac{\gamma}{2}, \alpha_2,\alpha_3) |z|^{2(1-c)}+ o(|z|^{2(1-c)})\label{z=0exp}.
\end{equation}
\end{lemma}
Theorem  \ref{theo4point} and Corollary  \ref{3pointconstant} follow from Lemmas \ref{lemmaPDEsbis} and \ref{asyex} under the condition \eqref{Condparameters}. From \eqref{defab} and \eqref{defc} we infer  $c=1+\frac{\gamma}{2}(\alpha_1-Q)$ and $c-a-b=\hf(1-\alpha_3)$. The function $\caT$ is continuous in $\alpha_1$ and $\alpha_3$. Therefore the Theorems hold in general.
 
\noindent{\it Proof of Lemma \ref{asyex}}.
We have
\begin{align*}
  \E [R(z)^{\frac{1}{2}-s}]- \E [R(0)^{\frac{1}{2}-s}]  &= (\hf-s) \int_0^1    \E[  (R(z)-R(0))R_t(z)^{-\frac{1}{2}-s}    ]      dt   
\end{align*}
where $R_t(z)=tR(0)+(1-t)R(z)$. We set 
$$m_z(u)=   \frac{|u-z|^{\frac{\gamma^2}{2}} -|u|^{\frac{\gamma^2}{2}} }{ |u|^{\gamma \alpha_1}  |u-1|^{\gamma \alpha_2}  } \hat{g}(u)^{1+\frac{\gamma^2}{8} - \frac{\gamma}{4} \sum_{k=1}^3 \alpha_k } 
$$ 
and
\begin{equation*}
R(z,u)= e^{\frac{\chi\gamma^2}{2}}  \int_{\C} e^{\gamma X(x)- \frac{\gamma^2}{2} \E[X(x)^2] } \frac{|x-z|^{\frac{\gamma^2}{2}}}{ |x|^{\gamma \alpha_1}  |x-1|^{\gamma \alpha_2} }  e^{\gamma^2G(u,x) }\hat{g}(x)^{1+\frac{\gamma^2}{8} - \frac{\gamma}{4} \sum_{k=1}^3 \alpha_k } d^2x.
\end{equation*}
Then, we have by Cameron-Martin's transformation  
\begin{align}
 \E[  (R(z)-R(0))R_t(z)^{-\frac{1}{2}-s}    ]   =  e^{\frac{\chi\gamma^2}{2}} \int_{\C}  m_z(u)  \E[  R_t(z,u)^{-\frac{1}{2}-s} ]d^2u   \label{ineqT} 
\end{align}
where $R_t(z,u)=tR(0,u)+(1-t)R(z,u)$. Since the Green function $G$ is bounded from below, we have the estimate
\begin{equation*}
\inf_{|z| \leq \frac{1}{2}, u \in \C} R(z,u)  \geq C\int_{   2 \leq |x| \leq 3}  e^{\gamma X(x)- \frac{\gamma^2}{2} \E[X(x)^2]}d^2x:=W 
\end{equation*} 
so that
\begin{align}
\sup_u \E[  R_t(z,u)^{-\frac{1}{2}-s} ] \leq C \E[ W^{-\frac{1}{2}-s} ] <\infty. \label{ineqT1} 
\end{align}
Hence
\begin{align}
  \int_{\C}  m_z(u)  1_{|u|  \geq \frac{1}{2}}  \E[  R_t(z,u)^{-\frac{1}{2}-s} ]d^2u    \leq  C \int_{\C}  |m_z(u)|  1_{|u|  \geq \frac{1}{2}} du \leq C|z|. \label{ineqT2} 
\end{align}
Therefore, in the equality \eqref{ineqT} and up to a $O(|z|)$ term, we can put the indicator $1_{|u| \leq \frac{1}{2}}$ in the integral.  By the change of variables  $u=|z| v $ we then get
\begin{align*}
 \int_{\C}  m_z(u) 1_{|u| \leq \frac{1}{2}} \E[  R_t(z,u)^{-\frac{1}{2}-s} ]d^2u =
  |z|^{2+\frac{\gamma^2}{2}-\gamma\alpha_1}  \int_{\C}  n_z(v) 1_{|v| \leq \frac{1}{2|z|}} \E[  R_t(z,v|z|)^{-\frac{1}{2}-s} ]d^2v
\end{align*}
with
$$
n_z(v):= \frac{|v-\frac{z}{|z|}|^{\frac{\gamma^2}{2}} -|v|^{\frac{\gamma^2}{2}} }{ |v|^{\gamma \alpha_1} |1-v|z||^{\gamma \alpha_2}   } \hat{g}(|z|v)^{1+\frac{\gamma^2}{8} - \frac{\gamma}{4} \sum_{k=1}^3 \alpha_k }.
$$
We have then
\begin{equation*}
 n_z(v)1_{|v| \leq \frac{1}{2|z|}}   \leq C \sup_{|w|=1}  \frac{||v-w |^{\frac{\gamma^2}{2}} -|v|^{\frac{\gamma^2}{2}}| }{ |v|^{\gamma \alpha_1}  } \equiv k(v) .
\end{equation*}
 Let $Q-\frac{1}{\gamma}< \alpha_1<Q-\frac{\gamma}{2}$. Then $k\in L^1(\C)$. Combining this with
 \eqref{ineqT1} implies that we may apply dominated convergence to conclude
\begin{align*}
\lim_{r\to 0} \int_{\C}  n_{re^{i\theta}}(v) 1_{|v| \leq \frac{1}{2r}} \E[  R_t(re^{i\theta},vr)^{-\frac{1}{2}-s} ]d^2v=\int_{\C} 
    \frac{|v-e^{i\theta}|^{\frac{\gamma^2}{2}} -|v|^{\frac{\gamma^2}{2}}} { |v|^{\gamma \alpha_1}  } \lim_{r\to 0} \E[  R_t(re^{i\theta},\frac{_v}{^r})^{-\frac{1}{2}-s} ]d^2v
\end{align*}
 Also, the following convergence holds almost surely 
\begin{equation*}
R_t(re^{i\theta},vr) \underset{r \to 0}{\rightarrow}e^{\gamma^2\chi} e^{\frac{\chi\gamma^2}{2}} \int_{\C} e^{\gamma X(x)- \frac{\gamma^2}{2} \E[X(x)^2] } \frac{1}{ |x|^{\gamma (\alpha_1+\frac{\gamma}{2})}  |x-1|^{\gamma \alpha_2} } \hat{g}(x)^{1-\frac{\gamma^2}{8} - \frac{\gamma}{4} \sum_{k=1}^3 \alpha_k } d^2x .
\end{equation*}
We conclude  that
\begin{align*}
 \E (R(z))^{\hf-s}= \E (R(0))^{\frac{1}{2}-s} + (\frac{1}{2}-s) |z|^{2(1-c)}e^{-\gamma^2\chi s}   \rho(\alpha_1+\frac{_\gamma}{^2},\alpha_2,\alpha_3)\int_{\C} 
    \frac{|v-1|^{\frac{\gamma^2}{2}} -|v|^{\frac{\gamma^2}{2}}} { |v|^{\gamma \alpha_1}  }d^2v + o(|z|^{2(1-c)}).
\end{align*}
    Combining this with \eqref{Tdefi} and \eqref{T(0)} we obtain the claim \eqref{z=0exp}
since 
$$
e^{-\gamma^2\chi s} \frac{B(-\frac{\gamma}{2},\alpha_1,\alpha_2,\alpha_3)}{B(\alpha_1+\frac{\gamma}{2},\alpha_2,\alpha_3)}=1
$$
and 
\begin{equation}\label{selberg}
 \int_{\C}     \frac{|v-1|^{\frac{\gamma^2}{2}} -|v|^{\frac{\gamma^2}{2}} }{ |v|^{\gamma \alpha_1}  } d^2v=  \frac{\pi}{  l(-\frac{\gamma^2}{4}) l(\frac{\gamma \alpha_1}{2})  l(2+\frac{\gamma^2}{4}- \frac{\gamma \alpha_1}{2}) } 
\end{equation}
which is a consequence of formula \eqref{formuleint2}  in the appendix.\qed
\section{Estimates on the correlation functions} \label{sectionestimates}
In this section, we state general estimates for  the correlation functions when two or three  insertion points get close: these estimates are called fusion estimates in the physics literature. Using  these estimates, we prove Proposition \ref{1and2point} 
and \ref{appliBPZ2}. In the whole section  $X_\epsilon$ stands for the $\epsilon$-regularization of the Free Field in terms of circle average or mollification (it does not matter).


\subsection{Fusion Estimates}

 Recall that we need to bound correlations of the form \eqref{defgeneralG} where the points $\z$ are fixed and non coinciding in a bounded region; in this context, the variables $x_1$ and $x_2$ may get together or close to one of the $z_i$ and one must estimate the corresponding asymptotic of the correlation \eqref{defgeneralG}. Also, we need to get  decay as $x_1$ or $x_2$ goes to infinity. We fix a constant $\delta>0$ so that the quantity $\delta$ will measure the minimal distance between the points not getting together and  the quantity $\delta^{-1}$ the minimal distance from the origin of the points going to infinity. The constants $C$ in the bounds will be $\delta$ dependent and all the weights $\alpha$ of the vertex operators satisfy $\alpha<Q$.

We will first give general bounds in the  propositions \ref{upperbound2} and \ref{upperbound3} below for correlations of the form 
\begin{equation}\label{estim1}
\langle V_{\beta_1,\epsilon}(x_1)V_{\beta_2,\epsilon}(x_2)  \prod_{k=r+1}^N  V_{\alpha_k,\epsilon  }(z_k)\rangle_{\epsilon}
\end{equation}
and
\begin{equation}\label{estim2}
\langle V_{\beta_1,\epsilon}(x_1)V_{\beta_2,\epsilon}(x_2) V_{\beta_3,\epsilon}(x_3)  \prod_{k=r+1}^N  V_{\alpha_k,\epsilon  }(z_k)\rangle_{\epsilon}
\end{equation}
where $r \in \lbrace 0,1 \rbrace$, $\beta_1, \beta_2,\beta_3$ and the sequence $(\alpha_i)_{r+1 \leq i \leq N}$ will satisfy certain conditions.


First let us consider correlation functions   when two points $x_1$ and $x_2$ get together in a ball of radius $\delta^{-1}$. Define for $r \in \lbrace 0,1 \rbrace$ 
$$ U_{N-r,\delta}=\{ \z_{N-r}=(z_{r+1}, \cdots, z_N) \in \C^{N-r}\ |\ \min_{ i\not= j} |z_i-z_j|\geq \delta\}$$
and  the corresponding set   
\begin{equation*}
O_{N-r,\delta}( \z_{N-r})= \C \setminus \cup_{k=r+1}^N  B(z_k,\delta)
\end{equation*}
 Set
$$
|z|_\epsilon:=\epsilon\vee |z|.
$$

Then we have the following fusion estimates:
\begin{proposition}
\label{upperbound2} Let $\z_{N-r}\in U_{N-r,\delta}$ and $x_1,x_2\in O_{N-r,\delta}(\z_{N-r})$ with $x_1,x_2,z_{r+1},\dots,z_N\in B(0,\delta^{-1})$. Let the weights 
$\beta_1,\beta_2,(\alpha_k)_{1+r \leq k \leq N}$ satisfy 
$$
\beta_1+\beta_2\geq Q\ \ {\rm and}\ \   \sum_{k=r+1}^N \alpha_k>Q\ \ \ {\rm or}\ \ \  \beta_1+\beta_2< Q\ \ {\rm and}\ \  \beta_1+\beta_2+\sum_{k=r+1}^N \alpha_k>2Q.
$$
 Then 
\begin{equation}\label{upperbound2eq}
\langle V_{\beta_1,\epsilon}(x_1)V_{\beta_2,\epsilon}(x_2)  \prod_{k=r+1}^N  V_{\alpha_k,\epsilon  }(z_k)\rangle_{\epsilon} \leq C_\delta \hat{g}(z_{r+1})^{\Delta_{\alpha_2}}|x_1-x_2|_\epsilon^{2\Delta_{(\beta_1+ \beta_2)\wedge Q}-2\Delta_{\beta_1}-2\Delta_{\beta_2}} |\ln (|x_1-x_2|_\epsilon )|^{-c_{\beta_1,\beta_2}}
\end{equation}
where $c_{\beta_1,\beta_2}=\frac{3}{2}$ if $ \beta_1+\beta_2>Q$,  $c_{\beta_1,\beta_2}=\frac{1}{2}$ if $ \beta_1+\beta_2=Q$ and $c_{\beta_1,\beta_2}=0$ if $ \beta_1+\beta_2<Q$.
\end{proposition}
{\begin{remark}
With a bit of work, one could turn inequality \eqref{upperbound2eq} into a precise asymptotics as $|x_1-x_2| \to 0$ with $x_1$ fixed say. Estimate \eqref{upperbound2eq} is consistent with the physicist's discussion of  fusion rules in LCFT. The fusion rules in CFT  describe the singularity of correlation functions as two primary field insertion points $x_1,x_2$ come together. The asymptotics is supposed to be given by the following rule as $|x_1-x_2| \to 0$ with $x_1$ fixed
$$
V_{\beta_1}(x_1)V_{\beta_2}(x_2)\sim\sum_{\beta_3}|x_1-x_2|^{2\Delta_{\beta_3}-2\Delta_{\beta_1}-2\Delta_{\beta_2}}V_{\beta_3}(x_1).
$$
In LCFT if $\beta_1+\beta_2<Q$,  it is expected  that the leading singularity is given by that $\beta_3$ which occurs in the sum with the largest conformal weight $\Delta_{\beta_3}$, which turns out to be given by $\beta_3=\beta_1+\beta_2$ and this is consistent with what we prove. If $\beta_1+\beta_2\geq Q$ the leading asymptotics is believed to involve an integral over $P\in\R$ with $\beta_3=Q+iP$. This leads to the logarithmic corrections $|\ln (|x_1-x_2|)|^{-\hf}$ for  $\beta_1+\beta_2= Q$ and $|\ln (|x_1-x_2|)|^{-\frac{3}{2}}$ for  $\beta_1+\beta_2> Q$ exactly as we prove.
\end{remark}}
Next we merge three vertex operators. Note that we need this only when all the insertions are in a bounded region. The next proposition reflects the above fusion rules: we first fuse the closest pair of points and then the third point is fused to the result.

\begin{proposition}\label{upperbound3}
Assume  that $x_1,x_2,x_3 \in O_{N-r,\delta}(\mathbf{z}_{N-r})\cap B(0,\delta^{-1})$ and let  $|x_1-x_2|\leq |x_1-x_3|\leq |x_2-x_3|$. Then 

\medskip

\noindent {\rm (a)} If $\sum_{k=r+1}^N\alpha_k>Q$, $\beta_1+\beta_2\geq Q$  and $\beta_3\geq 0$ then
\begin{align*}
   \langle V_{\beta_1,\epsilon}(x_1)V_{\beta_2,\epsilon}(x_2)V_{\beta_3,\epsilon}(x_3)  \prod_{k=r+1}^N  V_{\alpha_k,\epsilon  }(z_k)\rangle_{\epsilon} &\leq C_\delta|x_1-x_2|_\epsilon^
{\frac{Q^2}{2}-2\Delta_{\beta_1}-2\Delta_{\beta_2}} |x_1-x_3|_\epsilon^{-2\Delta_{\beta_3}}.
   \end{align*}
   
\medskip

\noindent {\rm (b)} If $\beta_3+\sum_{k=r+1}^N\alpha_k>Q$, $\beta_1+\beta_2\geq Q$ and $\beta_3<0$  then
\begin{align*}
   \langle V_{\beta_1,\epsilon}(x_1)V_{\beta_2,\epsilon}(x_2)V_{\beta_3,\epsilon}(x_3) \prod_{k=r+1}^N  V_{\alpha_k,\epsilon  }(z_k)\rangle_{\epsilon} \leq C_\delta |x_1-x_2|_\epsilon^{\frac{Q^2}{2}
  -2\Delta_{\beta_1}-2\Delta_{\beta_2}} |x_2-x_3|_\epsilon^{-\beta_3Q}.
   \end{align*}

\medskip

\noindent {\rm (c)} If $\sum_{k=r+1}^N\alpha_k>Q$, $\beta_1+\beta_2< Q$ and $\beta_1+\beta_2+\beta_3\geq Q$   then
\begin{align*}
   \langle V_{\beta_1,\epsilon}(x_1)V_{\beta_2,\epsilon}(x_2)V_{\beta_3,\epsilon}(x_3)  \prod_{k=r+1}^N  V_{\alpha_k,\epsilon  }(z_k)\rangle_{\epsilon} \leq C_\delta |x_1-x_2|_\epsilon^{-\beta_1\beta_2} |x_1-x_3|_\epsilon
^{\hf(\beta_1+\beta_2+\beta_3-Q)^2-\beta_1\beta_3}|x_2-x_3|_\epsilon^{-\beta_2\beta_3}.
   \end{align*}

\medskip

\noindent {\rm (d)} If $\sum_{k=1}^3\beta_k+\sum_{k=r+1}^N\alpha_k>2Q$ $\beta_1+\beta_2< Q$ and $\beta_1+\beta_2+\beta_3< Q$   then
\begin{align*}
   \langle V_{\beta_1,\epsilon}(x_1)V_{\beta_2,\epsilon}(x_2)V_{\beta_3,\epsilon}(x_3) \prod_{k=r+1}^N  V_{\alpha_k,\epsilon  }(z_k)\rangle_{\epsilon} \leq C_\delta |x_1-x_2|_\epsilon^{-\beta_1\beta_2} |x_1-x_3|_\epsilon
^{-\beta_1\beta_3}|x_2-x_3|_\epsilon^{-\beta_2\beta_3}.
   \end{align*}

\end{proposition}

\begin{remark}
In the above proposition we did not bother proving optimal bounds in the sense that we did not focus on the logarithmic corrections. We believe that  the power law we get is sharp.
\end{remark}   
  
Now we focus on the behaviour of the correlation functions close to $\infty$. First, we have for one point tending to infinity:

\begin{proposition}\label{lpinfty}
Let the weights satisfy $\sum_{k=r+1}^N\alpha_k>2Q$. Let $\mathbf{z}_{N-r}\in U_{N-r,\delta}$ and $z_k\in B(0,\delta^{-1})$ for all $k$.  Let $x\in B(0,\delta^{-1})^c\cap O_{N-r,\delta}(\z_{N-r})$. Then 
\begin{equation*}
\langle V_{\gamma,\epsilon}(x)\prod_{k=r+1}^N  V_{\alpha_k,\epsilon}(z_k)  \rangle_{\hat{g},\epsilon}\leq C_\delta|x|^{-4}.
\end{equation*}
\end{proposition} 

and then for two points fusing near $\infty$:

\begin{proposition}\label{2merginfinity}
 Let $\mathbf{z}_{N-r} \in U_{N-r,\delta}$ and $z_k\in B(0,\delta^{-1})$ for all $k$.   Let $x_1, x_2\in B(0,\delta^{-1})^c\cap O_{N-r,\delta}(\z_{N-r})$. Let the weights 
$\beta_1,\beta_2,(\alpha_k)_{1+r \leq k \leq N}$ satisfy 

\begin{enumerate}
\item either   $\beta_1+\beta_2\geq Q$ and $\sum_{k=r+1}^N\alpha_k>Q$ then 
$$ \langle V_{\beta_1,\epsilon}(x_1)V_{\beta_2,\epsilon}(x_2)  \prod_{k=r+1}^N  V_{\alpha_k,\epsilon  }(z_k)\rangle_{\epsilon} \leq C_\delta
 |x_1|^{-4\Delta_{\beta_1}} |x_2|^{-4\Delta_{\beta_2}}1\vee \big(\frac{|x_1x_2|}{|x_1-x_2|_\epsilon}\big)^{
2 \Delta_{\beta_1}+2\Delta_{\beta_2}-\frac{1}{2} Q^2}.
$$
\item or  $\beta_1+\beta_2<Q$ and $\beta_1+\beta_2+\sum_{k=r+1}^N\alpha_k>2Q$ then
$$ \langle V_{\beta_1,\epsilon}(x_1)V_{\beta_2,\epsilon}(x_2)  \prod_{k=r+1}^N  V_{\alpha_k,\epsilon  }(z_k)\rangle_{\epsilon} \leq C_\delta
 |x_1|^{-4\Delta_{\beta_1}} |x_2|^{-4\Delta_{\beta_2}}1\vee \big(\frac{|x_1x_2|}{|x_1-x_2|_\epsilon}\big)^{\beta_1\beta_2}.
$$
\end{enumerate}
\end{proposition}

Now, from these propositions, we deduce:

\subsection{Proof of Proposition  \ref{1and2point} }

Let us start by studying  $\bar{G} (x;\z)$. Proposition \ref{lpinfty} with $r=0$ ensures integrability of $ (1+ |\ln  |x-z_i|  |)^k \bar{G}(x;\z) $ around infinity for all $k$ and all $i$. Consider next local integrability near the insertion with weight $\alpha_1$ at $z_1$ (the other insertions with weight $\alpha_i$ at point $z_i$ can be dealt with similarly). We claim that for $\mathbf{z}_{N-1} \in U_{N-1,\delta} \cap B(0,\delta^{-1})^{N-1} $ and $x,z_1 \in B(0,\delta^{-1})\cap O_{N-1,\delta}(\z_{N-1})$
\begin{equation}\label{fepz}
 \bar{G}(x;\z) \leq 
C_\delta |x-z_1|^{-2+\zeta}
\end{equation}  
with $\zeta>0$. Let first  $\gamma+\alpha_1\geq Q$. Then proposition \ref{upperbound2} with $r=1$ and $\beta_1=\gamma, \beta_2=\alpha_1$  gives 
$\zeta=2 +2 \Delta_Q  -2\Delta_\gamma-2\Delta_{\alpha_1}  =\frac{1}{2}(Q-\alpha_1)^2>0$ 
since $\alpha_1< Q$. If $\gamma+\alpha_1\leq Q$,  proposition \ref{upperbound2} (with $r=1$ and $\beta_1=\gamma, \beta_2=\alpha_1$) gives $\zeta=2 +2 \Delta_{\gamma+\alpha_1}  -2\Delta_\gamma-2\Delta_{\alpha_1} = 2-\alpha_1 \gamma>0$.


%

We now consider $\bar{G} (x,y;\z)$. The claim (c) along with integrability in the region $x\in O_{N,\delta}(\z), y\in O_{N,\delta}(\z)^c$ follows from the same methods as the study of $\bar{G} (x;\z)$. We are thus left with the case where both
$x$ and $y$ are close to one insertion point, say $z_1$ and the case where  both $x$ and $y$ are outside a large ball.

Let us start with the case  when $x,y$ are close to an insertion place, say $z_1$. We consider only the case when $\alpha_1>0$ since the case  $\alpha_1<0$ is less singular. By the symmetry in exchanging $x$ and $y$ we have three cases to consider when applying  proposition \ref{upperbound3} with $r=1$ and  $\lbrace \beta_1,\beta_2,\beta_3 \rbrace = \lbrace \gamma, \alpha_1 \rbrace$:

\medskip

$\bullet$  $|x-y|\leq |x-z_1|\leq |y-z_1|$:
$$\bar{G} (x,y;\z)\leq 
 \left\{\begin{array}{ll} 
 C_\delta|x-y|^{\frac{Q^2}{2}-4}|x-z_1|^{-2\Delta_{\alpha_1}} ,&\text{ if }\quad   2\gamma+\alpha_1\geq Q \text{ and } 2\gamma\geq Q , \\
 C_\delta|x-y|^{-\gamma^2} |x-z_1|^{\hf Q^2-2\Delta_{\alpha_1}-2}    ,&\text{ if }\quad 2\gamma+\alpha_1\geq Q \,\,\text{ and } 2\gamma<Q,\\
   C_\delta|x-y|^{-\gamma^2}                       |x-z_1|^{-2\alpha_1\gamma}               ,&\text{ if }\quad  2\gamma+\alpha_1< Q  \,\,\text{ and } 2\gamma<Q                                                                              
 \end{array}\right.
. $$
 To discuss local  integrability, consider the first case. Suppose first $\Delta_{\alpha_1}<1$. Then $|x-z_1|^{-2\Delta_{\alpha_1}} $ is locally integrable and since $Q^2>4$ we get that $\bar{G} (x,y;\z) |x-y]^{-\xi}$ is locally integrable for some $\xi>0$. If $\Delta_{\alpha_1}\geq 1$ then 
\begin{align*}
&\int|x-y|^{\frac{Q^2}{2}-4-\xi}|x-z_1|^{-2\Delta_{\alpha_1}}1_{|x-y|\leq |x-z_1|}1_{|x-z_1|,|y-z_1|<1}dxdy\leq\\
&C\int|u|^{\frac{Q^2}{2}-2-\xi-2\Delta_{\alpha_1}}1_{|u|<2}du=
C\int|u|^{-2+\hf(Q-\alpha_i)^2-\xi}1_{|u|<2}du<\infty
 \end{align*} 
 for some $\xi>0$. For the two other cases the two factors are separately integrable.
 
 \medskip

$\bullet$ $  |x-z_1|\leq |y-z_1|\leq |x-y| $ :
$$\bar{G} (x,y;\z)\leq 
 \left\{\begin{array}{ll} 
 C_\delta|x-z_1|^{\frac{Q^2}{2}-2-2\Delta_{\alpha_1}}|y-z_1|^{-2} ,&\text{ if }\quad   2\gamma+\alpha_1\geq Q  \text{ and }\gamma+\alpha_1\geq Q, \\
 C_\delta|x-z_1|^{-\gamma\alpha_1}|y-z_1|^{\frac{Q^2}{2}-2\Delta_{\alpha_1}-4+\gamma\alpha_1}       ,&\text{ if }\quad 2\gamma+\alpha_1\geq Q\text{ and }\gamma+\alpha_1< Q ,\\
 C_\delta  |x-z_1|^{-\gamma \alpha_1}   |y-z_1|^{-\gamma (\alpha_1+\gamma)} 
                              ,&\text{ if }\quad   2\gamma+\alpha_1< Q  \text{ and }\gamma+\alpha_1< Q    .                                                                  
 \end{array}\right.
 $$
 For integrability use $\frac{Q^2}{2}-2\Delta_{\alpha_1}=\hf(Q-\alpha_i)^2>0$ in the first case and
 $\gamma (\alpha_1+\gamma)<\gamma (Q-\gamma)=2-\hf\gamma^2$ in the last case. For the second case, if $\frac{Q^2}{2}-2\Delta_{\alpha_1}-4+\gamma\alpha_1<-2$, integrating over that factor produces $|x-z_1|^{\frac{Q^2}{2}-2\Delta_{\alpha_1}-2} $ which is integrable.

 \medskip

$\bullet$  $  |x-z_1| \leq |x-y| \leq |y-z_1|$:
$$\bar{G} (x,y;\z)\leq 
 \left\{\begin{array}{ll} 
 C_\delta|x-z_1|^{\frac{Q^2}{2}-2-2\Delta_{\alpha_1}}|x-y|^{-2} ,&\text{ if }\quad   2\gamma+\alpha_1\geq Q\text{ and }\gamma+\alpha_1\geq Q , \\
C_\delta|x-z_1|^{-\gamma\alpha_1}|x-y|^{ \frac{Q^2}{2}-4-2\triangle_{\alpha_1}+\gamma\alpha_1}    ,&\text{ if }\quad  2\gamma+\alpha_1\geq Q,\,\, \gamma+\alpha_1< Q
,\\
     C_\delta  |x-z_1|^{-\alpha_1\gamma}     |x-y|^{- \gamma(\gamma+\alpha_1)}    
                ,&\text{ if }\quad  2\gamma+\alpha_1< Q \text{ and }\gamma+\alpha_1<Q .                                                                                         
 \end{array}\right.
 $$
Local  integrablity follows as in the previous case.

Finally, let $|x|,|y|> \frac{1}{\delta}$, we use Proposition \ref{2merginfinity} with $r=0$ and $\beta_1=\beta_2=\gamma$. Since $\Delta_\gamma=1$ it is readily seen that 
$$
\bar{G} (x,y;\z)\leq C_\delta |xy|^{-2-\zeta}|x-y|^{-2+\eta}
$$
for $\zeta,\eta>0$ whereby the bounds in Proposition \ref{1and2point} follow.\qed

\subsection{ H\"older estimates}

Now we turn to the H\"older  estimates for the correlation functions. We define, for $\epsilon>0$    the functions 
\begin{align*}
\caF_\epsilon(x,z)=
\langle V_{\gamma,\epsilon}(x)V_{-\frac{\gamma}{2},\epsilon}(z)   \prod_{k=1}^N  V_{\alpha_i,\epsilon  }(z_i)\rangle_{\epsilon},\\
\bar{\caF}_\epsilon(x,z)=\langle V_{\gamma,\epsilon}(x)V_{-\frac{2}{\gamma},\epsilon}(z)   \prod_{k=1}^N  V_{\alpha_i,\epsilon  }(z_i)\rangle_{\epsilon}.
\end{align*}
By the definition \eqref{Vdefi}  
\begin{equation*}
\caF_\epsilon(z,z)=(A\epsilon)^{\frac{\gamma^2}{2}}\langle  V_{\frac{\gamma}{2},\epsilon}(z)   \prod_{k=1}^N  V_{\alpha_i,\epsilon  }(z_i)\rangle_{\epsilon}
\end{equation*}
and
\begin{equation*}
\bar{\caF}_\epsilon(z,z)=(A\epsilon)^{2}\langle  V_{\gamma- \frac{2}{\gamma},\epsilon}(z)   \prod_{k=1}^N  V_{\alpha_i,\epsilon  }(z_i)\rangle_{\epsilon}.
\end{equation*}
Let $\caF(x,z)=\lim_{\epsilon\to 0}\caF_\epsilon(x,z)$ and $\bar{\caF}(x,z)=\lim_{\epsilon\to 0}\bar{\caF}_\epsilon(x,z)$ which are defined and continuous  in $(x,z)\in(\C\setminus\cup z_i)^2\setminus D $ where $D$ is the diagonal $\{(z,z)| z\in\C\}\subset\C^2$.
By  Proposition \ref{upperbound2} with $r=0$ and $\beta_1=\gamma, \beta_2=-\frac{\gamma}{2}$ we have for all  $x,z\in B(0,\delta^{-1})\cap O_{N,\delta}(\z)$
\begin{equation}\label{holderb}
|\caF(x,z)|\leq C_\delta  |x-z|^{\frac{\gamma^2}{2}}, \quad |\bar{\caF}(x,z)|\leq C_\delta |x-z|^{2}
\end{equation}
 and thus the functions $\caF$ and $\bar{\caF}$ extend continuously to $(\C\setminus\cup_{k=1}^N z_k)^2$  by setting $\caF(z,z)=\lim_{\epsilon\to 0}\caF_\epsilon(z,z)=0$ and $\bar{\caF}(z,z)=\lim_{\epsilon\to 0}\bar{\caF}_\epsilon(z,z)=0$. 

We have the following bound for $\caF_\epsilon$:
\begin{proposition}\label{holdercorrel} For any $\delta>0$, there exists a constant $C_\delta>0$ s.t. for all  $x,z\in B(0,\delta^{-1})\cap O_{N,\delta}(\z)$
$$ |\caF_\epsilon(x,z)-\caF_\epsilon(z,z)|\leq \left\{\begin{array}{ll}C_\delta \epsilon^{ \frac{\gamma^2}{2}-1}
|x-z|, & \text{if }\quad |x-z|\leq \epsilon \\  C_\delta |x-z|^{\frac{\gamma^2}{2}}&  \text{if }\quad |x-z|\geq \epsilon\end{array}\right..$$
\end{proposition}
Similarly, we get the following:  
 \begin{proposition}\label{holdercorrelbis} For any $\delta>0$, there exists a constant $C_\delta>0$ s.t. for all  $x,z\in B(0,\delta^{-1})\cap O_{N,\delta}(\z)$
$$ |\bar{\caF}_\epsilon(x,z)-\bar{\caF}_\epsilon(z,z)|\leq \left\{\begin{array}{ll}C_\delta 
\epsilon |x-z|, & \text{if }\quad |x-z|\leq \epsilon \\  C_\delta |x-z|^{2}&  \text{if }\quad |x-z|\geq \epsilon\end{array}\right..$$
\end{proposition}


\subsection{Proof of Lemma \ref{appliBPZ1}.}

The mapping $x\mapsto  \frac{1}{(z-x)^2}\caF(x,z)$ is obviously continuous on $\C\setminus\{z,z_1,\dots,z_N\}$. By Proposition \ref{lpinfty}, it is  integrable near $\infty$. Furthermore, it possesses  singularities at the points $\{z,z_1,\dots,z_N\}$. From Proposition \ref{upperbound2} we infer a bound $C|x-z_i|^{-2+\zeta}$ with $\zeta>0 $ near $z_i$ and from  \eqref{holderb} we get a bound $C|x-z|^{-2+\tfrac{\gamma^2}{2}}$ near  $z$. Hence integrability 
of $\bar A$ follows. 

The uniform convergence of  $\bar{A}_\epsilon(z)$ (defined in \eqref{defAbar}) on compact sets $K\subset\C\setminus\{z_1,\dots,z_N\}$ towards  $\bar A(z)$ follows now easily from these bounds and the corresponding convergence of $\caF_\epsilon(x,z)$. Consider the first term in  the definition \eqref{defAbar}, call it  $a_\epsilon(z)$.  Let $\delta>0$ and define
\begin{equation}\label{defA1}
a_{\epsilon,\delta}(z)=-\int_{\C}\frac{1}{(z-x)^2}(\caF_\epsilon(x,z)-\mathbf{1}_{|x-z|\leq 1}\caF_\epsilon(z,z))1_\delta(x)
\,d^2x
\end{equation}
where $1_\delta$ is the indicator of the set $\{x: |x-z|>\delta, |x|<\delta^{-1}, \forall i: |x-z_i|>\delta \}$. Since $\caF_\epsilon(x,z)$ converges uniformly on compacts in $(\C\setminus\{z_1,\dots,z_N\})^2\cap\{x\neq z\}$ we infer that $a_{\epsilon,\delta}(z)$ converges uniformly on $K$ to a limit $ a_{0,\delta}(z)$. From the above bounds on the function $x\to\caF_\epsilon(x,z)$ near the points $z_i$, $\infty$ and $z$ we infer
$$
\sup_\epsilon|a_\epsilon(z)-a_{\epsilon,\delta}(z)|\leq C_K\delta^b$$
uniformly on $K$ with $b>0$. Since $\delta$ was arbitrary the uniform convergence of  $a_\epsilon(z)$ on $K$ follows.
Convergence of the second term in   \eqref{defAbar} can be dealt with similarly.

\qed

\subsection{Proof of Lemma \ref{appliBPZ2}.}

Let us first discuss the integrability of the mapping $$(x,y)\mapsto \phi_\epsilon(x,y,z):=\frac{1}{(z-x)(z-y)}\langle V_{-\frac{\gamma}{2},\epsilon}(z)V_{\gamma,\epsilon}(x) V_{\gamma,\epsilon}(y)  \prod_{k=1}^N  V_{\alpha_k,\epsilon}(z_k)\rangle_{\epsilon}.$$

 If $x,y$ belong to a ball $B(0,\delta^{-1})$ then integrability in a neighborhood of the diagonal and away from the points $\{(w,w);w\in\{z,z_1,\dots,z_N)\}$ results from Proposition \ref{upperbound2} (with $N+1$ insertions at $z,z_1,\dots,z_N$ instead of $N$ insertions). We let the reader check this point (distinguish the cases $2\gamma>Q$, $2\gamma=Q$ and $2\gamma<Q$). Integrability when $x$ and/or $y$ approach $\infty$ results from Propositions \ref{lpinfty} and \ref{2merginfinity}.

We are left with the case when both $x$ and $y$ get close to $z$. This situation is described by Proposition \ref{upperbound3}. We proceed as in the proof of Proposition  \ref{1and2point} but this time paying attention to the fact that one weight is negative. 
There are once again two cases depending on $2\gamma\geq Q$ or $2\gamma<Q$. The most intricate  is the case $2\gamma\geq Q$ and this is the only case that we will discuss. Proposition \ref{upperbound3} with $r=0$ and $\lbrace \beta_1,\beta_2,\beta_3\rbrace=\lbrace \gamma, -\frac{\gamma}{2} \rbrace $ thus gives:

\medskip

$\bullet$ $|x-y|\leq |x-z|\leq |y-z|$ :
$$|\phi_\epsilon(x,y,z)|\leq 
 C_\delta |x-y|_\epsilon^{-4+\frac{Q^2}{2}}\frac{|y-z|_\epsilon^{1+\frac{\gamma^2}{4} } }{|x-z||y-z|}
 $$
 
\medskip

$\bullet$ $  |x-z|\leq |y-z|\leq |x-y| $: 
$$|\phi_\epsilon(x,y,z)|\leq C_\delta \frac{|x-y|_\epsilon^{-\gamma^2}}{|x-z||y-z|}
 \left\{\begin{array}{ll} 
|y-z|_\epsilon^{{\gamma^2}+\frac{(Q-3\gamma/2)^2}{2}} ,&\text{ if }\quad 3\gamma/2 \geq Q,\\
 |y-z|_\epsilon^{\gamma^2}      ,&\text{ if }\quad   3\gamma/2< Q .                                                                      
 \end{array}\right..
 $$
 
\medskip

$\bullet$  $  |x-z| \leq |x-y| \leq |y-z|$:
$$|\phi_\epsilon(x,y,z)|\leq 
 \left\{\begin{array}{ll} 
C_\delta\frac{|x-z|_\epsilon^{\frac{\gamma^2}{2}}}{|x-z|}|x-y|^{-\gamma^2+\frac{(Q-3\gamma/2)^2}{2}} \frac{(\epsilon+|y-z|)^{\frac{\gamma^2}{2}}}{|y-z|}   ,&\text{ if }\quad  3\gamma/2\geq Q,\\
     C_\delta\frac{ |x-z|_\epsilon^{\frac{\gamma^2}{2}} }{|x-z|}   (\epsilon+ |x-y|)^{- \gamma^2}      \frac{|y-z|_\epsilon^{\frac{\gamma^2}{2}}}{|y-z|}                        ,&\text{ if }\quad  3\gamma/2< Q .                                                                                         
 \end{array}\right.
 $$
Integrability of $\sup_\epsilon|\phi_\epsilon(x,y,z)|$ is now readily checked.
\medskip

In the same way,
for the $-\frac{2}{\gamma}$-insertion, i.e.  when considering the mapping
$$(x,y)\mapsto \psi_\epsilon(x,y,z):=\frac{1}{(z-x)(z-y)}\langle V_{-\frac{2}{\gamma} ,\epsilon}(z)V_{\gamma,\epsilon}(x) V_{\gamma,\epsilon}(y)  \prod_{k=1}^N  V_{\alpha_k,\epsilon}(z_k)\rangle_{\hat{g},\epsilon},$$ we get the following bounds:

\medskip

$\bullet$  $|x-y|\leq |x-z|\leq |y-z|$: 
$$|\psi_\epsilon(x,y,z)|\leq 
 C_\delta |x-y|_\epsilon^{-4+\frac{Q^2}{2}}\frac{|x-z|_\epsilon^{ \frac{4}{\gamma^2}-1 } }{|x-z|}\frac{|y-z|_\epsilon^{2}}{|y-z|} , $$

$\bullet$ $  |x-z|\leq |y-z|\leq |x-y| $: 
$$|\psi_\epsilon(x,y,z)|\leq 
 \left\{\begin{array}{ll} 
 C_\delta\frac{|x-z|_\epsilon^{2} }{|x-z|}\frac{ |y-z|_\epsilon^{2+\hf(Q-2\gamma+\frac{2}{\gamma})^2}}{|y-z|}( |x-y|_\epsilon)^{-\gamma^2}   ,&\text{ if }\quad 2\gamma-\frac{2}{\gamma} \geq Q,\\
 C_\delta  \frac{|x-z|_\epsilon^{2} }{|x-z|}\frac{|y-z|_\epsilon^{2}}{|y-z|}  |x-y|_\epsilon^{-\gamma^2}                                       ,&\text{ if }\quad   3\gamma/2< Q .                                                                      
 \end{array}\right..
 $$

$\bullet$ $  |x-z| \leq |x-y| \leq |y-z|$: 
$$|\psi_\epsilon(x,y,z)|\leq 
 \left\{\begin{array}{ll} 
C_\delta\frac{|x-z|_\epsilon^{\frac{\gamma^2}{2}}}{|x-z|}|x-y|_\epsilon^{-\gamma^2+\frac{(Q-2\gamma+\frac{2}{\gamma})^2}{2}} \frac{|y-z|_\epsilon^{\frac{\gamma^2}{2}}}{|y-z|}   ,&\text{ if }\quad  2\gamma-\frac{2}{\gamma}\geq Q,\\
     C_\delta\frac{ |x-z|_\epsilon^{2} }{|x-z|}   |x-y|_\epsilon^{- \gamma^2}      \frac{|y-z|_\epsilon^{2}}{|y-z|}                        ,&\text{ if }\quad  2\gamma-\frac{2}{\gamma}< Q .                                                                                         
 \end{array}\right.
 $$
Again, integrability follows.\qed

\section{Proof of Fusion Estimates}\label{sec:prem}

\subsection{Regularized correlations and Proof of Lemma  \ref{trinv}}\label{regcor}

The regularized version of \eqref{Z1} reads
\begin{equation} 
G_\epsilon(\z)=B(\pmb{\alpha})
 e^{\sum_{j < k}^N \alpha_j \alpha_kC_\epsilon(z_j-z_k)}\mu^{-s} \gamma^{-1}\Gamma(s)\E ( {Z}_\epsilon(\z)^{-s}  )\label{Z1reg}
\end{equation}
where
\begin{equation}\label{Z1rega}
Z_\epsilon(\z)=  (A\epsilon)^{\frac{\gamma^2}{2}} \int_{\C}   
e^{\gamma\sum_{k=1}^N \alpha_k  C_\epsilon(x-z_k)} \hat{g}_\epsilon(x)^{ 
-\frac{\gamma^2}{4}s}e^{\gamma X_\epsilon(x)} d^2x,
\end{equation}
$C_\epsilon(z)=
(\ln\frac{1}{|z|})_{\epsilon,\epsilon}$, $\ln \hat g_\epsilon=\rho_\epsilon\ast \ln \hat g$
 and $s=(\sum_{k=1}^N \alpha_k-2Q)/\gamma$. We have the estimate
$$
\ln{|x|_\epsilon}^{-1}-C\leq {C_\epsilon(x)}\leq \ln{|x|_\epsilon}^{-1}+C
$$
where $C$ is uniform in $\epsilon$ and $|x|_\epsilon:=|x|\vee\epsilon$.  
Moreover $\|\partial_x^n G_\epsilon\|_\infty<\infty$ if  $\epsilon>0$. From this  we deduce  the smoothness of the correlations if $\epsilon>0$ claimed in Proposition \ref{1and2point} {(a)}.

\medskip
\noindent{\it Proof of Lemma \ref{trinv}.} From the expression \eqref{Z1reg}, it suffices to prove $\E  {Z}_\epsilon(\z+y)^{-s}=\E {Z}_\epsilon(\z)^{-s}  $. We have
\begin{equation*}
{Z}_\epsilon(\z+y)
=(A\epsilon)^{\frac{\gamma^2}{2}} \int_{\C}
\prod_{k=1}^N 
e^{\gamma \alpha_kC_\epsilon(x-z_k)} \hat{g}_\epsilon(x+y)^{ 
-\frac{\gamma^2}{4}s}e^{\gamma X_\epsilon(x+y)} d^2x.
\end{equation*}
Now $X(\cdot+y)$ is equal in law to  $X(\cdot)-m_y(X)$ where $m_y(X):=\frac{1}{4\pi}\int_{\C} X(x)\hat g(x+y)d^2x$. Hence, making the change of variables $x\mapsto x+y$ in the integral, we get
\begin{equation*}
\E {Z}_\epsilon(\z+y)^{-s}=\E e^{ \gamma s m_y(X)  }
\Big((A\epsilon)^{\frac{\gamma^2}{2}} \int_{\C}
\prod_{k=1}^N 
e^{\gamma \alpha_kC_\epsilon(x-z_k)} \hat{g}_\epsilon(x+y)^{ 
-\frac{\gamma^2}{4}s}e^{\gamma X_\epsilon(x)} d^2x\Big)^{-s}
\end{equation*}
By the Girsanov theorem  this equals
\begin{equation*}
\E {Z}_\epsilon(\z+y)^{-s}=e^{D}
\E
((A\epsilon)^{\frac{\gamma^2}{2}}  \int_{\C}
\prod_{k=1}^N 
e^{\gamma \alpha_kG_\epsilon(x-z_k)} \hat{g}_\epsilon(x+y)^{ 
-\frac{\gamma^2}{4}s}e^{\gamma^2s\, \rho_\epsilon\ast h(x)}e^{\gamma X_\epsilon(x)} d^2x)^{-s}
\end{equation*}
where $D$ is the variance of the Girsanov transform, namely $$D:=\hf \gamma^2s^2(4\pi)^{-2}\int_{\C}\int_{\C} G(u,v)\hat g(u+y)\hat g(v+y)d^2ud^2v,$$ and $h$ the shift, i.e.  $$h(x)=(4\pi)^{-1}\int_{\C} G(x,u) \hat g(u+y) d^2u.$$ 
It is proved in  \cite[proof of theorems 3.5 and 3.7]{DKRV} that $h(x)=\frac{_1}{^4}(\ln\frac{\hat g(x+y)}{\hat g(x)} -\int_{\C}\ln \frac{\hat g(u+y)}{\hat g(u)}\hat g(u) d^2u)$ so that $\rho_\epsilon\ast h (x)=\frac{_1}{^4}(\ln\frac{\hat g_\epsilon(x+y)}{\hat g_\epsilon(x)}-\int\ln \frac{\hat g(u+y)}{\hat g(u)}\hat g(u) d^2u)$ and then
\begin{equation*}
\E {Z}_\epsilon(\z+y)^{-s}=e^{D+\gamma^2s^2\int\ln \frac{\hat g(u+y)}{\hat g(u)}\hat g(u) d^2u}\E
((A\epsilon)^{\frac{\gamma^2}{2}}  \int_{\C}
\prod_{k=1}^N 
e^{\gamma \alpha_kC_\epsilon(x-z_l)} \hat{g}_\epsilon(x)^{ 
-\frac{\gamma^2}{4}s}e^{\gamma X_\epsilon(x)} d^2x)^{-s}
\end{equation*}
The claim follows from $D=-\gamma^2s^2\int_{\C}\ln \frac{\hat g(u+y)}{\hat g(u)}\hat g(u) d^2u$, see again \cite[proof of theorems 3.5 and 3.7]{DKRV}.\qed

\subsection{Radial decomposition of  the  chaos measure}
Here we prepare for the proof of the fusion estimates. As the correlation functions are translationally invariant, we may assume that the points merge at $0$. We will use some decomposition of the correlation functions developed in  \cite{DKRV2}, which we summarize now. 

First, we want to trade the GFF $X$ for a field that is more appropriate to a local analysis around $0$. By shifting the mean of the GFF $X$, we can replace the GFF $X$ in \eqref{Liouville measure} by the GFF $X_0$ with vanishing mean on the unit circle, i.e. the Gaussian distribution with covariance structure
\begin{equation}\label{G0}
G_0(x,y)=\ln\frac{1}{|x-y|}+\ln|x|\mathbf{1}_{\{|x|\geq 1\}}+\ln|y|\mathbf{1}_{\{|y|\geq 1\}}.
\end{equation}
This covariance kernel is of exact log type in the ball $B(0,1)$, hence facilitates the analysis around $0$.  Furthermore, we write  $G_{0,\epsilon}$ for the mollification (with $\rho_\epsilon$) of the kernel $G_0$ with respect to both variables $x$ and $y$ and we introduce the function 
$$\tilde{H}(x):=\ln|x|1_{\{|x|\geq 1\}}-\tfrac{1}{4}\ln\hat{g}(x),$$
which is bounded over $\C$. $\tilde{H}_\epsilon$ will stand for the mollification of $\tilde{H}$ with respect  to $\rho_\epsilon$.

The regularized correlation functions then read
\begin{equation}\label{correl0}
\langle \prod_{k=1}^NV_{\alpha_k,\epsilon}(z_k)\rangle_\epsilon=
K_\epsilon(\mathbf{z})\mu^{-s}\Gamma(s) \gamma^{-1}\E\big[\big(\int_{\C}e^{\gamma H^0_\epsilon}dM^0_{\gamma,\epsilon}\big)^{-s}\big]
\end{equation}
where     
\begin{equation}\label{def:Hepsilon}
H^0_\epsilon(x)=\tilde{H}_\epsilon(x)+ \sum_k\alpha_kG_{0,\epsilon}(z_k,x) 
\end{equation}
\begin{equation}\label{def:Kepsilon}
K_\epsilon(\mathbf{z})=4 e^{\chi\big(Q^2+\sum_k\tfrac{\alpha_k^2}{2}\big)}e^{\sum_{k,k'}\alpha_k\alpha_{k'}G_{0,\epsilon}(z_k,z_{k'})+\sum_{k}\alpha_k\tilde{H}_\epsilon(z_k)}\Big(\prod_{k=1}^N\hat{g}(z_k)^{\triangle_{\alpha_k}}\Big)(1+o(1)),
\end{equation}
with $o(1)$ uniform on $\C$ as $\epsilon\to 0$, and the regularized potential is given by
\begin{align*}
M^0_{\gamma,\epsilon}(d^2x):=&(A\epsilon)^{\frac{\gamma^2}{2}}e^{\gamma (X_{0,\epsilon}(x)+\tfrac{Q}{2}\ln\hat{g}_\epsilon(x))}\,d^2x\\
 M^0_{\gamma}(d^2x):=&\lim_{\epsilon\to 0}M^0_{\gamma,\epsilon}(d^2x)\\
 =&A^{\frac{\gamma^2}{2}}e^{\gamma X_0(x)-\frac{\gamma^2}{2}\E[X_0^2(x)]}(|x|\vee 1)^{\gamma^2}\hat{g}(x)^{\frac{\gamma Q}{2}}\,d^2x .
 \end{align*}

We  denote  by $\mathcal{F}_\eta$ ($\eta>0$) the sigma-algebra generated by the field $X_0$ "away from the disc $B(0,\eta)$", namely
\begin{equation}
\mathcal{F}_\eta=\sigma\{X_0(f); f\text{ smooth}, {\rm supp}(f)\in B(0,\eta)^c\}.
\end{equation}
$\mathcal{F}_\infty$ stands for the sigma algebra generated by $\bigcup_{\eta>0}\mathcal{F}_\eta$. 

The following radial decomposition of the field $X_0$ will be useful for the analysis (this observation was already made in \cite{DMS}):

\begin{lemma}\label{independence} The field $X_0$  may be decomposed (in the sense of distributions) as 
\begin{equation}
 X_0(z)= X^r(|z|) + Y(z)
\end{equation} 
where the centered Gaussian fields $X^r,Y$ are independent have the following covariance
\begin{equation*}
\E[    Y(z) Y(z') ]  = \ln \frac{ |z| \vee |z'|}{|z-z'|}\quad \text{ and }\quad \E[X^r(|z|)X^r(|z'|)]=\ln\frac{1}{ |z| \vee |z'|}+\ln|z|\mathbf{1}_{\{|z|\geq 1\}}+\ln|z'|\mathbf{1}_{\{|z'|\geq 1\}}.
\end{equation*}
In particular, the process $t\mapsto X^r(e^{-t})$  evolves as a Brownian motion independent of $\mathcal{F}_1$.
\end{lemma}

Using the above lemma, we get the following decomposition of the chaos measure
$$
M^0_\gamma(d^2x)= c_\gamma  \hat g(x)  |x|^{\frac{\gamma^2}{2}}e^{\gamma X^r(|x|)}\,M_\gamma(d^2x,Y)
$$
 where $M_\gamma(d^2x,Y)$ is the multiplicative chaos measure of the field $Y$ with respect to the Lebesgue measure (i.e. $\E\, M_\gamma(d^2x,Y)=d^2x$) and $c_\gamma:= A^{\frac{\gamma^2}{2} }$ is a constant.

If $z$ belongs to the unit disk $\D$, we make change of variables $z=e^{-s+i\sigma}$, $s\in \R_+$, $\sigma\in [0,2\pi)$ and let $\mu_{Y}(ds,d\sigma)$ be  the multiplicative chaos measure of the field $Y(e^{-s+i\sigma})$ with respect to the measure $dsd\sigma$. We denote by $x_s$ the process
\begin{equation}\label{radialpr}
s\in \R_+\to x_s:=X^r(e^{-s}).
\end{equation}
 We  arrive  at the following useful "radial" decomposition of the chaos measure:  

\begin{lemma}\label{decompmeas} Let $f\in C_0(\D)$. Then
$$
 \int_{\D} f(x)\,M^0_\gamma(d^2x) =c_\gamma \int_0^\infty \int_0^{2\pi}f(e^{-s}e^{i\sigma})e^{\gamma x_s-\gamma Q s}\hat g(e^{-s})\,\mu_{Y}(ds,d\sigma)$$
where 
$\mu_{Y}(ds,d\sigma)$ is a measure independent of the whole process $(x_s)_{s\geq 0}$. 
\end{lemma}

\subsection{On the supremum of a drifted Brownian motion}
With Lemma \ref{decompmeas} in mind, we will  start with  two elementary  lemmas on Brownian motion with drift.  
 \begin{lemma}\label{lem:renorm} Let $B$ be a standard Brownian motion staring from $0$.\\
1) For $\beta,\alpha> 0$,  we have
$$\P(\sup_{u\leq t}B_u+\alpha u\leq \beta)\leq \frac{e^{-\frac{\alpha^2 t}{2}}}{ \alpha^2 t^{3/2}} \sqrt{\frac{2}{\pi}} \beta e^{\alpha\beta}.  $$
2) For $\beta > 0$,  we have
$$\P(\sup_{u\leq t}B_u\leq \beta)=\sqrt{\frac{2}{\pi}}\int_0^{\frac{\beta}{\sqrt{t}}} e^{- \frac{u^2}{2}}\,du\leq \sqrt{\frac{2}{\pi}} \frac{\beta}{\sqrt{t}}.$$
\end{lemma}
\begin{proof} 
 The density of $\sup_{u \leq t} (B_u+\alpha u) $ is explicit and one has (see \cite{BoSal} for example)
\begin{align*}
\P(\sup_{u\leq t}B_u+\alpha u\leq \beta) & = \frac{\beta}{\sqrt{2 \pi}}  \int_t^{\infty}  \frac{e^{- \frac{(\beta-\alpha s)^2}{2s}}}{s^{3/2}}  ds =    \frac{\beta e^{\beta \alpha}}{\sqrt{2 \pi}}  \int_t^{\infty}  \frac{e^{-\frac{\beta^2}{2 s}} e^{- \frac{\alpha^2 s}{2}}}{s^{3/2}}  ds \leq \frac{e^{-\frac{\alpha^2 t}{2}}}{ \alpha^2 t^{3/2}} \sqrt{\frac{2}{\pi}} \beta e^{\alpha\beta} .
\end{align*}

\end{proof}
%
For the next lemma  we fix some constants $\alpha,\tilde{\alpha}\in\R$ and   $s>0$ and define for $u\geq 0$
\begin{equation}\label{Fdefini} F(u):=\alpha u+\tilde{\alpha}\int_0^u  \mathbf{1}_{[0,s]}(v)\,dv
=\alpha u+\tilde\alpha \,u\wedge s.
\end{equation}
\begin{lemma}\label{warrior1}
We have the following estimates for $\beta\geq 0$ and $s<r$:\\
1) if $\alpha\geq 0$, $\tilde\alpha\geq 0$ then
$$\P(\sup_{u\in[0,r]}(B_u+F(u)) \leq \beta)\leq e^{-\frac{(\alpha+\tilde{\alpha})^2}{2}s-\frac{ \alpha^2}{2}(r-s)} e^{(\alpha  +\tilde{\alpha})\beta}.$$
2) if $\alpha \geq 0$, $\tilde\alpha\leq 0$ then 
$$\P(\sup_{u\in[0,r]}(B_u+F(u))\leq \beta)\leq e^{-\frac{ (\alpha+\tilde{\alpha})^2-\alpha^2-\tilde\alpha^2}{2}s-\frac{ \alpha^2}{2}r} e^{\alpha \beta}.$$
\end{lemma}

\begin{proof}Let us set $f(u)=\alpha  +\tilde{\alpha} \mathbf{1}_{[0,s]}(u)$. Using the Girsanov transform  yields
\begin{align}
\P(\sup_{u\in[0,r]} (B_u+F(u)) \leq \beta)= & \E\Big[e^{\int_0^rf(u)\,dB_u-\frac{1}{2}\int_0^rf(u)^2\,du}\mathbf{1}_{\{\sup_{u\in[0,r]}B_u\leq \beta\}}\Big]\nonumber\\
= &   e^{-\frac{(\alpha+\tilde{\alpha})^2}{2}s-\frac{\alpha^2}{2}(r-s)}\E\Big[e^{\alpha B_r+\tilde{\alpha} B_s}\mathbf{1}_{\{\sup_{u\in[0,r]}B_u\leq \beta\}}\Big].\label{heart}
\end{align}
 1) follows then by using $B_r\leq \beta$, $B_s\leq\beta$
and for  2), we plug the following estimate into \eqref{heart}
\begin{align*}
\E\Big[e^{\alpha B_r+\tilde{\alpha} B_s}\mathbf{1}_{\{\sup_{u\in[0,r]}B_u\leq \beta\}}\Big]\leq &e^{\alpha\beta}\E\Big[e^{\tilde{\alpha} B_s}\mathbf{1}_{\{\sup_{u\in[0,s]}B_u\leq \beta\}}\Big]\\
=&e^{\alpha\beta+\frac{\tilde\alpha^2}{2}s}\E\Big[e^{\tilde{\alpha} B_s-\frac{\tilde\alpha^2}{2}s}\mathbf{1}_{\{\sup_{u\in[0,s]}B_u\leq \beta\}}\Big]\\
=&e^{\alpha\beta+\frac{\tilde\alpha^2}{2}s}\E\Big[ \mathbf{1}_{\{\sup_{u\in[0,s]}B_u+\tilde\alpha u\leq \beta\}}\Big]
\leq  e^{\alpha\beta+\frac{\tilde\alpha^2}{2}s}.
\end{align*}
 \end{proof}

\subsection{Main technical lemma}
 We are now in position to state the main technical lemma  of this section. 
Let $x_u$ be the radial process \eqref{radialpr} and set 
$$y_u=x_u+  F(u).
$$
where the drift $F(u)$ is defined in \eqref{Fdefini}. Then

\begin{lemma}\label{corner}
Let   $q>0$, $\beta\geq 0$ and $r\geq s$ and define
 $$E_{r,\beta}:=\E\Big[\frac{\mathbf{1}_{\{\sup_{u\in [0,r]}y_u\in [\beta-1;\beta]\}}}{\Big( \int_{0}^{r}\int_0^{2\pi} e^{\gamma y_s} \,\mu_{Y}(ds,d\sigma) \Big)^{q}}\Big].
 $$
 Then we have the following estimates depending on the values of the parameters $\alpha,\tilde\alpha$:

\vskip 2mm

\noindent 1) if $\alpha> 0$ and $\tilde\alpha=0$
$$
E_{r,\beta}
 \leq C(\beta+1)  e^{(\alpha-q \gamma) \beta} r^{-3/2}e^{-\frac{\alpha^2}{2}r}.$$
2)   if $\alpha=\tilde\alpha= 0$ then
$$
E_{r,\beta}
\leq C(\beta  +1)e^{-q \gamma \beta}r^{-1/2}. $$
3) if $\alpha \geq 0$, $\tilde\alpha\geq 0$
$$
E_{r,\beta} \leq C e^{(\alpha+\tilde\alpha-q\gamma)\beta} e^{-\frac{(\alpha+\tilde\alpha)^2}{2}s-\frac{\alpha^2}{2}(r-s)}.$$
4) if $\alpha\geq 0$, $\tilde\alpha<0$ then
$$
E_{r,\beta}
\leq C e^{(\alpha -q\gamma)\beta} e^{-\frac{(\alpha+\tilde\alpha)^2-\alpha^2-\tilde\alpha^2}{2}s-\frac{\alpha^2}{2}r}.$$
\end{lemma}

\noindent {\it Proof.}   We  introduce the stopping time 
$$T_{\beta}=\inf\{s\geq 0;y_u\geq \beta-1\}.$$
 Then
\begin{align*}
E_{r,\beta} &=\E\Big[\frac{\mathbf{1}_{\{T_\beta<r-1\}}\mathbf{1}_{\{\sup_{u\in [0,r]}y_u\in [\beta-1;\beta]\}}}{\Big( \int_{0}^{r}\int_0^{2\pi} e^{\gamma y_s} \,\mu_{Y}(ds,d\sigma) \Big)^{q}}\Big] +\E\Big[\frac{\mathbf{1}_{\{T_\beta\geq r-1\}}\mathbf{1}_{\{\sup_{u\in [0,r]}y_u\in [\beta-1;\beta]\}}}{\Big( \int_{0}^{r}\int_0^{2\pi} e^{\gamma y_s} \,\mu_{Y}(ds,d\sigma) \Big)^{q}}\Big]\nonumber\\
&\leq \E\Big[\frac{\mathbf{1}_{\{T_\beta<r-1\}}\mathbf{1}_{\{\sup_{u\in [0,r]}y_u\in [\beta-1;\beta]\}}}{e^{\gamma q y_{T_\beta}}I(T_\beta)^{q}}\Big] +\E\Big[\frac{\mathbf{1}_{\{T_\beta\geq r-1\}}\mathbf{1}_{\{\sup_{u\in [0,r]}y_u\in [\beta-1;\beta]\}}}{e^{\gamma q y_{r-1}}I(r-1)^{q}}\Big]=:A+B,
\end{align*}
where we have set $$I(a)=\int_{a}^{a+1}\int_0^{2\pi} e^{\gamma  (y_s-y_a)}\,\mu_Y(ds,d\sigma).$$ 
We only  treat $ A$ because the same argument holds for $B$. Obviously,
\begin{align}\label{tam}
A\leq &  e^{-q\gamma (\beta-1)} \E\Big[
\mathbf{1}_{\{\max_{u\in [T_\beta+1,r]}y_u-y_{T_\beta+1}\leq \beta-y_{T_\beta+1}\}}  \frac{\mathbf{1}_{\{T_\beta+1<r\}} }{I(T_\beta)^{q}} \Big]. 
 \end{align}
 By the strong Markov property of the Brownian motion, we get
\begin{align*}
 \E\Big[\mathbf{1}_{\{\max_{u\in [T_\beta+1,r]}y_u-y_{T_\beta+1}\leq \beta-y_{T_\beta+1}\}}  \frac{\mathbf{1}_{\{T_\beta+1<r\}} }{I(T_\beta)^{q}} \Big]=& \E\Big[\mathbf{1}_{\{T_\beta+1<r\}}H(\beta-y_{T_\beta+1})\E\Big[   \frac{1 }{I(T_\beta)^{q}} |\beta-y_{T_\beta+1} \Big]\Big]
 \end{align*} 
 where we have set $H(t)=\E[\mathbf{1}_{\{\max_{u\in [T_\beta+1,r]}y_u-y_{T_\beta+1}\leq t\}} ]$. Now we use the  estimate 
\begin{equation}\label{bridge}
\E[I(a)^{-q}|y_{a+1}-y_a]\leq C\big(e^{- \gamma q(y_{a+1} -y_a)}+1\big),
\end{equation}
 which has been proven \cite[Lemma 6.1]{DKRV2}. Denote $\mathcal{F}_t=\sigma\{x_u;u\leq t\}$. Conditionally on $\mathcal{F}_{T_\beta}$, the random variable $\{y_{T_\beta+1}-y_{T_\beta}\}$ is a  Gaussian random variable with variance $1$ and mean $F(T_{\beta}+1)-F(T_\beta)$. Furthermore,  observe that the random variable  $F(T_{\beta}+1)-F(T_\beta)$ is bounded, indeed $|F(T_{\beta}+1)-F(T_\beta)|\leq |\alpha|+|\tilde\alpha|$. Let us set $c:= |\alpha|+|\tilde\alpha|$. We deduce
 \begin{align*}
A\leq & Ce^{-q\gamma \beta}  \E\big[  H(1-N+c) \big(e^{-q \gamma N+q \gamma c}+1\big) \big]
 \end{align*} 
 where $N$ is a centered standard Gaussian variable independent of everything. This can be estimated as
  \begin{align*}
A\leq &  Ce^{-q\gamma \beta} \E\Big[ \mathbf{1}_{\{\max_{u\in [0,r-1]}y_u\leq \beta+\max(0,N+3c+1)\}}  \big(e^{-q \gamma N+q \gamma c}+1\big) \big] \Big] .
 \end{align*}
 It suffices now to combine with the various items of  Lemmas \ref{lem:renorm}  and \ref{warrior1} depending on the values of $\alpha,\tilde\alpha$ to complete the proof.\qed

\subsection{Proof of Proposition \ref{upperbound2}}

Let us denote by $\caA(x,r)$ the annulus with center $x$, inner radius $r<1$ and outer radius $1$.
Let $\caA=\caA(x_1,|x_1-x_2|_\epsilon)$.   From \eqref{correl0} and \eqref{def:Kepsilon} we get
\begin{align}\label{key}
\langle  V_{\beta_1,\epsilon}(x_1)V_{\beta_2,\epsilon}(x_2) \prod_k  V_{\alpha_k,\epsilon  }(z_k)\rangle_\epsilon
 \leq C(\delta)\hat{g}(z_2)^{\Delta_{\alpha_2}}|x_1-x_2|_\epsilon^{-\beta_1\beta_2}I_{\caA}
 \end{align}
where $q=(\sum\alpha_i+\beta_1+\beta_2-2Q)/\gamma$ and
\begin{align}\label{keya}
I_{\caA}=\E\Big[\big(\int_{\caA}\frac{1}{|x-x_1|_\epsilon^{\gamma \beta_1}|x-x_2|_\epsilon^{\gamma\beta_2}}M^0_{\gamma,\epsilon} (d^2x)\big)^{-q}\Big].
\end{align} 
We have   obtained a lower bound for the expectation in \eqref{correl0} by restricting the integral in the expectation to $\caA$ and then used the fact that the part of the integrand that we dropped  is bounded from below by a $\delta$ dependent constant, uniformly in $\epsilon$ if $x\in \caA$.   
Furthermore, for  $x\in \caA $,  we have $|x-x_2|\leq 2|x-x_1|$ so that
\begin{align}\label{ret1}
I_{\caA}
\leq C\, \E\Big[\big(\int_{\caA}\frac{1}{|x-x_1|_\epsilon^{\gamma(\beta_1+\beta_2)}}M^0_{\gamma,\epsilon} (d^2x)\big)^{-q}\Big].
\end{align} 

It is convenient at this point to work with the measure $M^0_\gamma$ instead of its regularization $M^0_{\gamma,\epsilon}$.  We have 
$$
\E X_{0,\epsilon}(u)X_{0,\epsilon}(v)\leq C+\E X_{0}X_{0}(v)
$$
where $C$ is uniform in $\epsilon$. Hence
by Kahane convexity \eqref{ret1} holds if we replace $M^0_{\gamma,\epsilon}$ by $M^0_{\gamma}$ in this relation (with a different constant $C$).

The expectation in the right-hand side of  \eqref{ret1}  may now be written in terms of 
 the radial decomposition of the chaos measure  Lemma \ref{decompmeas}:
 $$
  \E\Big[\big(\int_{\caA}\frac{1}{|x-x_1|_\epsilon^{\gamma(\beta_1+\beta_2)}}M^0_{\gamma} (d^2x)\big)^{-q}\Big]= \E\Big[\big(\int_{0}^{-\ln|x_1-x_2|_\epsilon}\int_0^{2\pi} e^{\gamma y_s} \,\mu_{Y}(ds,d\sigma)\big)^{-q}\Big]
 $$
 where  
 $$ y_s=x_s+(\beta_1+\beta_2-Q)s .$$
Now we partition the probability space according to the values of the maximum of $y_s$. Define
\begin{align}
M_{n}= &\big\{\max_{s\in[0,-\ln|x_1-x_2|_\epsilon]} y_s \in [n-1,n]\big\},\,\,\,\,n\geq 1,\label{def:mn}\\
M_{0}= &\big\{\max_{s\in[0,-\ln|x_1-x_2|_\epsilon]} y_s  \leq 0\big\}\label{def:m0}.
\end{align}
Then
\begin{align}\label{ret12}
I_{\caA}
\leq C\sum_{n\geq 0} \E\Big[\mathbf{1}_{M_{n}}\big(\int_{0}^{-\ln|x_1-x_2|_\epsilon}\int_0^{2\pi} e^{\gamma y_s} \,\mu_{Y}(ds,d\sigma)\big)^{-q}\Big] :=C\sum_{n\geq 0}A_n.
\end{align}
By Lemma \ref{corner} item 1, with $\alpha=\beta_1+\beta_2-Q>0$, $\tilde{\alpha}=0$, $q= \frac{\beta_1+\beta_2+\sum_i\alpha_i-2Q}{\gamma}$ and  $\beta=n$, we deduce
$$A_n\leq C(n+1) e^{-(\sum_i\alpha_i-Q)n}|\ln|x_1-x_2|_\epsilon|^{-3/2}|x_1-x_2|_\epsilon^{\frac{(\beta_1+\beta_2-Q)^2}{2}}.$$
Combining this with \eqref{ret12} and \eqref{key} we arrive at the claim since  $\frac{(\beta_1+\beta_2-Q)^2}{2}-\beta_1\beta_2=2(\Delta_{(\beta_1+\beta_2)\wedge Q}-\Delta_{\beta_1}-\Delta_{\beta_2})$ if $\beta_1+\beta_2>Q$. 

For $\beta_1+\beta_2=Q$ we use Lemma \ref{corner} 2) and for $\beta_1+\beta_2<Q$ $I_{\caA}$ is uniformly bounded in $\epsilon$.\qed

\subsection{Proof of Proposition \ref{upperbound3}}
We proceed as in the previous section, starting with
\begin{align}\label{key1}
\langle  V_{\beta_1,\epsilon}(x_1)V_{\beta_2,\epsilon}(x_2) V_{\beta_3,\epsilon}(x_3) \prod_k  V_{\alpha_k,\epsilon  }(z_k)\rangle_\epsilon
 \leq C(\delta)|x_1-x_2|_\epsilon^{-\beta_1\beta_2}|x_1-x_3|_\epsilon^{-\beta_1\beta_3}|x_2-x_3|_\epsilon^{-\beta_2\beta_3}I_{\caA}
 \end{align}
where 
\begin{align}\label{qdefi}
q=(\sum_k\alpha_k+\beta_1+\beta_2+\beta_3-2Q)/\gamma
 \end{align}
 and
\begin{align}\label{keya}
I_{\caA}=\E\Big[\big(\int_{\caA}\frac{1}{|x-x_1|_\epsilon^{\gamma \beta_1}|x-x_2|_\epsilon^{\gamma\beta_2}|x-x_3|_\epsilon^{\gamma \beta_3}}M_\gamma^0 (d^2x)\big)^{-q}\Big].
\end{align} 
The choice of the region $\caA$ will depend on the $\beta_i$'s.  We need to consider several cases.

\medskip

\noindent {\bf 1. }$\beta_1+\beta_2\geq Q$, $\beta_3\geq 0$. 
We take $\caA=\caA(x_1,|x_1-x_2|_\epsilon)$. 
 Inserting the estimates, valid  for $x\in \caA$,
\begin{equation} \label{zxybounds}
|x-x_2|_\epsilon \leq 2|x-x_1|_\epsilon,
\quad \quad |x-x_3|_\epsilon\leq 2|x_1-x_3|_\epsilon\mathbf{1}_{\{|x-x_1|_\epsilon\leq |x_1-x_3|_\epsilon}+2|x-x_1|_\epsilon \mathbf{1}_{\{|x-x_1|_\epsilon> |x_1-x_3|_\epsilon}
\end{equation}
into \eqref{key1} and then use the polar decomposition of the chaos measure around $0$ we deduce  that the expectation in the right-hand side of \eqref{key1} is bounded by
\begin{align}\label{ret13}
&I_{\caA}\leq C\sum_{n\geq 0}\E\Big[\mathbf{1}_{M_{n}}\big(\int_{0}^{-\ln|x_1-x_2|_\epsilon}\int_0^{2\pi} e^{\gamma y_s} \,\mu_{Y}(ds,d\sigma)\big)^{-q}\Big]:=
 C\sum_{n\geq 0}A_n
\end{align}
 where $y_s$ is the process
\begin{align}
y_s=x_s+(\beta_1+ \beta_2-Q)s+\beta_3 \,s\wedge\ln{|x_1-x_3|_\epsilon}^{-1}\end{align}
and $M_{n}$ is as in \eqref{def:mn} and \eqref{def:m0}. 

We can now apply item 3 of Lemma \ref{corner} with $\alpha=\beta_1+\beta_2-Q$, $\tilde\alpha=\beta_3$, $\beta=n$, $r=\ln|x_1-x_2|_\epsilon^{-1}$ and $s =\ln|x_1-x_3|_\epsilon^{-1}$  and $q$ as in \eqref{qdefi} to bound
\begin{equation} \label{anbound}
A_n \leq     C  e^{-n(\sum_i\alpha_i-Q)}|x_1-x_2|_\epsilon^{\frac{(\beta_1+\beta_2 -Q)^2}{2}}|x_1-x_3|_\epsilon^{\frac{(\beta_1+\beta_2+\beta_3-Q)^2}{2}-\frac{(\beta_1+\beta_2-Q)^2}{2}}   .
\end{equation}
Combining \eqref{key1},\eqref{ret13} and \eqref{anbound} with $|x_2-x_3|\geq |x_1-x_3|$ the claim follows.

\medskip
 
\noindent {\bf 2.} $\sum_i\alpha_i+\beta_3-Q>0$, $\beta_1+\beta_2\geq Q$, $\beta_3<0$.
We replace $\caA$ in the   previous case by the half annulus $\caA':=\caA\cap \{x: \arg(\frac{x-x_1}{x_3-x_1} )\in [\frac{_\pi}{^2},\frac{_{3\pi}}{^2}]\}$. This has the virtue that for $x\in\caA'$
$$
C|x-x_3|_\epsilon\geq |x_1-x_3|_\epsilon\mathbf{1}_{\{|x-x_1|_\epsilon\leq |x_1-x_3|_\epsilon\}}+|x-x_1|_\epsilon \mathbf{1}_{\{|x-x_1|_\epsilon> |x_1-x_3|_\epsilon\}}.
$$
Then  we  apply Lemma \ref{corner} item 4 since $\tilde\alpha=\beta_3<0$ to get
\begin{equation} \label{anbound1}
A_n \leq     C  e^{-n(\sum_i\alpha_i+\beta_3-Q)}|x_1-x_2|_\epsilon^{\frac{(\beta_1+\beta_2 -Q)^2}{2}}|x_1-x_3|_\epsilon^{(\beta_1+\beta_2-Q)\beta_3}  .
\end{equation}
The claim follows by using  $ |x_1-x_3|\leq|x_2-x_3|$ and $(\beta_2-Q)\beta_3\geq 0$.
\medskip
 

\medskip
 
\noindent {\bf 3.} $\sum_i\alpha_i>Q$, $\beta_1+\beta_2< Q$, $\beta_2>0$ and  $\beta_1+\beta_2+\beta_3\geq Q$.
Now we take $\caA=\caA(x_3,|x_3-x_1|)$. On $\caA$ the integrand in \eqref{keya} is bounded by
$|x-x_3|^{-\gamma(\beta_1+\beta_2+\beta_3)}$ and we arrive at
\begin{equation*}
A_n \leq     Ce^{-n(\sum_i\alpha_i-Q)}|x_1-x_3|_\epsilon^{\frac{(\beta_1+\beta_2+\beta_3-Q)^2}{2}}   .
\end{equation*}

\medskip
 
\noindent {\bf 4.} $\sum_i\alpha_i>Q$, $\beta_1+\beta_2< Q$, $\beta_2<0$ and  $\beta_1+\beta_2+\beta_3\geq Q$. By taking  $\caA$ an appropriate half of the annulus in  {\bf 3.} we can guarantee that $C|x-x_2|>|z-x_3|$. Then we may repeat the argument in  {\bf 3.}

\medskip
 
\noindent {\bf 5.}  $\beta_1+\beta_2< Q$, $\beta_1+\beta_2+\beta_3< Q$. Take $\caA$ a fixed unit ball with distance at least $1$ to $\{x_1,x_2,x_3\}$. Then $I_\caA\leq C$ and we get from \eqref{key1} the desired bound.

\medskip 
We have listed all the possibilities so that the proof of Proposition \ref{upperbound3} is complete.\qed


\subsection{Proof of Propositions \ref{holdercorrel} and \ref{holdercorrelbis}}

Without loss of generality we can supoose that $x=0$ and $|z_i|\geq 2$.   From \eqref{Z1reg} we obtain
\begin{align}\label{cafe}
\caF_\epsilon(0,z)=Ce^{-\frac{\gamma^2}{2}G_\epsilon(z)}e^{-\frac{\gamma}{2}\sum_k\alpha_kG_\epsilon(z-z_k)} h_\epsilon(z)
\end{align}
with
\begin{align*}
h_\epsilon(z)=\E\big[\big(\int e^{-\frac{\gamma^2}{2}G_\epsilon(u-z)+\gamma^2G_\epsilon(u)}r_\epsilon(u)M_\epsilon(d^2u)
\big)^{-s}\big]
\end{align*}
Here $s=\sum\alpha_i+\hf\gamma-2Q$ and $r_\epsilon(z)\leq C$ on $|z|\leq 1$. Since 
$h_\epsilon(z)\leq C$ and  
$$C^{-1}|z|_\epsilon\leq e^{-G_\epsilon(z)}\leq C|z|_\epsilon
$$ we only need to consider the case $|z|\leq \epsilon$. Set $k_\epsilon(z,u)=\partial_ze^{-\frac{\gamma^2}{2}G_\epsilon(u-z)+\gamma^2G_\epsilon(u)}$. We compute
\begin{align}\nonumber
\partial_z h_\epsilon(z)&= -s\int_{|u|\leq 1}k_\epsilon(z,u)r_\epsilon(u)\E\big[\big(\int e^{-\frac{\gamma^2}{2}G_\epsilon(v-z)+\gamma^2G_\epsilon(v)+{\gamma^2}G_\epsilon(v-u)}r_\epsilon(v)M_\epsilon(d^2v)
\big)^{-s-1}\big] du\\
&-s\E \int_{|u|\geq 1} k_\epsilon(z,u) r_\epsilon(u)\big(\int e^{-\frac{\gamma^2}{2}G_\epsilon(v-z)+\gamma^2G_\epsilon(v)}r_\epsilon(v) M_\epsilon(d^2v)
\big)^{-s-1}M_\epsilon(d^2u)\label{split}
\end{align}
where in the first term we got rid of $M_\epsilon(d^2u)$ by Girsanov transform. Now use 
$$|\partial_zG_\epsilon(u-z)|\leq C|u-z|_\epsilon^{-1}
$$
 whereby 
$$
|k_\epsilon(z,u)|\leq Ce^{-\frac{\gamma^2}{2}G_\epsilon(u-z)+\gamma^2G_\epsilon(u)}
$$
for $|u|\geq 1$, $|z|\leq\epsilon$ so that the second term in \eqref{split} is bounded by
$$
\E \big[\Big(\int_{|u|\geq 1} e^{-\frac{\gamma^2}{2}G_\epsilon(v-z)+\gamma^2G_\epsilon(v)}r_\epsilon(v)M_\epsilon(d^2v)
\Big)^{-s}\big]\leq C.
$$
In the first term the expectation is bounded by $C$ and 
$$
|k_\epsilon(z,u)|\leq C|u|_\epsilon^{-\frac{\gamma^2}{2}-1}
$$
since  $|z|\leq\epsilon$. We obtain then
\begin{align*}
|\partial_z h_\epsilon(z)|\leq C(\epsilon^{1-\frac{\gamma^2}{2}}+1).
\end{align*}
The second term in  \eqref{cafe} satisfies
$$
|\partial_z e^{-\frac{\gamma}{2}\sum_i\alpha_iG_\epsilon(z-z_i)} |\leq C
$$
and the first term 
\begin{align*}
|\partial_ze^{-\frac{\gamma^2}{2}G_\epsilon(z)}|\leq C\epsilon^{\frac{\gamma^2}{2}-1}.
\end{align*}
Altogether we get for $|z|\leq\epsilon$
$$
|\caF_\epsilon(0,z)-\caF_\epsilon(0,0)|\leq C\epsilon^{\frac{\gamma^2}{2}-1}|z|.
$$
The proof of Proposition \ref{holdercorrelbis} is similar.
\qed

\subsection{Proof of Proposition \ref{lpinfty}}

From \eqref{Z1} we get 
\begin{align*} 
  \langle V_{\gamma,\epsilon  }(y) \prod_k  V_{\alpha_k,\epsilon  }(z_k)\rangle_{\epsilon}  \leq  &C(\delta)|y|^{-\gamma\sum\alpha_i}f(y)
  \end{align*}
where
\begin{equation*}
f(y)=\E\big[ \int_{A_{\delta} } e^{\gamma X(x)- \frac{\gamma^2}{2} \E[X(x)^2] }  \frac{1}{|x-y|^{\gamma^2 }}  \prod_l \frac{1}{|x-z_l|^{\gamma \alpha_l}}   \hat{g}(x)^{1 - \frac{\gamma}{4} \sum_l \alpha_l } d^2x\big]^{-\frac{\sum\alpha_i+\gamma-2Q}{\gamma}} 
\end{equation*}
where $A_\delta$ is the annulus around origin with radi $2/\delta$ and $3/\delta$.  We get 
$$
\lim_{y\to\infty}|y|^{-\gamma(\sum\alpha_i+\gamma-2Q)}f(y)<\infty
$$
  so that 
  \begin{align*} 
  \langle V_{\gamma,\epsilon  }(y) \prod_k  V_{\alpha_k,\epsilon  }(z_k)\rangle_{\epsilon}  \leq  &C(\delta)|y|^{\gamma^2-2Q\gamma}=C(\delta)|y|^{-4}.
  \end{align*}
  \qed

\subsection{Proof of Proposition \ref{2merginfinity}}

 Let us assume that all the insertion points $z_1,\dots,z_N$ are  distinct from $0$ (otherwise replace the M\"obius transform $z\mapsto 1/z$ in what follows by $z\mapsto 1/(z-a)$ for some $a$ distinct from all the $(z_i)_i$). 
 
The first step is to replace the regularized potential \eqref{Meps1} in the regularized correlation functions \eqref{Z1reg} by the limiting potential $M_\gamma$ of \eqref{law}. This can be done with the help  of Kahane's convexity inequalities (see \cite[Theorem 2.2]{RV}) as we have
$$\E X_\epsilon(u)X_\epsilon(v)\leq C+\E X(u)X(v)$$ for some global constant $C>0$ and all $u,v\in\C$. Thus we have
we have
\begin{align}\label{psiexp}
 \langle V_{\beta_1,\epsilon}(x_1)V_{\beta_2,\epsilon}(x_2)  \prod_{k=r+1}^N V_{\alpha_k,\epsilon  }(z_k)\rangle_{\epsilon}\leq &C  \langle V_{\beta_1,\epsilon}(x_1)V_{\beta_2,\epsilon}(x_2)  \prod_{k=r+1}^N  V_{\alpha_k,\epsilon  }(z_k)\rangle.
\end{align}
for some irrelevant constant $C$ (which may change along lines). 
 M\"obius invariance of the Liouville field $\phi$ (see \cite[section 3.2]{DKRV}) gives 
 \begin{align*}
 \langle V_{\beta_1,\epsilon}(x_1)V_{\beta_2,\epsilon}(x_2) \prod_{k=r+1}^N  V_{\alpha_k,\epsilon  }(z_k)\rangle=
 \langle V^\psi_{\beta_1,\epsilon}(x_1)V^\psi_{\beta_2,\epsilon}(x_2)  \prod_{k=r+1}^N  V^\psi_{\alpha_k,\epsilon  }(z_k)\rangle\end{align*}
 where 
 $$
 V^\psi_{\beta,\epsilon}(x)=(A\epsilon)^{\frac{\beta^2}{2}}e^{\beta ((X\circ\psi)_\epsilon(x)+\hf Q(\ln \hat g\circ\psi)_\epsilon(x)+Q(\ln|\psi'|)_\epsilon+c)}.
 $$

%

As in \eqref{funamental} we get
\begin{align}\label{funamental1}
\eqref{psiexp} \leq &C \Big(\prod_{k=1}^N e^{2\Delta_{\beta_k}(\ln|\psi'|)_\epsilon(x_k)}\Big) 
e^{\frac{1}{2}\sum_{k\not =k'}\beta_k\beta_{k'} G_\epsilon (x_k,x_{k'})}  \E \Big[\big(\int_{\C} e^{ \gamma \sum_k \beta_k   G_\epsilon^\psi(x,x_k)}  M_\gamma(d^2x)\big)^{-s}\Big]  
\end{align}
where $\{\beta_1,\dots,\beta_{N}\}:=\{\beta_1,\beta_2,\alpha_{r+1},\dots,\alpha_n\}$, $\{x_1,\dots,x_{N}\}=\{x,y,z_{r+1}\dots,z_n\}$ and
 \begin{align*}
  G_\epsilon^\psi(x,y)&=\int G(x,v)\rho_\epsilon(y-\frac{_1}{^v})\frac{d^2v}{|v|^4}.
 \end{align*}
Note that 
 \begin{align*}
 \lim_{\epsilon\to 0}  G_\epsilon^\psi(x,y)&=G({x},\frac{_1}{^y})\ \ \ x\neq \frac{_1}{^y}.
  \end{align*}
If $x,1/y$  belong to a ball $B(0,R)$ for some $R>0$ we have the estimate (for some $R$-dependent constant $C_R>0$)
\begin{equation}\label{estgreeninf2}
  G_\epsilon^\psi(x,y)\geq \ln\frac{1}{\big|x-\frac{1}{y}\big|+\epsilon C R^2}-C_R.
\end{equation}
Indeed, by applying Jensen inequality to $-\ln|\cdot|$ and the triangle inequality, we get
\begin{equation*}
 G_\epsilon^\psi(x,y)= -\int_\C  \ln | x-\frac{1}{y+\epsilon u}   | \:  \rho(|u|^2) d^2u \geq  - \ln  \left (  | x- \frac{1}{y}|+ \int_\C | \frac{1}{y}- \frac{1}{y+\epsilon u}   | \:  \rho(|u|^2) d^2u \right ).
\end{equation*}
Now, notice that $ \int | \frac{1}{y}- \frac{1}{y+\epsilon u}   | \:  \rho(|u|^2) du= \frac{\epsilon }{y} \int | \frac{u}{y+\epsilon u}   | \:  \rho(|u|^2) du \leq C \epsilon R^2  $.

Now we focus on the item 1 (items 2 and 3 are dealt with in the same way). In that case $\beta_1,\beta_2>0$. The expectation 
in \eqref{funamental1} is then bounded by
$$C_{\delta} \E \Big[\big(\int_{\mathcal{A}} |z-\frac{1}{x_1}|^{-\gamma\beta_1} |z-\frac{1}{x_2}|^{-\gamma\beta_2}M_\gamma(dx)  \big)^{-s}\Big]  $$
where the set $\mathcal{A}$ is given by $\mathcal{A}:=\{z\in\C;|\frac{1}{x_2}-\frac{1}{x_1}|_\epsilon\leq  |z-\frac{1}{x_1}|\leq 1\}$. We can then complete the proof as done in the proof of Proposition \ref{upperbound2}.\qed
%
%


\section{Appendix}

\subsection{Proof of Lemma \ref{lemmaPDEsbis}}
We start with an a'priori bound on the dimension of the set of solutions:
\begin{lemma} The space of real solutions to equation  \eqref{hypergeo} on $\mathbb{C} \setminus \lbrace 0,1 \rbrace$ is at most four dimensional.
\end{lemma}
\begin{proof}
First observe that applying $\partial^2_{\bar{z}}$ to \eqref{hypergeo}  we see that $\caT$ is a distributional solution of a linear PDE with analytic coefficients in $\mathbb{C} \setminus \lbrace 0,1 \rbrace$ whose highest degree symbol is $\Delta^2$ hence of an analytic hypoelliptic system. Therefore, $\caT$ is real analytic in $\mathbb{C} \setminus \lbrace 0,1 \rbrace$. Hence it suffices to bound the dimension of the solution set locally, in any open subset of $\mathbb{C} \setminus \lbrace 0,1 \rbrace$.

Thus let $\caT(x+iy):=u(x,y)$ be a real valued solution of  \eqref{hypergeo} in a neighborhood of $x_0\in\R\setminus\{0,1\}$. The real part  of equation \eqref{hypergeo} reads:
\begin{equation}\label{realpart}
u_{xx}-u_{yy}+L_1u_x+L_2u_y+L_3u=0
\end{equation} 
where $L_1,L_2,L_3$ are real analytic functions. This is a  hyperbolic equation and hence $u$ is determined in a neighborhood of $x_0$ by the data $u_0(x)= u(x,0)$ and $v_0(x)=u_y(x,0)$. The imaginary part of \eqref{hypergeo} restricted to $\R$ (recall that $a,b,c$ are real) yields
\begin{equation*}
-\frac{1}{2} v_0'(x)+  \frac{(c-x(a+b+1))}{x(1-x)} v_0(x)=0
\end{equation*}
Hence $v_0$ lives in a space of dimension $1$. We know want to show that given the data $v_0$, the function $u_0$ lives at most in a space of dimension 3. The imaginary part of \eqref{hypergeo} gives an equation
\begin{equation}\label{impart}
u_{xy}+M_1u_x+M_2u_y+M_3u=0
\end{equation}
where $M_1,M_2,M_3$  are real analytic. The $y$ derivative of this equation can be written as
\begin{equation*}
\partial_xu_{yy}+N_1u_{xy}+N_2u_{yy}+N_3u_x+N_4u_y+N_5u=0.
\end{equation*}
Now, plug into this last equation $u_{yy}$ solved from \eqref{realpart}  and $u_{xy}$ solved from \eqref{impart}. The resulting equation evaluated at  $y=0$ yields an equation of the form
\begin{equation*}
 u_0^{(3)}(x)+f_1(x) u_0^{(2)}(x)+ f_2(x) u_0'(x)+f_3(x) u_0(x)=f_4(x)
\end{equation*}
 where the functions $f_i(x)$ can depend on $v_0(x)$ and its derivatives. The solution space of this equation is three dimensional and hence,  the space of real local solutions  is four dimensional. Hence the set of global solutions is at most four dimensional.
  \end{proof}
 Recall next the definitions \eqref{Fpmdef}. We have $F_-=F_1$ and $F_+=z^{1-c}F_2$ where $F_1$ and $F_2$ Gauss hypergeometric functions. The latter  have an analytic continuation to $ \C \setminus (1,\infty)  $ and they take real values on the negative real axis. We define the function  $z^{1-c}$ in the cut plane $ \C \setminus   (-\infty,0] $. Then $F_\pm$ give  two linearly independent complex solutions to \eqref{hypergeo} on  $ \C \setminus \{  (-\infty,0] \cup [1,\infty)  \}$.  Hence  
 $$\{|F_{\pm}(z)|^2, \  \text{Re}(F_{-}(z)F_{+}(z)), \  \text {Im}(F_{-}(z)F_{+}(z))\}
 $$ 
 provide four linearly independent real solutions to   \eqref{hypergeo} on the set $ \C \setminus \{  (-\infty,0] \cup [1,\infty)  \}$. Our task is to inquire which of them extend to to the set $\mathbb{C} \setminus \lbrace 0,1 \rbrace$. For this we need to check continuity of the functions and their derivatives at $ (-\infty,0) \cup (1,\infty) $.
 
Consider first continuity across the negative real axis. Obviously $|F_{-}|^2=|F_1|^2$ and $ |F_{+}|^2=|z|^{2-2c}|F_2|^2$ are continuos together with their derivatives. Next, we have
$$
F_{-}(z)F_{+}(z)=z^{1-c}F_1(z)F_2(z).
$$
Recall $z^{1-c}$ was defined with a cut on $\R_-$ i.e. for $z=e^{i\varphi}r$ with $\varphi\in (-\pi,\pi)$ $z^{1-c}=e^{(1-c)i\varphi}r^{1-c}$. Since $F_i$ are real on the negative real axis we conclude that $\text{Re}(F_{-}(z)F_{+}(z))$ is continuous but $ \text {Im}(F_{-}(z)F_{+}(z))$ is not.  Next, consider the derivative $H(z):=\partial_y(F_{-}(z)F_{+}(z))$. Since $F_i(0)=1$ we obtain for small $z$
 $$
 H(z)=(\partial_y (F_1(z)F_2(z))z^{1-c}+i(1-c)F_1(z)F_2(z)z^{-c}=i(1-c)z^{-c}(1+\caO(|z|))
 $$
 Hence $\text{Re}\, H$ is not continuous on $\R_-$ near origin.  Thus only the solutions $|F_{\pm}(z)|^2$ extend to $\C\setminus  [1,\infty)$. Next we need to inquire about their continuity across $[1,\infty) $.

To do this we need to  introduce another pair of hypergeometric functions:
 \begin{equation*}
G_1(z):={_2}F_1(a,b,1+a+b-c,1-z), \;  G_2(z):={_2}F_1(c-a,c-b,1+c-a-b,1-z) .
 \end{equation*}
Then $G_1$ and  $(1-z)^{c-a-b}G_2$ are linearly independent solutions of eq. \eqref{hypergeo} on $\C\setminus  (-\infty,1]$   and we have the following relation 
 for $z \in \C \setminus \{(-\infty,0] \cup [1,\infty) \}$ (see e.g. \cite{Rib}):
\begin{equation*}
F_{-}(z)= \frac{\Gamma (c) \Gamma (c- a- b) }{\Gamma (c- a) \Gamma (c- b)} G_1(z)   + \frac{\Gamma (c) \Gamma ( a+ b- c)  }{\Gamma ( a) \Gamma ( b)}  (1-z)^{c-a-b}G_2(z)
\end{equation*}
and 
\begin{equation*}
F_{+}(z)= \frac{\Gamma (2- c) \Gamma ( c- a- b) }{\Gamma (1- a) \Gamma (1- b)}  G_1(z)  + \frac{\Gamma (2-c) \Gamma ( a+ b- c)  }{\Gamma ( a- c+1) \Gamma ( b- c+1)} (1-z)^{c-a-b}G_2(z) .
\end{equation*}
Consider the linear combination $K(z):=\lambda_1|F_{-}(z)|^2+\lambda_2|F_{+}(z)|^2$. This becomes in the other basis
\begin{equation*}
K(z)=A |G_1(z)|^2 + B|1-z|^{2(c-a-b)} |G_2(z)|^2 + D\,  \text{Re}  (     (1-z)^{c-a-b} {G}_1(z){G}_2 (z)   )  
\end{equation*}   
with  $D$  given by
\begin{equation*}
D=2 \Gamma (c-a-b) \Gamma (a+b-c) \left ( \lambda_1   \frac{\Gamma(c)^2}{   \Gamma(c-a)  \Gamma(c-b)  \Gamma (a)  \Gamma(b) }+ \lambda_2   \frac{\Gamma(2-c)^2}{   \Gamma(1-a)  \Gamma(1-b)  \Gamma(a-c+1)  \Gamma(b-c+1) }  \right ).
\end{equation*}
As in the case of $H(z)$ above, studying $\partial_yK(z)$ for $z$ close to $1$ we conclude that if $D\neq 0$ then $\partial_yK(z)$ is not continuous across $(1,\infty)$ near $1$. For $D=0$ $K$ and its derivatives are continuous across $(1,\infty)$. Hence the relation \eqref{Fundrelation} follows. 
\qed

\subsection{Proof of \eqref{selberg} }
We start from the formula 
\begin{equation*}
\int_{\R^2}  |z|^{2(\alpha-1)} |z-1|^{2(\beta-1)}      dz  =  \left ( \frac{\Gamma(\alpha) \Gamma(\beta) }{\Gamma (\alpha+\beta)}\right )^2  \frac{\sin (\alpha \pi) \sin (\beta \pi)}{\sin( (\alpha+\beta) \pi )}
\end{equation*}   
(see p. 504 in \cite{selberg}) which holds for  $\alpha, \beta>0$ and $\alpha+\beta<1$.
Using $\Gamma (z)\Gamma (1-z)=\frac{\pi}{\sin \pi z} $ this becomes  
\begin{equation}\label{formuleint}
\int_{\R^2}  |z|^{2(\alpha-1)} |z-1|^{2(\beta-1)}      dz  =  \pi \frac{l(\alpha) l (\beta) }{l(\alpha+\beta)}= \pi \frac{1}{l(1-\alpha) l(1-\beta)  l(\alpha+\beta) }.
\end{equation}   
This formula can be analytically continued to get for $\alpha,\beta>0$ and $1<\alpha+\beta<3/2$
\begin{equation}\label{formuleint2}
\int_{\R^2}  |z|^{2(\alpha-1)} \big(|z-1|^{2(\beta-1)}-|z|^{2(\beta-1)}   \big)   dz  = \pi \frac{1}{l(1-\alpha) l(1-\beta)  l(\alpha+\beta) }.
\end{equation} 
Analytic continuation is a consequence of the following observation.  Consider the function
$$F(\alpha,\beta):=\int_{\R^2} |z|^{2(\alpha-1)} \big(|z-1|^{2(\beta-1)}-\mathbf{1}_{\{|z|\geq 1\}}|z|^{2(\beta-1)}   \big)   dz -\frac{\pi}{\alpha+\beta-1} .$$
A simple check shows that it is analytic for $\alpha,\beta>0$ and $\alpha+\beta \in ]0, 3/2[ \setminus \lbrace 1 \rbrace$. Furthermore it coincides with the integral in \eqref{formuleint} for  $\alpha, \beta>0$ and $\alpha+\beta<1$ and with the integral  \eqref{formuleint2} for  $\alpha, \beta>0$ and $1<\alpha+\beta<3/2$.

\subsection{A special function entering the DOZZ formula} \label{expressionUpsilon}
Here, for the sake of completeness, we recall the definition of $\Upsilon_{\frac{\gamma}{2}}$ which enters in an essential way the DOZZ formula \eqref{DOZZformula}. The function $\Upsilon_{\frac{\gamma}{2}}$ is defined for $0<\Re (z)<Q$ (recall that $Q=\frac{2}{\gamma}+\frac{\gamma}{2}$) by the formula
\begin{equation*}
\ln \Upsilon_{\frac{\gamma}{2}} (z)  = \int_{0}^\infty  \left ( \Big (\frac{Q}{2}-z \Big )^2  e^{-t}-  \frac{( \sinh( (\frac{Q}{2}-z )\frac{t}{2}  )   )^2}{\sinh (\frac{t \gamma}{4}) \sinh( \frac{t}{\gamma} )}    \right ) \frac{dt}{t}
\end{equation*}
The function $\Upsilon_{\frac{\gamma}{2}}$ can be analytically continued to $\C$ and the zeros are simple and given by the discrete set $(-\frac{\gamma}{2} \N-\frac{2}{\gamma} \N) \cup (Q+\frac{\gamma}{2} \N+\frac{2}{\gamma} \N )$: for more on the function $\Upsilon_{\frac{\gamma}{2}}$ and its properties, see the reviews \cite{nakayama,Rib} for instance.

 \end{document}